\documentclass[11pt]{amsart}
\usepackage{amsthm, amsmath}
\usepackage{amssymb}
\usepackage[all]{xy}

\theoremstyle{plain}
\newtheorem{theorem}{Theorem}[section]
\newtheorem{lemma}[theorem]{Lemma}
\newtheorem{cor}[theorem]{Corollary}
\newtheorem{prop}[theorem]{Proposition}

\theoremstyle{remark}

\theoremstyle{definition}

\newtheorem{defn}[theorem]{Definition}
\newtheorem{Remark}[theorem]{Remark}
\newtheorem*{remark}{Remark} 

\numberwithin{equation}{subsection}
\numberwithin{theorem}{subsection}

\newcommand{\CC}{\mathbb C}

\newcommand{\A}{\mathbb A}
\newcommand{\isom}{\stackrel{\backsim}{\longrightarrow}}

\newcommand\un{\underline}

\newcommand\hfld[2]{\smash{\mathop{\hbox to 10mm{\rightarrowfill}}
     \limits^{\scriptstyle#1}_{\scriptstyle#2}}}
\newcommand\hflg[2]{\smash{\mathop{\hbox to 10mm{\leftarrowfill}}
     \limits^{\scriptstyle#1}_{\scriptstyle#2}}}
\newcommand\vfld[2]{\llap{$\scriptstyle#1$}
     \left\downarrow\vbox to 6mm{}\right.\rlap{$\scriptstyle #2$}}

\title{The base change fundamental lemma for central \\ elements in parahoric Hecke algebras}
\author{Thomas J. Haines}

\date{}

\begin{document}

\thanks{Research partially supported by NSF grants DMS-0303605 and FRG-0554254, and by a Sloan 
Research Fellowship.}


\begin{abstract} 
Let $G$ be an unramified group over a $p$-adic field $F$, and let $E/F$ be a finite unramified extension field.  Let $\theta$ denote a generator of ${\rm Gal}(E/F)$.  
This paper concerns the matching, at all semi-simple elements, of orbital integrals on $G(F)$ with $\theta$-twisted orbital integrals on $G(E)$.  More precisely, suppose $\phi$ belongs to the center of a parahoric Hecke algebra for $G(E)$.  This paper introduces a {\em base change homomorphism} $\phi \mapsto b\phi$ taking values in the center of the corresponding parahoric Hecke algebra for $G(F)$.  It proves that the functions $\phi$ and $b\phi$ are {\em associated}, in the sense that the stable orbital integrals (for semi-simple elements) of $b\phi$ can be expressed in terms of the stable twisted orbital integrals of $\phi$.  In the special case of spherical Hecke algebras (which are commutative) this result becomes precisely the base change fundamental lemma proved previously by Clozel \cite{Cl90} and Labesse \cite{Lab90}.  
As has been explained in \cite{H05}, the fundamental lemma proved in this paper is a key ingredient for the study of Shimura varieties with parahoric level structure at the prime $p$.  \end{abstract}

\maketitle



\markboth{T. Haines}
{Base change fundamental lemma for parahoric Hecke algebras}

\section{Introduction}

Let $F$ be a $p$-adic field, and $E/F$ an unramified extension of
degree $r$. Let $\theta$ denote a generator for ${\rm Gal}(E/F)$.
Let $G$ be an unramified connected reductive group over $F$. Let
${\rm Res}_{E/F}G_E$ denote the Weil restriction of scalars of 
$G_E=G\otimes_F E$ to a group over $F$. The
automorphism $\theta$ of $E$ determines an $F$-automorphism of
${\rm Res}_{E/F}G_E$ as well as an automorphism of its $F$-points
$G(E)$, which will also be denoted by the symbol $\theta$.  

Let ${\mathcal H}(G)$ denote the convolution algebra of locally
constant and compactly supported $\CC$-valued functions on $G(F)$,
convolution being defined using some choice of Haar measure,
specified later. For a compact open subgroup ${\mathcal P}
\subset G(F)$, we denote by ${\mathcal H}_{\mathcal P}(G)$ the
subalgebra of ${\mathcal H}(G)$ consisting of ${\mathcal
P}$-bi-invariant functions.

To simplify things, for the remainder of this introduction we assume that $G_{\rm der}$
is simply connected.  In this case stable conjugacy classes in $G(F)$
are intersections of $G(\overline{F})$-conjugacy classes with $G(F)$.  
Consider the {\em concrete norm} $N:G(E) \rightarrow
G(E)$ given by $N\delta = \delta \theta(\delta) \cdots
\theta^{r-1}(\delta)$. It is known (\cite{Ko82}) that $N\delta$ is
stably conjugate to an element ${\mathcal N}\delta \in G(F)$, and
that this determines a well-defined {\em norm map}
$$
{\mathcal N} : \lbrace \mbox{stable} \,\, \theta\mbox{-conjugacy
classes in} \,\, G(E) \rbrace \rightarrow \lbrace \mbox{stable
conjugacy classes in} \,\, G(F) \rbrace.
$$
(Of course the norm map still exists when $G_{\rm der} \neq G_{\rm sc}$; see \cite{Ko82}.)  
We shall say $\gamma \in G(F)$ {\em is a norm} if $\gamma$ is
stably conjugate to ${\mathcal N}\delta$, for some $\delta \in
G(E)$.  

Fix a semi-simple element $\gamma \in G(F)$ and let $G_\gamma
\subset G$ denote its centralizer.  For $f \in {\mathcal H}(G)$
we can then define the {\em orbital integral}
$$
{\rm O}^G_\gamma(f) = \int_{G_\gamma(F) \backslash G(F)} f(g^{-1} \gamma g)
{dg \over dt},
$$
depending on the choice of Haar measures $dg$ and $dt$ on $G(F)$
and $G_\gamma(F)$. We may also consider the {\em stable orbital
integral}
$$
{\rm SO}^G_\gamma(f) = \sum_{\gamma'} e(G_{\gamma'}){\rm O}^G_{\gamma'}(f).
$$
Here $\gamma'$ ranges over the conjugacy classes in $G(F)$ which
are stably conjugate to $\gamma$, and $e(H) \in \lbrace 1,-1
\rbrace$ is the sign attached by Kottwitz \cite{Ko83} to a connected reductive group
$H$. Note that our assumption that $G_{\rm der}$ is simply connected
implies that the centralizers $G_{\gamma'}$ are connected.
If $\gamma$ and $\gamma'$ are stably conjugate, their centralizers
are inner forms of each other and therefore it makes sense to 
require the Haar measures on these groups to be compatible with each other, 
see \cite{Ko88}, p.~631.

Similarly, for an element $\delta \in G(E)$ such that $\mathcal N\delta$ is
semi-simple, we define its $\theta$-centralizer to be the 
connected reductive group $G_{\delta\theta}$ over $F$ such that 
$$
G_{\delta \theta}(F)= \lbrace x \in G(E) \,\,
| \,\, x^{-1} \delta \theta(x) = \delta \rbrace.
$$  
Then for $\phi
\in {\mathcal H}(G(E))$ we define the {\em twisted orbital
integral}
$$
{\rm TO}^{G(E)}_{\delta \theta}(\phi) = \int_{G_{\delta \theta}(F) \backslash
G(E)} \phi(h^{-1} \delta \theta(h)) \frac{dh}{dt}
$$
and its stable version
$$
{\rm SO}^{G(E)}_{\delta \theta}(\phi) = \sum_{\delta'} e(G_{\delta'
\theta}){\rm TO}^{G(E)}_{\delta' \theta}(\phi).
$$
Here $\delta'$ ranges over the $\theta$-conjugacy classes in $G(E)$
whose norm down to $G(F)$ is in the same stable conjugacy class as that 
of $\delta$. If $\gamma\in G(F)$ lies in the stable conjugacy class
of $N\delta$, $G_{\delta\theta}$ is an inner form of $G_\gamma$, and therefore 
it makes sense to require the Haar measures on these groups 
to be compatible (see loc.~cit.~). 

\begin{defn} \label{associated_defn}
The functions $f \in {\mathcal H}(G)$ and $\phi \in {\mathcal
H}(G(E))$ are {\em associated} if the following condition holds:
for every semi-simple $\gamma \in G(F)$, the stable orbital
integral ${\rm SO}^G_\gamma(f)$ vanishes if $\gamma$ is not a norm, and
if there exists $\delta \in  G(E)$ such that ${\mathcal N}\delta =
\gamma$, then
$$
{\rm SO}^G_\gamma(f) = {\rm SO}^{G(E)}_{\delta \theta}(\phi).
$$
\end{defn}

Denote by $K$ a hyperspecial maximal compact subgroup of $G(F)$ and let
${\mathcal H}_K(G)$ denote the corresponding spherical Hecke
algebra.  Of course $K$ also gives rise to $K(E) \subset G(E)$,
and a corresponding spherical Hecke algebra ${\mathcal H}_K(G(E))$.

The {\em base change homomorphism} for spherical Hecke algebras 
is a homomorphism of $\mathbb C$-algebras
$$
b: {\mathcal H}_K(G(E)) \rightarrow {\mathcal H}_K(G).
$$
It is characterized by the following property.  Let $W_F$ denote the Weil group of $F$.  For an unramified
admissible homomorphism $\psi: W_F \, \rightarrow \, ^LG$, let
$\psi':W_E \, \rightarrow \, ^LG$ denote its restriction to the
subgroup $W_E $ of $W_F$. Let $\pi_\psi$ and $\pi_{\psi'}$ denote
the corresponding representations of $G(F)$ and $G(E)$.  Then for
any $\phi \in {\mathcal H}_K(G(E))$ we have
$$
\langle {\rm trace}\, \pi_{\psi'}, \phi \rangle  = \langle {\rm trace} \,\pi_\psi, b\phi \rangle.
$$

The fundamental lemma for stable base change normally refers to the
following statement pertaining to spherical Hecke algebras, proved by
Clozel \cite{Cl90}, and also by Labesse \cite{Lab90}.

\begin{theorem} \label{sph_fl} \mbox{\rm ({Clozel, Labesse})} 
If $\phi \in {\mathcal H}_K(G(E))$, then $b\phi$ and $\phi$ are
associated.
\end{theorem}

The proof relies on a global argument using the simple trace
formula of Deligne-Kazhdan \cite{DKV} as well as Kottwitz's
stabilization of the elliptic regular part of the twisted trace
formula.  Another essential ingredient is the special case of the
theorem, wherein $\phi$ is the {\em unit element} of the spherical
Hecke algebra, also proved by Kottwitz \cite{Ko86b}.

In this article we prove an analogue of Theorem \ref{sph_fl} for central elements in
parahoric Hecke algebras.
We fix an Iwahori subgroup $I \subset G(F)$ which is contained in
$K$.  Also, fix a parahoric subgroup $J$ containing $I$.  The subgroups $I,J,K$ of $G(F)$ 
give rise to corresponding subgroups of $G(E)$, which we denote by the same symbols.
Let ${\mathcal H}_J(G)$ denote the
corresponding parahoric Hecke algebra. In general this 
algebra is non-commutative.  Denote its center by $Z({\mathcal
H}_J(G))$.

We can define a {\em base-change homomorphism} 
$$
b: Z(\mathcal H_J(G(E))) \rightarrow Z(\mathcal H_J(G))
$$
which is characterized in much the same way as in the spherical case (see section \ref{base_change_hom_section}).  Provided that $J \subseteq K$, it is compatible with the spherical case in the following sense.  Let $\mathbb I_K$ denote the characteristic function of $K$.  By virtue of the Bernstein isomorphism and its compatibility with the Satake isomorphism, we have an isomorphism of algebras
$$
-*_J \mathbb I_K : Z(\mathcal H_J(G)) ~ \widetilde{\rightarrow} ~ \mathcal H_K(G)
$$
(as well as the obvious analogue of this for $G(E)$ replacing $G(F)$).  Then the aforementioned compatibility is the commutativity of the following diagram 

$$\renewcommand\arraystretch{1.7}
\begin{array} {ccc}
Z({\mathcal H}_J(G(E))) & \hfld{- \, *_{J(E)}{\mathbb I}_{K(E)}}{\sim} &
{\mathcal H}_K(G(E)) \\
\vfld{b}{} && \vfld{}{b}\\
Z({\mathcal H}_J(G)) & \hfld{- \, *_J {\mathbb I}_K}{\sim} & {\mathcal
H}_K(G).
\end{array}
$$

The main theorem of this paper is the following result.

\begin{theorem} \label{fl}
If $\phi \in Z({\mathcal H}_J(G(E)))$, then $b\phi$ and $\phi$ 
are associated.
\end{theorem}

As in the spherical case, the proof breaks naturally into two parts.  The theorem is proved by induction on the semi-simple rank of $G$.  The first part is to use descent formulas (see section \ref{descent}) and the induction hypothesis to reduce the problem to elliptic semi-simple elements $\gamma \in G(F)$.  Various other reductions (see section 
\ref{reductions_section}) allow us to assume $G$ is adjoint and $\gamma$ is strongly regular and elliptic.  The second part is to prove the theorem in the strongly regular elliptic case using a global argument (see sections \ref{global_trace_formula_section} and \ref{putting_it_all_together_section}). 

\medskip

Our initial approach to this problem, in the special case where $J= I$ and $G$ is $F$-split, followed closely the strategy of Labesse \cite{Lab90}.  This case of the theorem was proved in collaboration with Ng\^{o} Bao Ch\^{a}u.  Somewhat surprisingly, the fact that $\phi$ belongs to the center of the Iwahori-Hecke algebra permitted us to bypass certain technical difficulties that arise for the case of spherical functions, allowing for an even easier proof.  However, in attempting to generalize Labesse's arguments to general parahoric subgroups $J$, we ran into serious difficulties.  In fact it seems that Labesse's elementary functions do not in principle carry enough information to handle the general parahoric case.  Because of this we eventually settled on an approach much closer to that of Clozel \cite{Cl90}.  However, as suggested by a referee's remark recorded in the introduction of \cite{Cl90}, one can further streamline Clozel's original argument by replacing elliptic traces with compact traces.  We do so in this article, and in the process were heavily influenced by a paper of Hales \cite{Ha95} which carried out that idea in a different situation.

\medskip

The motivation for this article comes from our program to compute the local factors of the Hasse-Weil zeta functions of some Shimura varieties with parahoric level structure at a prime $p$.  This has already been carried out in a special case in \cite{H05}, to which we refer for more details.  Here, let us only indicate very briefly how Theorem \ref{fl} is used. 

Suppose $Sh$ is a PEL Shimura variety attached to Shimura data $({\bf G}, X, {\bf K})$, where ${\bf K} = K^pK_p$ and $K_p$ is a parahoric subgroup of ${\bf G}(\mathbb Q_p) = G(\mathbb Q_p)$.  We suppose $Sh$ is defined by a moduli problem over the local ring $\mathcal O_{{\bf E}_{\mathfrak p}}$, where $\mathfrak p$ is a prime ideal in the reflex field ${\bf E}$ dividing $p$.  To compute the semi-simple local Hasse-Weil zeta function at $\mathfrak p$ following the method of Kottwitz \cite{Ko92}, one needs to express the semi-simple Lefschetz number in the form
\begin{equation} \label{Lef}
\sum_{x \in Sh(k_r)} {\rm Tr}^{\rm ss}(\Phi_{\mathfrak p}^r, R\Psi_x(\overline{\mathbb Q}_\ell)) = \sum_{\gamma_0} \sum_{(\gamma, \delta)} c(\gamma_0; \gamma, \delta) \, {\rm O}_{\gamma}(f^p) \, {\rm TO}_{\delta \sigma}(\phi_r),
\end{equation}
for every finite extension $k_r$ of the residue field of ${\bf E}_{\mathfrak p}$ (notation as in \cite{H05}).  Now suppose $Sh$ is attached to a Siegel moduli problem (type C) or a ``fake unitary'' 
Shimura variety (type A), and the group $G = {\bf G}_{\mathbb Q_p}$ is split of type C resp. A.  The theory of Rapoport-Zink local models \cite{RZ} shows that $R\Psi$ is computed from a parahoric-equivariant perverse sheaf $R\Psi^{{\bf M}^{{\rm loc}}}$ on an appropriate parahoric flag variety.  The main theorem of \cite{HN02} shows that  $R\Psi^{{\bf M}^{{\rm loc}}}$ is ``central'' with respect to convolution of perverse sheaves.  From this, one can show (\cite{H05}) that the test function $\phi_r$ making (\ref{Lef}) hold belongs to the {\em center} of a parahoric Hecke algebra for $G$ (over an unramified extension of $\mathbb Q_p$).  

Next, a ``pseudo-stabilization'' is performed on the right hand side of (\ref{Lef}), and then via the Arthur-Selberg trace formula it is expressed (albeit with some strong assumptions here to avoid mention of endoscopy) in the form
\begin{equation} \label{pseudo_stab}
\sum_{(\gamma_0,\gamma, \delta)} c(\gamma_0; \gamma, \delta) \, {\rm O}_{\gamma}(f^p) \, {\rm TO}_{\delta \sigma}(\phi_r) = \sum_{\pi} m(\pi) ~ {\rm Tr} \, \pi(f^p \, f^{(r)}_p \, f_\infty).
\end{equation}
Here, $\pi$ runs over certain automorphic representations of ${\bf G}(\A_\mathbb Q)$.  The main 
problem is to find and describe a function $f^{(r)}_p$ on $G(\mathbb Q_p)$ which makes this identity hold.  Theorem \ref{fl} implies the following: {\em if we put $f^{(r)}_p = b\phi_r$, then (\ref{pseudo_stab}) holds}.  

The end result of all this is an expression of the semi-simple local $L$-factor at $\mathfrak p$ in terms of semi-simple automorphic $L$-functions (see \cite{H05}).  The generality of Theorem \ref{fl} (and the fact that the geometric techniques of \cite{HN02} will apply to certain forms of type A or C) means that we expect these methods to apply to certain unitary groups which are non-split at $p$.  For example, we expect to be able to handle those ``fake unitary'' cases where ${\bf G}_{\mathbb Q_p}$ is a quasi-split unitary group attached to a quadratic extension of $\mathbb Q_p$ in which $p$ remains prime, as well as some natural generalizations of this situation.

\medskip

Let us now summarize the contents of the paper.  Section \ref{preliminaries_section} gives notation and basic definitions.  Section \ref{base_change_hom_section} describes the Bernstein isomorphism, defines the base change homomorphism for centers of parahoric Hecke algebras, and gives some necessary properties thereof.  Section \ref{descent} establishes the descent formulas and proves the lemmas needed to compare them effectively.  Section \ref{reductions_section} contains the reduction steps mentioned above.  Sections \ref{compact_trace_identity_section} and \ref{temperedness_section} provide technical tools needed in the global argument explained in sections \ref{global_trace_formula_section} and \ref{putting_it_all_together_section}.

\medskip

We point out that the proof simplifies significantly in two special cases: (1) $J$ is a hyperspecial maximal parahoric subgroup $K$; and (2) $G$ is split over $F$ and $J$ is an Iwahori subgroup $I$.  The most technical aspects of the paper arise in connection with the subset $_EW(P,J)$ of the relative Weyl group $_EW$ over $E$, which parametrizes the set $P(E) \backslash G(E)/J(E)$, where $P(E)$ is a standard parabolic subgroup (see section \ref{descent}).  But the set $_EW(P,K)$ is a singleton, and for split groups the set $_EW(P,I)$ can be taken to be the Kostant representatives $W^P$ in the (absolute) Weyl group $W$.  Thus in both cases the equality $\theta(w) = w$ holds for all $w \in \, _EW(P,J)$, and we may ignore the somewhat technical Lemma \ref{lemma_C}.  Also, that equality makes the arguments in subsection \ref{simplification_subsection} much more transparent.

In case (1) our proof essentially reduces to Clozel's argument in \cite{Cl90}, with a few notable differences stemming from our use of compact traces.  The use of compact traces appears to force some minor changes to Clozel's global argument (see sections \ref{global_trace_formula_section} and \ref{putting_it_all_together_section}).  For example, in 8.1 (b),  the ``stabilizing functions'' at the place $v_1$ (which is inert here) are different from those used in \cite{Cl90} 
(where $v_1$ is split).  Why the difference?  Using compact traces seems to require one to work with adjoint groups (see e.g.~subsection \ref{deformation_subsection}).  Lemma \ref{A_injectivity}, which is used to justify the ``stabilizing property'' of the functions $\phi_{v_1}$ and $f_{v_1}$ of 8.1 (b), differs from Clozel's analogous assertion (in the proof of Lemma 6.5 of \cite{Cl90}) and comes in because it is adapted to adjoint groups.  (Clozel works instead with certain non-semi-simple groups with simply connected derived group.)  But the proof of Lemma \ref{A_injectivity} seems to require that the place $v_1$ be inert rather than split.  Generalized Kottwitz functions have all the necessary properties at inert places (whereas Clozel's original stabilizing function $\phi_{v_1}$ does not even make sense unless the place $v_1$ is split).  That is why we used generalized Kottwitz functions at $v_1$.

\medskip

Finally, we explain the relation of Theorem \ref{fl} to the recent progress on the fundamental lemma due to Ng\^{o}, Laumon, Waldspurger, and others.  Recently Ng\^{o} has proved the (standard endoscopy) fundamental lemma for Lie algebras \cite{Ngo}.  Combined with work of Hales \cite{Ha95} and Waldspurger (\cite{W97}, \cite{W07}, \cite{W08}), Ng\^{o}'s result implies the transfer theorem for standard endoscopy, and presumably also for twisted endoscopy.  Stable base change is related to a very special case of twisted endoscopy.  So in the special case related to base change, these theorems yield the following statement: suppose $\phi$ is any element in $C^\infty_c(G(E))$ (for example we could take $\phi \in Z(\mathcal H_J(G(E)))$).  Then there exists a function $f \in C^\infty_c(G(F))$ such that $(\phi,f)$ are associated in the sense of Definition \ref{associated_defn}.  However $f$ is not uniquely determined as a function (one can change it by any function all of whose stable orbital integrals vanish) and no explicit characterization of $f$ is given in general.  Theorem \ref{fl} proves that when $\phi \in Z(\mathcal H_J(G(E)))$, then $f$ can be taken to be in $Z(\mathcal H_J(G(F)))$ and in fact such an $f$ is determined from $\phi$ using a simple and explicit rule.

\medskip

\noindent {\em Acknowledgments:}  The main results of this paper were discovered in collaboration with Ng\^{o} Bao Ch\^{a}u.  In 2000-2001, Ng\^{o} and I worked out a proof of the main theorem in a special case 
(mentioned above).  Ng\^{o}'s influence remained an important factor in my subsequent attempts to prove the main theorem in its current generality, during which time the original plan of attack had to be significantly altered.  However, Ng\^{o} has declined to be named as a coauthor for the end result.  I wish to express my debt to Ng\^{o} for his insights and for the impetus he gave to this project in its early stages.

During the long gestation period for this paper, I also benefited from conversations and/or correspondence with several other people, including the following: J. Arthur, J.-F. Dat, T. Hales, R. Kottwitz, J.-P. Labesse, M. Rapoport, M. Reeder, A. Roche, and J.-K. Yu.  It is a pleasure to thank all these people.  In addition, I 
am especially grateful to R. Kottwitz and M. Rapoport for their continuing interest in this project, and to Kottwitz for some helpful remarks on the manuscript, and for pointing out an inaccuracy in section 5.

I am indebted to Christian Kaiser, who discovered an error in an early version of this article.  I heartily thank Ulrich G\"{o}rtz for his communication to me of Kaiser's remarks and for various helpful conversations over the years.

I thank the referees for their attention to detail and their numerous and very helpful remarks.


\section{Parahoric subgroups and other preliminaries} \label{preliminaries_section}

\subsection{Basic notation}
Let $F$ denote a p-adic field.  Let $\mathcal O_F$ denote the ring of integers in $F$, and $\varpi \in \mathcal O_F$ a uniformizer.  Let $q = p^n$ denote the cardinality of the residue field of $F$.  Fix an algebraic closure $\overline{F}$ for $F$, and let $L$ denote the completion of the maximal unramified extension of $F$ inside $\overline{F}$.  Let $\sigma \in {\rm Aut}(L/F)$ denote the Frobenius automorphism of $L$ over $F$.  Let $\mathcal O_L$ denote the ring of integers in $L$.  The valuation ${\rm val}_F : F^\times \rightarrow \mathbb Z$ will be normalized such that ${\rm val}_F(\varpi) = 1$.  Define $|x|_F = q^{-{\rm val}_F(x)}$ for $x \in F^\times$.

We let $G$ denote a connected reductive group which is defined and unramified over $F$.  Sometimes we use the symbol $G$ to denote the group $G(F)$ of $F$-points.  Let $A$ denote a maximal $F$-split torus in $G$, and set $T := {\rm Cent}_G(A)$, a maximal torus in $G$ defined over $F$ and split over $L$.  

We consider the (extended) Bruhat-Tits building $\mathcal B(G(L)))$ resp. $\mathcal B(G)$ for $G(L)$ resp. $G(F)$, cf. \cite{BT1},\cite{BT2}.  The Bruhat-Tits buildings associated to the semi-simple $F$-groups $G_{\rm ad}$ and $G_{\rm der}$  are canonically identified and will be denoted $\mathcal B_{\rm ss}(G)$.    The group $G(L) \rtimes {\rm Aut}(L/F)$ resp. $G(F)$ acts on $\mathcal B(G(L))$ resp. $\mathcal B(G)$.  The $\sigma$-fixed subset $\mathcal B(G(L))^\sigma$ can be identified with $\mathcal B(G)$. 

Let $\mathcal A^L$ resp. $\mathcal A$ denote the apartment of $\mathcal B(G(L))$ resp. 
$\mathcal B(G)$ corresponding to the torus $T$ resp. $A$.  
Then $\mathcal A^L$ resp. $\mathcal A$ is endowed with a family of hyperplanes given by the vanishing of the affine roots $\Phi_{\rm aff}(G, T,L)$ resp. $\Phi_{\rm aff}(G,A,F)$ (see \cite{Tits}).  Under the identification $\mathcal B(G) = \mathcal B(G(L))^\sigma$, the apartment $\mathcal A$ is identified with $(\mathcal A^L)^\sigma$.  Moreover, the affine roots $\Phi_{\rm aff}(G,A,F)$ are the non-constant restrictions to $\mathcal A = \mathcal (A^L)^\sigma$ of the affine roots $\Phi_{\rm aff}(G,T,L)$ (\cite{Tits},1.10.1).  The affine roots determine the notions of alcoves, facets, and Weyl chambers used throughout this article.

\subsection{Parahoric subgroups}
 
Fix once and for all a $\sigma$-invariant alcove ${\bf a}$ in $\mathcal A^L$.  Also, fix a $\sigma$-invariant facet ${\bf a}_J$ and a $\sigma$-invariant hyperspecial ``vertex'' ${\bf a}_0$, both contained in the closure of $\bf a$.  According to Bruhat-Tits theory (cf. \cite{Tits}, \cite{BT2}), associated to ${\bf a}_J$ is a smooth affine $\mathcal O_L$-group scheme ${\mathcal G}_{{\mathbf a}_J}^\circ$ with generic fiber $G$ and connected special fiber, with the property that ${\mathcal G}_{{\mathbf a}_J}^\circ(\mathcal O_L)$ fixes pointwise the facet $\mathbf a_J$.  The ``ambient'' group scheme ${\mathcal G}_{{\mathbf a}_J}$ here, of which ${\mathcal G}^\circ_{{\mathbf a}_J}$ is the maximal subgroup scheme with connected geometric fibers, is defined/characterized in \cite{Tits} 3.4.1 as the smooth affine $\mathcal O_L$-group scheme with generic fiber $G$ such that ${\mathcal G}_{{\mathbf a}_J}(\mathcal O_L)$ is the subgroup of $G(L)$ which fixes ${\mathbf a}_J $ pointwise.  Since we have assumed ${\bf a}_J$ is $\sigma$-invariant, $\mathcal G_{{\mathbf a}_J}$ is actually defined over $\mathcal O_F$.

\begin{defn}  A {\em parahoric subgroup} of $G(L)$ is one of the form ${\mathcal G}_{{\mathbf a}_J}^\circ(\mathcal O_L)$, which we shall denote simply by $J(L)$.  A parahoric subgroup of $G(F)$ will be a subgroup of $G(F)$ of the form $J(L) \cap G(F) =: J$.  
\end{defn}

Thus, the facets ${\bf a}_0$, ${\bf a}$ and ${\bf a}_J$ give rise to parahoric subgroups of $G(L)$, specifically a (hyperspecial) maximal parahoric subgroup $K(L)$, an Iwahori subgroup $I(L)$, and a parahoric subgroup $J(L)$.  We have $K(L) \supset I(L) \subset J(L)$.  Since the facets ${\bf a}_0$, ${\bf a}$ and ${\bf a}_J$ are $\sigma$-invariant, we get the corresponding parahoric subgroups $K$, $I$ and $J$ in $G(F)$ by intersecting the parahoric subgroups in $G(L)$ with $G(F)$.  

Of course, one defines similarly the notion of parahoric subgroup for an arbitrary connected reductive group over the field $L$ (or $F$).

\subsection{Alternative descriptions of parahoric subgroups} \label{alt_desc_subsection}

In this subsection we allow $G$ to be any connected reductive group over $L$.  Let $A^L$ denote a maximal $L$-split torus of $G$ (containing $A$ and defined over $F$ in case $G$ is defined over $F$).  Suppose $\mathcal A^L$ is the corresponding apartment in $\mathcal B(G(L))$.  In \cite{Ko97} Kottwitz defined a surjective homomorphism
$$
\kappa_G: G(L) \twoheadrightarrow X^*(Z(\widehat{G})^{\Gamma_0})
$$
where $\Gamma_0 := {\rm Gal}(\overline{L}/L)$ denotes the inertia group.  The homomorphisms $\kappa_G$ vary with $G$ in a functorial manner.  Let $G(L)_1$ denote the kernel of $\kappa_G$.  

The group $G(L)$ acts on the building $\mathcal B_{\rm ss} = \mathcal B(G_{\rm ad}(L))$, and the facet ${\bf a}_J$ determines a facet ${\bf a}^{\rm ss}_J$ in $\mathcal B_{\rm ss}$.  In fact there is a $G(L)$-equivariant isomorphism (which is $G(L) \rtimes {\rm Aut}(L/F)$-equivariant if $G$ is defined over $F$)
$$
\mathcal B(G(L)) = \mathcal B_{\rm ss} \times V_G,
$$
where $V_G := X_*(Z(G))_{\Gamma_0} \otimes \mathbb R$.  Furthermore, we have parallel decompositions
\begin{align*}
\mathcal A^L &= \mathcal A^L_{\rm ss} \times V_G \\
\mathbf a_J &= \mathbf a^{\rm ss}_J \times V_G,
\end{align*} 
where $\mathcal A^L_{\rm ss}$ denotes the apartment of $\mathcal B_{\rm ss}$ corresponding to the maximal $L$-split torus $A^L_{\rm der} \subset G_{\rm der}$. 

 Let ${\rm Fix}({\bf a}^{\rm ss}_J)$ denote the subgroup of $G(L)$ which fixes ${\bf a}^{\rm ss}_J$ pointwise.  

\begin{theorem}[\cite{HR}] \label{HR_parahoric_description}
The parahoric subgroup $J(L) := \mathcal G^\circ_{{\mathbf a}_J}(\mathcal O_L)$ of $G(L)$ can be described as
\begin{equation} \label{1st_description}
J(L) = {\rm Fix}({\bf a}^{\rm ss}_J) \cap G(L)_1 = {\mathcal G}_{{\mathbf a}_J}(\mathcal O_L) 
\cap G(L)_1.
\end{equation}
\end{theorem}

\smallskip

Now suppose the torus $T$ is the centralizer of the maximal $L$-split 
torus $A^L$.  According to \cite{BT2}, the group $T(L)$ contains a unique parahoric subgroup $\mathcal T^\circ(\mathcal O_L)$, where $\mathcal T^\circ$ is the neutral component of the lft N\'{e}ron model $\mathcal T$ for the $L$-torus $T$.  Moreover, we have $\mathcal T^\circ(\mathcal O_L) = T(L)_1$ (cf. \cite{R}).  In \cite{BT2}, 5.2, the group $T^\circ(\mathcal O_L)$ is denoted $\mathfrak Z^\circ(\mathcal O_L)$.

Let $T(L)_b$ denote the maximal bounded subgroup of $T(L)$, that is, the set of $t \in T(L)$ such that ${\rm val}(\chi(t)) = 0$ for all $\chi \in X^*(T)^{\Gamma_0}$.  We have $T(L)_b \supseteq T(L)_1$, and $T(L)_b$ coincides with the kernel of the homomorphism $v_T: T(L) \rightarrow X_*(T)_{\Gamma_0}/torsion$ derived from $\kappa_T$ in the obvious way (\cite{Ko97}, 7.2).  In \cite{BT2}, 5.2, the group $T(L)_b$ is denoted by $\mathfrak Z(\mathcal O_L)$.

Putting the equalities $\mathfrak Z^\circ(\mathcal O_L) = T(L)_1$ and $\mathfrak Z(\mathcal O_L) = T(L)_b$ together with \cite{BT2} 5.2.4 (and identifying our $\mathcal G_{{\mathbf a}_J}$ with the group-scheme $\widehat{\mathcal G}_{{\mathbf a}^{\rm ss}_J}$ of loc.~cit.) we get the decompositions
\begin{align} \label{2nd_description}
J(L) = \mathcal G^\circ_{{\mathbf a}_J}(\mathcal O_L) &= T(L)_1 \cdot \mathfrak U_{{\bf a}_J}(\mathcal O_L) \\
\label{2nd_description'} \mathcal G_{{\mathbf a}_J}(\mathcal O_L) &= T(L)_b \cdot \mathfrak U_{{\bf a}_J}(\mathcal O_L). 
\end{align} 
Here, $\mathfrak U_{{\bf a}_J}$ denotes the group-scheme generated by the root-group schemes $\mathfrak U_{a,{\bf a}_J}$ of loc.~cit.  In particular $a$ denotes a (relative) root of $G_L$ for the torus $A^L$ with corresponding root subgroup $U_a \subset G_L$ and in the notation of \cite{Tits}, 1.4,
\begin{equation} \label{mathfrakU_defn}
\mathfrak U^*_{a,{\bf a}_J}(\mathcal O_L) = \{ u \in U^*_a(L) ~ | ~ \alpha(a,u)(x) \geq 0 \,\, \forall x \in {\bf a}_J \}.
\end{equation}
(The superscript $*$ designates the non-identity elements: $\alpha(a,u)$ is an affine-linear functional on $X_*(A^L)_\mathbb R$ which is defined for $u \neq 1$.)  
We recall that $\alpha(a,u)$ is used to define the filtration $U_{a+l}$ ($l \in \mathbb R$) of $U_a(L)$: an element $u \in U^*_a(L)$ belongs to $U^*_{a+l}$ if and only if $\alpha(a,u) \geq a + l$ (cf. \cite{Tits}, 1.4).  

Finally, suppose $G_L$ is {\em split}.  Then $T$ is $L$-split and so $T(L)_b = T(L)_1$.  Thus (\ref{2nd_description}) and (\ref{2nd_description'}) imply the following corollary.

\begin{cor} \label{HR_for_der=sc}
If $G$ splits over $L$, then $\mathcal G_{{\mathbf a}_J}^\circ(\mathcal O_L) = \mathcal G_{{\mathbf a}_J}(\mathcal O_L)$.
\end{cor}

Clearly this applies to our usual situation, where $G$ is an unramified $F$-group.

\subsection{Unramified characters of $T(F)$} \label{unram_char_subsection}

For this subsection we temporarily change notation and let $T$ denote any $F$-torus and $A$ the maximal $F$-split subtorus of $T$.  Let $T(L)_b$ resp. $T_b$ denote the maximal bounded subgroup of $T(L)$ resp. $T(F)$.

We can embed $X_*(A)$ into $A(F)$ by identifying $\mu \in X_*(A)$ with $\varpi^\mu := \mu(\varpi) \in A(F)$.  

\begin{lemma} \label{T_b_description}   Suppose that $T$ splits over $L$.  Then
we have
\begin{enumerate}
\item[(i)] $T_b = T(F) \cap T(L)_1$;
\item[(ii)] $T(F)/T_b \cong X_*(A)$ via $\kappa_T$.
\end{enumerate}
\end{lemma}

\begin{proof}  Let $V^L$ resp. $V$ denote the real vector space underlying $\mathcal A^L$ resp. $\mathcal A$.   That is, $V^L := X_*(T)_\mathbb R = {\rm Hom}(X^*(T),\mathbb R)$ and $V := X_*(A)_\mathbb R = {\rm Hom}(X^*(A),\mathbb R)$.  Following \cite{Tits}, define $\nu: T(L) \rightarrow V^L$ by mapping $t \in T(L)$ to the homomorphism 
\begin{align*}\nu(t) \,\, :\,\, &X^*(T) \rightarrow \mathbb R \\
&\chi \mapsto -{\rm val}_L(\chi(t)).
\end{align*}
Since $\nu$ is $\sigma$-equivariant and $(V^L)^\sigma = V$, this restricts to a homomorphism $\nu: T(F) \rightarrow V$; the kernel of the latter is $T_b$.  

 There is an inclusion $i: X_*(T) \hookrightarrow X_*(T)_\mathbb R$.  By \cite{Ko97}, (7.4.5), there is a commutative diagram with exact rows
$$
\xymatrix{
0 \ar[r] & T(L)_1 \ar[r] & T(L) \ar[r]^{i \circ \kappa_T} & X_*(T)_\mathbb R \\
0 \ar[r] & T_b \ar[r]  & T(F) \ar[r]^{-\nu} \ar[u] & X_*(A)_\mathbb R \ar[u],}
$$
where the vertical arrows are injective.  This shows that $i \circ \kappa_T$ and $-\nu$ agree as maps $T(F) \rightarrow V$, proving part (i) and the equality of the images in $V$ of $-\nu$ and $\kappa_T|_{T(F)}$.  Furthermore, consider the exact sequence 
$$
\xymatrix{
0 \ar[r] & T(L)_1 \cap T(F) \ar[r] & T(F) \ar[r]^{\kappa_T} & X_*(A) \ar[r] & 0
}
$$
 resulting from taking $\sigma$-invariants of 
$$
\xymatrix{
0 \ar[r] & T(L)_1 \ar[r] & T(L) \ar[r]^{\kappa_T} & X_*(T) \ar[r] & 0,}
$$
and using the fact that 
\begin{equation} \label{coh.vanish} 
H^1(\langle \sigma \rangle, T(L)_1) = 0
\end{equation}
(cf. \cite{Ko97}, (7.6.1)).  We see that $\kappa_T|_{T(F)}$ has image $X_*(A)$, proving (ii).  
\end{proof}

An {\em unramified character} of $T(F)$ is a homomorphism $\chi: T(F)/T(F) \cap T(L)_1 \rightarrow \mathbb C^\times$.  From the previous lemma we deduce the following result.

\begin{lemma} \label{unramified_characters}
If $T$ splits over $L$, an unramified character of $T(F)$ can be described as any of the following:
\begin{enumerate}
\item[(i)] a homomorphism $T(F)/T(F) \cap T(L)_1 \rightarrow \mathbb C^\times$;
\item[(ii)] a homomorphism $T(F)/T_b \rightarrow \mathbb C^\times$;
\item[(iii)] a homomorphism $X_*(A) \rightarrow \mathbb C^\times$.
\end{enumerate}
\end{lemma}

\subsection{Weyl groups and Weyl chambers}

Now we return to the conventions preceding subsection \ref{alt_desc_subsection}, so that $T$ denotes the centralizer of the fixed maximal $F$-split torus $A$ in the unramified $F$-group $G$.  Let $N_S = {\rm Norm}_G(S)$ for any torus $S \subset G$.  We use $_FW$, or $W$, or sometimes $W(F)$ to denote the relative Weyl group $W := N_A(F)/T(F)$ of $G$ over $F$.  The absolute Weyl group of $G$ can be identified with $W(L) := N_T(L)/T(L)$.  

\begin{lemma}\label{absWeyl_relWeyl} We have  $W(L)^\sigma =\,_FW$.  
\end{lemma}

\begin{proof}
Let $R \subset (V^L)^*$ resp. $_FR \subset V^*$ denote the absolute resp. relative roots associated to $(G,T)$ resp. $(G,A)$.  By a result of Steinberg \cite{St}, $W(L)^\sigma$ is the Weyl group $W(R_\sigma)$ of the Steinberg 
root system $R_{\sigma} \subset V^*$.  Let $j: (V^L)^* \rightarrow V^*$ be the restriction homomorphism.  By its very construction $R_\sigma$ consists of the non-zero elements of $j(R)$ which are not smaller multiples of other elements in $j(R)$.  Further, by \cite{Bor91}, 21.8, $_FR$ consists of the non-zero elements of $j(R)$.  It follows that $R_\sigma = (\, _FR)_{\rm nm}$, the non-multipliable relative roots.   Thus, $W(R_\sigma) = W(\, _FR) =  \, _FW$.
\end{proof}   

We will think of ${\bf a}_0$ resp. ${\bf a}^\sigma_0$ as the ``origin'' of our apartment $\mathcal A^L$ resp. $\mathcal A$.  The {\em base alcove} ${\bf a}$ belongs to a unique Weyl chamber $\mathcal C^L_0$ having vertex ${\bf a}_0$, which we will call the {\em base chamber} or the {\em dominant chamber} in $\mathcal A^L$.  It determines the dominant Weyl chamber $\mathcal C_0$ in the apartment $\mathcal A$.

 Let ${\mathfrak B}(T)$ denote the set of Borel subgroups $B = TU$ which contain $T$ and are defined over $F$.  The set $\mathfrak B(T)$ is a torsor for the relative Weyl group $W$ (for $w \in W$ and $B \in \mathfrak B(T)$, let $wB$ or $^wB$ denote $wBw^{-1}$).  For each $B = TU \in \mathfrak B(T)$, define the Weyl chamber $\mathcal C_{U}$ in $\mathcal A$, and the notion of $B$-positive root, as follows.  The chamber $\mathcal C_U$ is the unique one with vertex ${\bf a}^\sigma_0$ 
such that $T_b \, U(F)$ is the union of the fixers of all ``quartiers'' $x + \mathcal C_U$ ($x \in V$) in $\mathcal B(G)$ having the direction of $\mathcal C_U$.  Furthermore, a $B$-positive root is one that appears in ${\rm Lie}(B)$.  Equivalently, a root $\alpha$ is $B$-positive if and only if it takes positive values on the chamber $\mathcal C_{\overline{U}}$, where $\overline{B} = T\overline{U}$ is the unique element of $\mathfrak B(T)$ which is {\em opposite} to $B$.

We may write $\mathcal C_0$ in the form $\mathcal C_{\overline{U}_0}$ for a unique Borel $B_0 = TU_0 \in {\mathfrak B}(T)$.  Thus, the roots $\alpha \in {\rm Lie}(B_0)$ are positive on $\mathcal C_{\overline{U}_0}$, and a coweight $\lambda$ belonging to the closure of $\mathcal C_{\overline{U}_0}$ is $B_0$-{\em dominant}.

In the sequel, a ``positive root'' will mean a $B_0$-positive root, and a ``dominant coweight'' will mean $B_0$-dominant coweight.  Note that by our conventions, the ``reduction modulo $\varpi$'' of $I$ is $\overline{B}_0$.  More precisely, we have $\overline{B}_0 \cap I = \overline{B}_0 \cap K$.

Throughout the paper, we will often write $B$ in place of $B_0$.  Occasionally $B$ will denote instead a general element of $\mathfrak B(T)$, but we shall indicate when this is what is meant.

\subsection{Iwahori Weyl group and extended affine Weyl group} \label{Iwahori_Weyl_grp}

In Bruhat-Tits theory is defined a homomorphism $\nu: N_A(F) \rightarrow V \rtimes W$, which extends the homomorphism $\nu:T(F) \rightarrow V$ discussed earlier (see \cite{Tits}).  Its kernel is $T_b$.  Via $\nu$ the group $\widetilde{W} := N_A(F)/T_b$ can be viewed as a group of affine-linear transformations of $V$.  It splits (the splitting depending on a choice of special vertex in $\mathcal A$) as a semi-direct product
$$
\widetilde{W} \cong \Lambda \rtimes W,
$$
where $\Lambda$  is the group of translations on $V$ isomorphic via $\nu$ to $T(F)/T_b$; $\widetilde{W}$ is termed the {\em extended affine Weyl group}.  There is a natural inclusion of lattices $X_*(A) \hookrightarrow \Lambda$.  In fact Lemma \ref{T_b_description} above shows that $X_*(A) = \Lambda$  (for another proof of this, see \cite{Bor79}, 9.5). 

Define the {\em Iwahori-Weyl group} $\widetilde{W}(L)$ to be the quotient group $N_T(L)/T(L)_1$.  Using $N_T(F) = N_A(F)$, (\ref{coh.vanish}), and Lemma \ref{T_b_description} we see that 
$$
\widetilde{W}(L)^{\sigma} = N_A(F)/T(F) \cap T(L)_1 = \widetilde{W}.
$$
(This can be used to give another proof of Lemma \ref{absWeyl_relWeyl}.)

Now suppose $J(L)$ and $J'(L)$ are parahoric subgroups of $G(L)$ associated to $\sigma$-invariant facets contained in $\overline{\bf a}$.  We may define the subgroup of $\widetilde{W}(L)$ resp. $\widetilde{W} = \widetilde{W}(L)^\sigma$
\begin{align*}
\widetilde{W}_J(L)  := &N_T(L) \cap J(L)/T(L)_1 \\
\mbox{resp.} \,\, \widetilde{W}_J := &\widetilde{W}_J(L)^\sigma = N_A(F) \cap J/T(F) \cap T(L)_1
\end{align*}  
where $J$ is defined as before to be $G(F)\cap J(L)$.

Then by a result from \cite{HR}, there is a canonical bijection
$$J \backslash G(F)/J' = \widetilde{W}_J \backslash \widetilde{W} /\widetilde{W}_{J'}.
$$
We have an analogous result holding for the objects over $L$ in place of $F$.  In particular, 
we have the Bruhat-Tits and also the Iwasawa decompositions over $L$ and $F$:
\begin{align*}
G(L) &= \coprod_{w \in \widetilde{W}(L)} I(L) \, w \, I(L) = \coprod_{w \in \widetilde{W}(L)} U(L) \, w \, I(L) \\
G(F) &= \coprod_{w \in \widetilde{W}} I \, w \, I = \coprod_{w \in \widetilde{W}} U(F) \, w \, I.
\end{align*}
The Iwasawa decompositions on the right hold for any $B = TU \in \mathfrak B(T)$.  Implicit in these decompositions is the choice of a $\sigma$-equivariant {\em set-theoretic} embedding $\widetilde{W}(L) = X_*(T) \rtimes W(L) \hookrightarrow N_T(L)$.  We embed $X_*(T)$ using $\mu \mapsto \mu(\varpi)$, and $W(L)$ by choosing fixed representatives of $W(L)$ in $K(L) \cap 
N_T(L)$.

\subsection{Bruhat orders and length functions}
 When we speak of the Bruhat order on $W$ or $\widetilde{W}$ or the {\em affine Weyl group} $W_{\rm aff}$, we will always mean the Bruhat order defined relative to the reflections through the walls of $\mathcal C_0$ resp. ${\bf a}^\sigma$.  Also, the length function $\ell$ on $\widetilde{W}$ is defined in terms of the reflections through the walls of ${\bf a}^\sigma$.  Analogous conventions specify the Bruhat order and length function on $\widetilde{W}(L)$.

\subsection{(Semi)standard Levi and parabolic subgroups}

In this paper, the following notions of (semi)standard Levi and parabolic subgroups will be useful.  A Levi subgroup $M \subseteq G$ is called an $F$-{\em Levi subgroup} if $M$ is the centralizer in $G$ of an $F$-split torus in $G$.  Equivalently, $M$ is a Levi component of a parabolic subgroup $P \subseteq G$ which is defined over $F$; moreover, in that case $M = {\rm Cent}_G(A_M)$, where $A_M$ is the $F$-split component of the center of $M$.  (See \cite{Bor91}, 20.4 and \cite{Spr}, 15.1).

Recall we have fixed a maximal $F$-split torus $A \subset G$.  
We call an $F$-Levi subgroup $M$ in $G$ {\em semistandard} provided that $A_M \subseteq A$ (it follows that $M \supseteq T$).

We call a parabolic subgroup $P \subset G$ an $F$-{\em parabolic} subgroup if $P$ is defined over $F$, and {\em semistandard} if $P \supset A$.  A semistandard $F$-Levi subgroup $M \subset G$ is a Levi factor of a semistandard $F$-parabolic subgroup (\cite{Bor91}, 20.4).  

A parabolic subgroup $P \subset G$ will be called {\em standard} if $P \supseteq B_0$.  Similarly, a Levi subgroup $M$ will be called {\em standard} if $M$ is the unique semistandard Levi factor of a standard parabolic subgroup.

\subsection{Parahoric subgroups of $F$-Levi subgroups} \label{parahoric_vs_Levis_section}

Let $M$ denote a semistandard $F$-Levi subgroup of $G$.  Suppose that $M$ is the Levi factor of an $F$-parabolic subgroup $P$.  Write $P = MN$, where $N$ denotes the unipotent radical of $P$.  

We wish to show that $J \cap M$ is a parahoric subgroup of $M$, and to describe the corresponding facet in the building $\mathcal B(M)$.  This gives rise to an important cohomology vanishing result (Lemma \ref{J_cap_M_is_parahoric} (c) below) which will be used in the descent of twisted orbital integrals (cf. section \ref{descent}).

Let $\mathcal A^L_M$ denote the apartment in $\mathcal B(M(L))$ corresponding to the torus $T$, so that neglecting the Coxeter complex structures we have $\mathcal A_M^L \cong X_*(T)_{\mathbb R} \cong \mathcal A^L$.  Since the Bruhat-Tits affine roots of $T$ associated to $M(L)$ form a subset of those which are associated to $G(L)$, the affine root hyperplanes in $X_*(T)_{\mathbb R}$ coming from $M$ form a subset of those coming from $G$.  Thus every facet in $\mathcal A^L$ is contained in a unique facet of $\mathcal A_M^L$.  We now define the facet ${\bf a}^M_J \subset \mathcal A_M^L$ to be the unique facet containing ${\bf a}_J$.

\begin{lemma} \label{J_cap_M_is_parahoric} Let $M$ denote a semistandard $F$-Levi subgroup of $G$, and let $P = MN$ denote an $F$-rational parabolic subgroup having $M$ as Levi factor.  Then the following statements hold:
\begin{enumerate}
\item[(a)] $J \cap M$ is the parahoric subgroup of $M$ corresponding to the facet ${\bf a}^M_J$; in particular 
$I \cap M$ is an Iwahori subgroup of $M$;
\item[(b)] $J \cap P = (J \cap M)(J \cap N)$;
\item[(c)] $H^1(\langle \sigma \rangle, J(L) \cap P) = 1$.
\end{enumerate}
\end{lemma}

\begin{proof}

\noindent Part (a):  As a first approximation to this result, note that ${\mathbf a}_J^M \supseteq {\mathbf a}_J$ implies the containment 
\begin{equation} \label{1st_approx_cont}
\mathcal G_{{\mathbf a}_J}(\mathcal O_L) \cap M \supseteq \mathcal G_{{\mathbf a}_J^M}(\mathcal O_L).
\end{equation}

 In view of Theorem \ref{HR_parahoric_description}, we need to show that 
\begin{equation*}
{\rm Fix}({\bf a}^{\rm ss}_J) \cap {\rm ker}(\kappa_G) \cap M = {\rm Fix}({\bf a}^{M,{\rm ss}}_J) \cap {\rm ker}(\kappa_M),
\end{equation*}
or equivalently that 
\begin{equation} \label{Fix_cap_ker}
\mathcal G_{{\mathbf a}_J}(\mathcal O_L) \cap {\rm ker}(\kappa_G) \cap M = \mathcal G_{{\mathbf a}^M_J}(\mathcal O_L) \cap {\rm ker}(\kappa_M).
\end{equation}
By the functoriality of $G \mapsto \kappa_G$ (see \cite{Ko97}) we have a commutative diagram
$$
\xymatrix{
M(L) \ar[r] \ar[d]_{\kappa_M} & G(L) \ar[d]_{\kappa_G} \\
X^*(Z(\widehat{M})^{\Gamma_0}) \ar[r] & X^*(Z(\widehat{G})^{\Gamma_0}),}
$$
which together with (\ref{1st_approx_cont}) shows that the inclusion $''\supseteq''$ holds 
in (\ref{Fix_cap_ker}).
 
We need to prove the inclusion ``$\subseteq$'' in (\ref{Fix_cap_ker}).
Let $I'_M$ denote any Iwahori subgroup of $M(L)$ corresponding to an alcove ${\mathbf a}'_M$ of $\mathcal A^L_M$ whose closure contains ${\mathbf a}_J^M$.  The following claim clearly suffices, since the affine Weyl group $W_{M, \rm aff}$ and the Iwahori subgroup $I'_M$ for $M$ belong to ${\rm ker}(\kappa_M)$.

\smallskip

 {\bf Claim:}  Any element of $G(L)_1 \cap M$ which fixes one point in 
${\bf a}^M_J$ belongs to $I'_M \, W_{M, \rm aff}\, I'_M$ and fixes every point in ${\bf a}^M_J$.

\smallskip

{\em Proof:}  Suppose $x \in G(L)_1 \cap M$ fixes a point in ${\bf a}^M_J \subset \overline{{\mathbf a}_M'}$.  
By the Bruhat-Tits decomposition $I'_M(L) \backslash M(L)/I'_M(L) = \widetilde{W}_M(L)$, we may assume $x \in \widetilde{W}_M(L) = X_*(T) \rtimes W_M(L)$.  Write $x = \varpi^\lambda \, w$ for $\lambda \in X_*(T)$ and $w \in W_M(L)$ (where $w$ is viewed in $M(L)$ using any choice of representative in $K(L) \cap M$).  We have $\varpi^\lambda \in G(L)_1$.

Now since $G$ is split over $L$, we have $X^*(Z(\widehat{G})^{\Gamma_0}) = X^*(Z(\widehat{G})) = X_*(T)/Q^\vee$, where $Q^\vee$ denotes the lattice in $X_*(T)$ spanned by the coroots for $T$ in ${\rm Lie}(G)$ (the final equality is due to Borovoi \cite{Bo}).  Furthermore, the compatibility of $\kappa_G$ and 
$$
\kappa_T : T(L) \rightarrow X_*(T)$$ 
means that $\kappa_G$ sends $\varpi^\lambda$ to the image of $\lambda$ in $X_*(T)/Q^\vee$.  It follows that $\lambda \in Q^\vee$.  

Let $Q_M^\vee$ denote the lattice in $X_*(T)$ spanned by the coroots for $T$ in ${\rm Lie}(M)$.  Let $n$ denote the order of $w$.  Then $x^n$ is the translation by the coweight $\mu := \sum_{i = 0}^{n-1} w^i\lambda$.  Since this fixes a point in ${\mathbf a}^M_J$, we must have $\mu = 0$.  But $w^i\lambda \equiv \lambda$ modulo $Q^\vee_M$, and it follows that $\lambda \in (Q^\vee_M)_\mathbb Q \cap Q^\vee = Q^\vee_M$.  Thus, $x \in W_{M,\rm aff}$.  This means that $x$ is a type-preserving automorphism of the apartment $\mathcal A^L_M$ which fixes a point in the facet ${\bf a}_J^M$, and this implies that it fixes ${\bf a}_J^M$ pointwise.  This proves the claim.  \qed

\medskip

\noindent Part (b):  Write $\mathcal O_L =: \mathcal O$.  Choosing a faithful $\mathcal O$-representation $\rho: 
\mathcal G^\circ_{{\mathbf a}_J} \rightarrow {\rm GL}_{n,\mathcal O}$, and viewing $J(L) = \mathcal G^\circ_{{\mathbf a}_J}(\mathcal O)$ as the subgroup of matrices in $G(L) \hookrightarrow {\rm GL}_n(L)$ which have entries in $\mathcal O$, the equality $\mathcal G^\circ_{{\mathbf a}_J}(\mathcal O) \cap P = (\mathcal G^\circ_{{\mathbf a}_J}(\mathcal O) \cap M)(\mathcal G^\circ_{{\mathbf a}_J}(\mathcal O) \cap N)$ follows from the corresponding one for parabolic subgroups of ${\rm GL}_n$.  

\medskip

\noindent Part (c):  Since by (a) $J(L) \cap M$ is the $\mathcal O_L$-points of an $\mathcal O_F$-group scheme with connected geometric fibers, we have $H^1(\langle \sigma \rangle, J(L) \cap M) = 1$ (by Greenberg's theorem \cite{Gr}).  Thus in view of (b) it is enough to show that $H^1(\langle \sigma \rangle, J(L) \cap N ) = 1$.  We use \cite{BT2}, 5.1.8, 5.1.16-18, 5.2.2-4.  Since $J(L) \cap N$ is $\sigma$-stable, by descent it comes from a smooth connected group scheme over $\mathcal O_F$.  
The latter is itself a product of $\sigma$-stable connected smooth $\mathcal O_F$-subgroups ${\mathfrak U}^\natural_{b,(k,l)}$ (notation of loc.~cit. 5.2.2) which is the scheme associated to a locally free $\mathcal O_F$-module of finite type.  It then follows as in loc.~cit. 5.1.18 (i.e. ultimately from Shapiro's lemma and 
$H^1(\langle \sigma \rangle, \mathcal O_L) = 1$) that $H^1(\langle \sigma \rangle, ({\mathfrak U}^{\natural}_{b,(k,l)})_L) = 1$, and from this $H^1(\langle \sigma \rangle, J(L) \cap N) = 1$ follows.
\end{proof}

\section{Definition and properties of the base change homomorphism} \label{base_change_hom_section}

Let $G,A,T$ etc. be as in the previous section.

\subsection{Bernstein isomorphism}

Let $\chi$ be an unramified character of $T(F)$, that is, a character $\chi: T(F) \rightarrow \mathbb C^\times$ which is trivial on $T_b = T(F) \cap T(L)_1$.  By Lemma \ref{unramified_characters}, $\chi$ can be identified with a homomorphism $X_*(A) \rightarrow \mathbb C^\times$.  Note that $\chi: X_*(A) \rightarrow \mathbb C^\times$ extends by linearity to give a $\mathbb C$-algebra homomorphism $\chi: \mathbb C[X_*(A)] \rightarrow \mathbb C$.  

Let $B \in \mathfrak B(T)$.  For any unramified character $\chi$, consider the (unitarily normalized) unramified 
principal series $i^G_B(\chi) := {\rm Ind}^G_B(\delta_B^{1/2}\chi)$.  
The Hecke algebra $\mathcal H_J(G)$ acts on the right on the $J$-invariants 
$i^G_B(\chi)^J$ by convolution $*$ of functions (using the Haar measure on $G$ giving $J$ measure 1).  

Now let $R := \mathbb C[X_*(A)]$, an algebra with action of the Weyl group $_FW =: W$.  Recall that for $J = I$, the Bernstein isomorphism is an isomorphism
$$
B: R^{W} ~ \widetilde{\rightarrow} ~ Z(\mathcal H_I(G)),
$$
characterized in terms of the action of $Z(\mathcal H_I(G))$ on unramified principal series, as follows.

Write $B = TU$ as before, and abbreviate $A(F) =: A$ and $A(\mathcal O_F) =: A_\mathcal O$.  Further write 
$\mathcal H$ in place of $\mathcal H_I(G)$.  Let ${\bf M} = {\bf M}_B$ be defined to be the complex vector space
$$
{\bf M} = C_c(A_\mathcal O U\backslash G/I) = C_c(T_bU \backslash G/I).
$$
The subscript ``c'' means that we consider functions supported on only finitely many double cosets.  As a complex vector space, ${\bf M}$ has a basis consisting of the functions  $v_x := 1_{A_\mathcal OUxI}$ ($x \in \widetilde{W}$).

The vector space ${\bf M}$ is an $(R,\mathcal H)$-bimodule.  It is clear that $\mathcal H$ acts on the right on ${\bf M}$ by right convolutions $*$.  One proves as in \cite{HKP} Lemma 1.6.1 that ${\bf M}$ is free of rank 1 as an $\mathcal H$-module, with canonical generator $v_1$.  (This fact was observed earlier by Chriss and Khuri-Makdisi \cite{CK}, who derived it from the Bernstein presentation of $\mathcal H$.)  The ring $R$ can be viewed as the Hecke algebra $C_c(A/A_\mathcal O)$, where convolution of functions is defined using the Haar measure $da$ on $A$ which gives $A_\mathcal O$ measure 1.  With this in mind, $R$ acts on the left on ${\bf M}$ by {\em normalized} left convolutions $\cdot$.  More precisely, letting $r \in R$, we define the left action of $r$ on $\varphi \in {\bf M}$ by the integral 
\begin{equation} \label{norm_left_convolution}
r \cdot \varphi(g) = \int_{A} r(a) \, \delta^{1/2}_B(a) \, \varphi(a^{-1}g) \, da.
\end{equation}
In other words, if $\lambda \in X_*(A)$ and if $\varpi^\lambda$ is regarded as both an element in $A/A_\mathcal O$ {\em and} as the characteristic function on $A/A_\mathcal O$ for the subset $\varpi^\lambda$, then
\begin{equation*}  
\varpi^\lambda \cdot v_{x} = \delta^{1/2}_B(\varpi^\lambda)\, v_{\varpi^\lambda x}.
\end{equation*}

The Bernstein isomorphism $B : R^W ~ \widetilde{\rightarrow} ~ Z(\mathcal H_I(G))$ is characterized by the identity
\begin{equation} \label{R_action_char_of_B}
B^{-1}(z) \cdot \varphi = \varphi * z
\end{equation}
for every $z \in Z(\mathcal H_I(G))$, $\varphi \in {\bf M}$.

We have an identification $\mathbb C \otimes_{R, \chi^{-1}} {\bf M} = i^G_B(\chi)^I$ which respects the right $\mathcal H$-actions on each side (see \cite{HKP}, \cite{H07}).  We use the identification to transport the left $R$-action on ${\bf M}$ to one on $i^G_B(\chi)^I$, also denoted $\cdot$.  Say $\phi \in i^G_B(\chi)^I$ corresponds to $\varphi \in {\bf M}$.  Then for $r \in R$ and $g \in G(F)$, we have
\begin{equation} \label{R_action}
(r \cdot \phi)(g) = \int_{A} r(a) \, \delta_B^{1/2}(a) \, \phi(a^{-1}g) \, da = 
\sum_{a \in A/A_{\mathcal O}} r(a) \, \delta_B^{1/2}(a) \, \phi(a^{-1}g),
\end{equation}
which since $\phi(a^{-1}g) = (\delta_B^{-1/2}\chi^{-1})(a) \, \phi(g)$ can be written as
\begin{equation} \label{R_action'}
 \Big(\sum_{a \in A/A_{\mathcal O}} r(a) \, \chi^{-1}(a)\Big) \, \phi(g) = \chi^{-1}(r) \, \phi(g).
\end{equation}
In other words, {\em  $z \in Z(\mathcal H_I(G))$ acts  by right convolutions on $i^G_B(\chi)^I$ by the scalar $ch_\chi(z) := \chi^{-1}(B^{-1}(z))$}.

An exposition of these facts for split groups can be found in \cite{HKP}, and for unramified groups they can be proved in a similar way, see \cite{H07}.   Some other references containing overlapping material are \cite{BD} and \cite{CK}.

\smallskip

The following generalization of the Bernstein isomorphism to parahoric Hecke algebras can be deduced from the theory of the Bernstein center (see \cite{BD}, Prop. 3.14), but we shall give a more elementary proof here, still based on ideas of Bernstein (as presented in \cite{HKP}, \cite{H07}).  We now change slightly our notation.  Let ${\bf \cdot}$ (resp. $*$) denote the convolution product in ${\mathcal H}_I(G)$ (resp. $\mathcal H_J(G)$) determined by the Haar measure which gives $I$ (resp. $J$) measure 1.

\begin{theorem}[Bernstein isomorphism] \label{Bernstein_isomorphism}
Convolution by the characteristic function $\mathbb I_J$ gives an isomorphism 
$$- {\bf \cdot} \mathbb I_J : Z(\mathcal H_I(G)) \rightarrow Z(\mathcal H_J(G)).$$  
The composition 
$$\xymatrix{
R^{_FW} \ar[r]^B & Z(\mathcal H_I(G)) \ar[r]^{-{\bf \cdot} \mathbb I_J} & Z(\mathcal H_J(G))
}$$
is an isomorphism (still denoted $B$) characterized by the following property:  under the right convolution action, any $z \in Z(\mathcal H_J(G))$ acts on the representation $i^G_B(\chi)^J$ by the scalar
$$
ch_\chi(z) := \chi^{-1}(B^{-1}(z)).
$$
\end{theorem}
 
\begin{proof}
Again write $\mathcal H$ in place of $\mathcal H_I(G)$. 
Write $L$ for the fraction field ${\rm Frac}(R)$ and note that $L^W = {\rm Frac}(R^W)$.  Write $e_J = [J:I]^{-1}{\mathbb I}_J$, an idempotent in $\mathcal H_I(G)$.  The map $h \mapsto [J:I]^{-1}h$ gives an isomorphism of algebras
$$
({\mathcal H}_J(G), *, \mathbb I_J) ~ \widetilde{\rightarrow} ~ (e_J{\mathcal H}e_J, {\bf \cdot}, e_J).$$
Thus, it suffices to show that $-{\bf \cdot} e_J: Z(\mathcal H) \rightarrow Z(e_J\mathcal H e_J)$ is an isomorphism.  We have a commutative diagram
$$
\xymatrix{
 \mathcal H \ar[r]^{e_J {\bf \cdot} - {\bf \cdot} e_J} & \mathcal H_J \\
  Z(\mathcal H) \ar[u] \ar[r]^{-{\bf \cdot} e_J} \ar[d] & Z(\mathcal H_J) \ar[d] \ar[u]\\
 L^W \ar[r] & L^W \otimes_{R^W, e_J} Z(\mathcal H_J),}$$
where we are using the Bernstein isomorphism for the case $J = I$ to identify $Z(\mathcal H)$ with $R^W$, and thus to embed $Z(\mathcal H)$ into $L^W$.  Further, the tensor product is formed using the inclusion $R^W \hookrightarrow L^W$ and $- {\bf \cdot} e_J: R^W \rightarrow Z(\mathcal H_J)$.  

The latter map is injective.  Indeed, suppose $z \in Z(\mathcal H)$.  Then for any $\chi$, the element $z e_J \in Z(\mathcal H_J)$ acts on the right on $i^G_B(\chi)^J$ by the same scalar, namely $ch_\chi(z)$, by which $z$ acts on $i^G_B(\chi)^I$.  Also, $i^G_B(\chi)^J \neq 0$ for all $\chi$.  So, if $z e_J = 0$, we have $ch_\chi(z) = 0$ for every $\chi$, and this implies that $z = 0$.

Further, $\mathcal H_J$ is torsion-free as an $R^W$-module (via $- {\bf \cdot} e_J$), and so $Z(\mathcal H_J) \rightarrow L^W \otimes_{R^W,e_J} Z(\mathcal H_J)$ is injective too.

\begin{lemma} \label{key_lemma}
The canonical map $L^W \rightarrow L^W \otimes_{R^W, e_J} Z(\mathcal H_J)$ is an isomorphism.  In particular, $L^W \otimes_{R^W,e_J} Z(\mathcal H_J)$ is a field, and its subring $Z(\mathcal H_J)$ is a domain.
\end{lemma}

To prove Lemma \ref{key_lemma}, first we show that
$$
L^W \otimes_{R^W,e_J} Z(\mathcal H_J) = Z(L^W \otimes_{R^W,e_J} {\mathcal H}_J).
$$
Indeed, it is clear that the left hand side is contained in the right hand side.  To prove the opposite inclusion, observe that elements in $L^W \otimes_{R^W,e_J} \mathcal H_J$ can be expressed as pure tensors $\frac{r}{s} \otimes h$ ($r,s \in R^W$, and $h \in \mathcal H_J$), and that $h \mapsto 1 \otimes h$ gives an 
{\em injection} $\mathcal H_J \hookrightarrow L^W \otimes_{R^W,e_J} \mathcal H_J$ (since $\mathcal H_J$ is torsion-free as an $R^W$-module via $- {\bf \cdot} e_J$).

Next, note that 
$$
L^W \otimes_{R^W,e_J} \mathcal H_J = (1 \otimes e_J) (L^W \otimes_{R^W} \mathcal H)(1 \otimes e_J),
$$
as $L^W$-algebras.  The isomorphism associates an element $\frac{1}{s} \otimes e_J h e_J$ on the left hand side to the element $(1 \otimes e_J)(\frac{1}{s} \otimes h)(1 \otimes e_J)$ on the right hand side.

Finally, as in the proof of Lemma 2.3.1 from \cite{HKP}, the algebra $L^W \otimes_{R^W} \mathcal H$ is a matrix algebra over $L^W$, and $1 \otimes e_J$ is an idempotent in that algebra.  But it is an elementary exercise in linear algebra to show that for any field $k$ and idempotent $e \in M_n(k)$, we have $Z(eM_n(k)e) = ek$.  This completes the proof of Lemma \ref{key_lemma}. \qed

\medskip

Now we complete the proof of the theorem.  The ring $R$ is finite over $R^W$, and $\mathcal H$ is finite over $R$.  Thus $\mathcal H$ is finite over $R^W$.  Since $R^W$ is Noetherian, $Z(\mathcal H_J)$ is finite hence integral over $e_JR^W$ as well.  By Lemma \ref{key_lemma} and the above remarks, $R^W \isom e_JR^W$ and $Z(\mathcal H_J)$ have the same fraction field, namely $L^W$.  The equality $e_JR^W = Z(\mathcal H_J)$ follows since $R$ (hence also $R^W$) is an integrally closed domain.

The assertion concerning the action of $Z(\mathcal H_J)$ on unramified principal series now follows from the special case $J =I$, see (\ref{R_action_char_of_B}-\ref{R_action'}).
\end{proof}

In the sequel, for a coweight $\mu \in X_*(A)$, the function $z_\mu = z^J_\mu$ will denote the unique element of $Z(\mathcal H_J(G))$ such that 
$$
B^{-1}z_\mu = \sum_{\lambda \in W(F) \mu} t_\lambda
$$
where $t_\lambda$ denotes the the element $\lambda \in X_*(A)$, viewed as an element of the group algebra $\mathbb C[X_*(A)]$.  

\subsection{Base change homomorphism}

Let $E \supset F$ be an unramified extension of degree $r$, and let $\theta$ denote a generator for ${\rm Gal}(E/F)$.  Let $A^E$ denote a maximal $E$-split torus of $G$, containing $A$.  We may assume $A^E$ is defined over $F$ 
(using e.g. \cite{BT2}, 5.1.12).  Write $W(E)$ (resp. $W(F)$) for the relative Weyl group of $G$ over $E$ (resp. $F$); $W(E)$ then acts on the cocharacter lattice $X_*(A^E)$.  For $\nu \in X_*(A^E)$ note that $\sum_{i=0}^{r-1} \theta^i\nu \in X_*(A)$.  This defines the {\em norm homomorphism}
$$
N: \mathbb C[X_*(A^E)]^{W(E)} \rightarrow \mathbb C[X_*(A)]^{W(F)}.
$$
(Here we used $W(F) = W(E)^\theta$; see Lemma \ref{absWeyl_relWeyl}.)  Now we define the {\em base change homomorphism} 
$$
b : Z(\mathcal H_J(G(E)))  \rightarrow Z(\mathcal H_J(G(F)))
$$
to be the unique map making the following diagram commute:
$$
\xymatrix{
\mathbb C[X_*(A^E)]^{W(E)} \ar[r]^B_{\sim} \ar[d]^N & Z(\mathcal H_J(G(E))) \ar[d]^b \\
\mathbb C[X_*(A)]^{W(F)} \ar[r]^B_{\sim}  & Z(\mathcal H_J(G(F))).}
$$

\subsection{Compatibilities of the Bernstein isomorphism}
\subsubsection{Compatibility with change of parahoric, and Satake isomorphism}
First suppose that $J_1 \subset J_2$ are two parahoric subgroups containing $I$.  By construction of the Bernstein isomorphisms $B_1$ and $B_2$, the following diagram commutes.
$$
\xymatrix{
R^W \ar[d]^{=} \ar[r]^{B_1\,\,\,\,\,\,\,}_{\sim \,\,\,\,\,\,\,} & Z(\mathcal H_{J_1}(G)) \ar[d]^{- *_{J_1} 
\mathbb I_{J_2}} \\
R^W \ar[r]^{B_2 \,\,\,\,\,\,\,}_{\sim \,\,\,\,\,\,\,} & Z(\mathcal H_{J_2}(G)).}
$$
Here $-*_{J_1}$ refers to convolution using the Haar measure giving $J_1$ measure 1.

In a special case this is a restatement of the compatibility of the Bernstein and Satake isomorphisms (see \cite{HKP}, section 4.6).  Namely, if $J_1 = J$ is contained in a (hyper)special  maximal compact $J_2 = K$, then the following diagram commutes
$$
\xymatrix{
R^W \ar[d]^{=}\ar[r]^{B \,\,\,\,\,\,}_{\sim \,\,\,\,\,\,\,} & Z(\mathcal H_J(G)) \ar[d]^{-*_J \mathbb I_K} \\
R^W & \mathcal H_K(G) \ar[l]_S^{\sim},
}
$$
where $S$ denotes the Satake isomorphism.

It follows that for the base change homomorphisms $b_1$ and $b_2$, and $z \in Z(\mathcal H_{J_1}G(E))$ we have
\begin{equation} \label{eq:b_vs_change_in_J}
b_2(z *_{J_1(E)} \mathbb I_{J_2(E)}) = b_1(z) *_{J_1(F)} \mathbb I_{J_2(F)}.
\end{equation}


\subsubsection{Compatibility with constant term homomorphism}

Fix an $F$-Levi subgroup $M \supset T$ and an $F$-parabolic $P=MN$.  Recall the definition of the modulus character for $P$, a function on $m \in M(F)$ given by
$$
\delta_P(m) := |{\rm det}({\rm Ad}(m); {\rm Lie} N(F))|_F,
$$
where $|\cdot|_F$ is the normalized absolute value on $F$.  Here ${\rm Ad}(m)$ is the derivative of the conjugation action $(m,n) \mapsto mnm^{-1}$, for $n \in N(F)$.

For a compactly-supported locally constant function $f$ on $G(F)$, we define the locally constant function $f^{(P)}$ on $M(F)$ by the formula
\begin{equation}
f^{(P)}(m) = \delta_P^{1/2}(m) \int_{N(F)} f(mn) \, dn = \delta_P^{-1/2}(m) \int_{N(F)} f(nm) \, dn,
\end{equation}
where $dn$ is the Haar measure giving $N(F) \cap J$ measure 1.  (Warning: $f^{(P)}$ depends on the choice of $J$ since the latter is used to define the Haar measure; in the sequel we will allow $J$ to vary.)  

Note that this definition of $f^{(P)}$ differs from the conventional one (as in, e.g. \cite{Cl90}, 2.1), in that it does not incorporate ``averaging over $K$-conjugacy''.  However, this notion is exactly what is needed in our context, in part because of the following important fact, which will be proved in subsection \ref{proof_of_B_constant_term_compatibility}.  Also, it is well-adapted to descent in our setting, as shown by the descent formulas (\ref{eq:descent_over_E}) and (\ref{eq:1st_descent_over_F}) below which relate (twisted) orbital integrals on $G$ and $M$.

\begin{prop} \label{B_and_constant_term_compatibility}
The constant term preserves centers of parahoric Hecke algebras.  If $f^{(P)}$ is defined using the parahoric subgroup 
$J$, then the following diagram is commutative.
$$
\xymatrix{
R^{W_G(F)} \ar[r]^{B \,\,\,\,\,\,\,\,}_{\sim \,\,\,\,\,\,\,\, } \ar[d]_{incl.} & Z(\mathcal H_J(G(F))) 
\ar[d]^{f \mapsto f^{(P)}} \\
R^{W_M(F)} \ar[r]^{B \,\,\,\,\,\,\,\,\,\,}_{\sim \,\,\,\,\,\,\,\,\,\,} & Z(\mathcal H_{J\cap M}(M(F))).}
$$
\end{prop}

Since the norm homomorphisms for $M$ and $G$ are obviously compatible with the inclusion 
$$\mathbb C[X_*(A^E)]^{W_G(E)} \hookrightarrow \mathbb C[X_*(A^E)]^{W_M(E)}$$ and its analogue for $F$ replacing $E$, 
we have the following corollary.

\begin{cor} \label{b_constant_term_compatibility}With the assumptions above,
the following diagram commutes
$$
\xymatrix{
Z(\mathcal H_J(G(E))) \ar[r]^{f \mapsto f^{(P)}\,\,\,\,\,\,} \ar[d]^b & Z(\mathcal H_{J \cap M}(M(E))) \ar[d]^b \\
Z(\mathcal H_J(G(F))) \ar[r]^{f \mapsto f^{(P)}\,\,\,\,\,\,\,} & Z(\mathcal H_{J \cap M}(M(F))).}
$$
\end{cor}

\subsubsection{Compatibility with conjugation by $w \in W(F)$} \label{w_conjugation_subsection}

Recall that for $\phi \in \mathcal H_J(G(F))$, and $w \in W(F)$, we set $^w\phi(g) = \phi(w^{-1}gw)$.  The map $\phi \mapsto \, ^w\phi$ obviously gives an isomorphism $\mathcal H_J(G(F)) ~\widetilde{\rightarrow} \mathcal ~ H_{\, ^wJ}(G(F))$.

Fix an unramified character $\xi$ of $T(F)$, and define for $\Phi \in i^G_B(\xi)^J$ the function $^w\Phi \in i^G_{\,^wB}(\, ^w\xi)^{\, ^wJ}$ by $^w\Phi(g) := \Phi(w^{-1}gw)$.  The map $\Phi \mapsto \, ^w\Phi$ determines an isomorphism
$$
i^G_B(\xi)^J ~ \widetilde{\rightarrow} ~ i^G_{\, ^wB}(\, ^w\xi)^{\, ^wJ},
$$
where $^w\xi(t) := \xi(w^{-1}tw)$ for $t \in T(F)$.   This intertwines the right actions of $\mathcal H_J(G)$ resp. $\mathcal H_{\, ^wJ}(G)$, in the sense that
$$
^w\Phi * \, ^w\phi  = \, ^w(\Phi * \phi).
$$
From this it easily follows that the following diagram commutes:
$$
\xymatrix{
 R^{W(F)} \ar[r]^{B \,\,\,\,\,\,\,}_{\sim \,\,\,\,\,\,\,} \ar[d]^{=} & Z(\mathcal H_J(G(F))) \ar[d]^{\phi \mapsto \, ^w\phi} \\
 R^{W(F)} \ar[r]^{B \,\,\,\,\,\,\,\,\,}_{\sim \,\,\,\,\,\,\,\,\,} & Z(\mathcal H_{\, ^wJ}(G(F))).}
$$

Of course the analogue of this also holds for the field extension $E$ replacing $F$.  As a consequence we get 
the following lemma.

\begin{lemma} \label{w_conjugation_compatibility}
For $w \in W(F)$, the following diagram commutes
$$
\xymatrix{
Z(\mathcal H_J(G(E))) \ar[r]^{\phi \mapsto \, ^w\phi} \ar[d]^b & Z(\mathcal H_{\, ^wJ}(G(E)))  \ar[d]^b \\
Z(\mathcal H_J(G(F))) \ar[r]^{\phi \mapsto \, ^w\phi} & Z(\mathcal H_{\, ^wJ}(G(F))).}
$$
\end{lemma}

\subsection{Extending $b$ to a map $\overline{b}: \mathcal H_{J(E)}/{\rm ker} \rightarrow \mathcal H_J/{\rm ker}$}

The discussion in this and the following subsection is aimed toward proving the Schwartz-continuity of $b$ (Corollary \ref{b_is_Schwartz_continuous}).  This is needed in the final argument of this paper (section \ref{putting_it_all_together_section}), where we will apply Clozel's temperedness argument (section \ref{temperedness_section}) to a certain linear form which involves $b$.  The Schwartz-continuity of $b$ is necessary in order to know that that form is also Schwartz-continuous, as section 8 requires.  

In the spherical case considered in \cite{Cl90}, the situation is substantially simplified due to the commutativity of the spherical Hecke algebras, and much of the following material is unnecessary.  

\subsubsection{Preliminaries on $N: \widehat{A} \rightarrow \widehat{A^E}$}
Let $\widehat{A}$ denote the complex torus which is dual to the $F$-split torus $A$; it carries a canonical action of the relative Weyl group $W(F)$ associated to $A \subset G$.  Let $\widehat{A}_u$ denote the maximal compact subgroup of $\widehat{A}$.  Replacing the field $F$ with its extension $E$, we define analogous objects $\widehat{A^E}$ and $\widehat{A^E}_u$.  Let $\mathbb C[\widehat{A}/W(F)]$ denote the ring of regular functions on the complex variety $\widehat{A}/W(F)$, which we identify with $\mathbb C[X_*(A)]^{W(F)}$.  The norm homomorphism $N: \mathbb C[X_*(A^E)]^{W(E)} \rightarrow 
\mathbb C[X_*(A)]^{W(F)}$ induces a homomorphism
\begin{equation} \label{N_on_functions}
N: \mathbb C[\widehat{A^E}/W(E)] \rightarrow \mathbb C[\widehat{A}/W(F)].
\end{equation}
Let us describe the corresponding morphism of complex varieties. Let $\Gamma_{E/F}$ denote the Galois group 
${\rm Gal}(E/F) = \langle \theta \rangle$; it acts on the complex torus $\widehat{A^E}$.  Since $A$ is the $F$-split component of $A^E$, it follows that 
$$
\widehat{A} = (\widehat{A^E})_{\Gamma_{E/F}}.
$$
For $t \in \widehat{A^E}$, we define $Nt = t \theta(t) \cdots \theta^{r-1}(t) \in \widehat{A^E}$.  This determines a homomorphism of complex tori
$$
N: \widehat{A} \rightarrow \widehat{A^E}
$$
and thus a morphism (also denoted with the symbol $N$) of complex varieties
\begin{equation} \label{N_on_tori}
N: \widehat{A}/W \rightarrow \widehat{A^E}/W(E).
\end{equation}
The corresponding homomorphism on the level of regular functions is precisely (\ref{N_on_functions}) which is defined to be compatible with the norm homomorphism $N = \sum_{i=0}^{r-1} \theta^i$ defined earlier.

\subsubsection{Definition of $f \mapsto \widehat{f}$}

An element $t \in \widehat{A}$ can be regarded as an unramified character on $T(F)$.   We set $\pi_t := i^G_B(t)$.  To any function $f \in \mathcal H_J(G)$ we associate its {\em Fourier transform} $\widehat{f}$, a regular function on the variety $\widehat{A}/W$ defined by the equation
\begin{equation}
\widehat{f}(t) := \langle {\rm trace}\, \pi_t, f \rangle.
\end{equation}
Set
$$
{\rm ker} = \{ f \in \mathcal H_J(G) ~  | ~ \widehat{f}(t) = 0 , \,\,\,\ \forall t \in \widehat{A}\}.
$$
The map $f \mapsto \widehat{f}$ determines a $\mathbb C$-vector space isomorphism
\begin{equation}
\mathcal H_J(G)/{\rm ker} ~ \widetilde{\rightarrow} ~ \mathbb C[\widehat{A}/W].
\end{equation}
(The homomorphism is surjective since its restriction to $Z(\mathcal H_J(G))$ is obviously surjective, by the Bernstein isomorphism, Theorem \ref{Bernstein_isomorphism}.)

This discussion applies just as well when we replace $F$ with its extension field $E$.  We can then define the $\mathbb C$-vector space homomorphism
$$
\overline{b}: \mathcal H_J(G(E))/{\rm ker} \rightarrow \mathcal H_J(G)/{\rm ker}
$$
to be the unique one making the following diagram commutative:
$$
\xymatrix{
\mathcal H_J(G(E))/{\rm ker} \ar[r]^{f \mapsto \widehat{f}}_{\sim} \ar[d]_{\overline{b}} & \mathbb C[\widehat{A^E}/W(E)] \ar[d]_{N} \\
\mathcal H_J(G)/{\rm ker} \ar[r]^{f \mapsto \widehat{f}}_{\sim} & \mathbb C[\widehat{A}/W].}
$$
In other words, $\overline{b}f$ is characterized by the identity
\begin{equation}
\langle {\rm trace} \, \pi_t, \overline{b}f \rangle = \langle {\rm trace} \, \pi_{Nt}, f \rangle.
\end{equation}

\subsubsection{Compatibility of $\overline{b}$ and $b$}

From the discussion above, the $\mathbb C$-vector space homomorphism $Z(\mathcal H_J(G)) \rightarrow 
\mathcal H_J(G)/{\rm ker}$ is an isomorphism.  

\begin{lemma} \label{b_vs_overline_b}
\begin{enumerate}
\item[(i)] The map $\overline{b}$ is linear with respect to $b$: for any $f \in \mathcal H_J(G(E))/{\rm ker}$ and $z \in Z(\mathcal H_J(G(E)))$, we have $\overline{b}(zf) = b(z)\overline{b}f$.
\item[(ii)] The following diagram commutes:
$$
\xymatrix{
Z(\mathcal H_J(G(E))) \ar[r]^{\sim} \ar[d]_{C\cdot b}  & \mathcal H_J(G(E))/{\rm ker} 
\ar[d]_{\overline{b}} \\
Z(\mathcal H_J(G)) \ar[r]^{\sim} & \mathcal H_J(G)/{\rm ker},}
$$
where the constant $C$ is the cardinality of $B(E)\backslash G(E)/J(E)$ divided by the cardinality of $B(F)\backslash G(F)/J(F)$.
\end{enumerate}
\end{lemma}

\subsection{The Schwartz-continuity of $b$}
\subsubsection{Preliminaries for extending $\overline{b}$ to Schwartz spaces}

We will define an extension $\widetilde{b}$ of $\overline{b}$ to appropriate spaces of Schwartz functions.

For $t \in \widehat{A}_u$, the representation $\pi_t$ is tempered and we may define
$$
\widehat{f}(t) := \langle {\rm trace} \, \pi_t, f \rangle,
$$
for $f$ belonging to the Schwartz space $\mathcal C(G(F))$.  We will consider this for $f \in \mathcal C_J(G)$, the space of $J$-bi-invariant Schwartz functions on $G(F)$.  Set 
$$
{\rm ker} = \{ f \in C_J(G)  ~ | ~ \widehat{f}(t) = 0, \,\,\,\, \forall t \in \widehat{A}_u \}.
$$

Since $\widehat{A}_u$ is a Zariski-dense subset of $\widehat{A}$, we have inclusions
\begin{align*}
\mathbb C[\widehat{A}]^W &\hookrightarrow C^\infty[\widehat{A}_u]^W \\
\mathcal H_J(G)/{\rm ker} &\hookrightarrow \mathcal C_J(G)/{\rm ker}.
\end{align*}

We will endow $\mathcal C_J(G)$ and its quotient $\mathcal C_J(G)/{\rm ker}$ with the Schwartz topology (see \cite{Sil79}, Chap. 4).   We will endow $C^\infty[\widehat{A}_u]$ with the $C^\infty$-topology.  Note that $\mathcal H_J(G)/{\rm ker}$ resp. $\mathbb C[\widehat{A}]^W$ is a dense subspace of $\mathcal C_J(G)/{\rm ker}$ resp. $C^\infty[\widehat{A}_u]^W$.

\subsubsection{The trace Paley-Wiener theorem for Schwartz functions}

The following is a consequence of a much more general result of Arthur \cite{Ar94}.

\begin{prop} \label{arthur_prop} \mbox{\rm (Arthur)} The map $f \mapsto \widehat{f}$ determines an open surjective homomorphism of topological vector spaces
$$
\mathcal C_J(G) \rightarrow C^\infty[\widehat{A}_u]^W.
$$
Hence $\mathcal C_J(G)/{\rm ker} ~ \cong ~ C^\infty[\widehat{A}_u]^W$ as topological vector spaces.  
\end{prop}

\subsubsection{The definition and Schwartz-continuity of $\widetilde{b}$}

We define $\widetilde{b}: \mathcal C_J(G(E))/{\rm ker} \rightarrow \mathcal C_J(G)/{\rm ker}$ to be the unique map making the following diagram commute:
$$
\xymatrix{
\mathcal C_J(G(E))/{\rm ker} \ar[r]^{f \mapsto \widehat{f}}_{\sim} \ar[d]_{\widetilde{b}} & 
C^\infty[\widehat{A^E}_u]^{W(E)} \ar[d]_{N} \\
\mathcal C_J(G)/{\rm ker} \ar[r]^{f \mapsto \widehat{f}}_{\sim} & C^\infty[\widehat{A}_u]^W.}
$$

It is clear that $\widetilde{b}$ is an extension of $\overline{b}$, in an obvious sense.  The following is immediate in light of Proposition \ref{arthur_prop}.

\begin{lemma} \label{widetilde_b_is_continuous}
The map $\widetilde{b}: \mathcal C_J(G(E))/{\rm ker} \rightarrow \mathcal C_J(G)/{\rm ker}$ is continuous with respect to the Schwartz topologies.
\end{lemma}

Using Lemma \ref{b_vs_overline_b} we deduce the following result, which was the goal of this subsection.

\begin{cor} \label{b_is_Schwartz_continuous}
The homomorphism $b: Z(\mathcal H_J(G(E))) \rightarrow Z(\mathcal H_J(G))$ is continuous, when each algebra $Z(\mathcal H_J)$ is given the Schwartz topology (the topology it inherits as a subspace of the appropriate $\mathcal C_J$).
\end{cor}

\section{Descent of twisted orbital integrals} \label{descent}

\subsection{On Galois cohomology of $\mathfrak k$-Levi subgroups}

Here we let $\mathfrak k$ denote any field, and $G$ any connected reductive group over $\mathfrak k$.  Recall that a $\mathfrak k$-Levi subgroup $M \subseteq G$ is a Levi factor of a $\mathfrak k$-rational parabolic subgroup of $G$.  It is easy to see that
\begin{equation} \label{H(M)->H(G)}
{\rm ker}[H^1(\mathfrak k, M) \rightarrow H^1(\mathfrak k, G)] = \{ 1 \}
\end{equation}
if $M$ is a $\mathfrak k$-Levi subgroup. 
In fact a twisting argument (cf. \cite{Ko86a}, 1.3) shows further that in that case,
\begin{equation*}
H^1(\mathfrak k, M) \rightarrow H^1(\mathfrak k,G)
\end{equation*}
is injective.

\subsection{Elements with non-elliptic norm}

Now we return to the notation of section \ref{preliminaries_section}.  Further, let $E/F$ be an unramified extension of degree $r$, and use the symbol $\sigma$ to denote both the Frobenius automorphism in ${\rm Aut}(L/F)$ and its restriction to ${\rm Gal}(E/F)$.  

The following result is needed for our reduction, via descent, of the fundamental lemma to semi-simple elements $\gamma \in G(F)$ which are elliptic.  It seems not to appear elsewhere, although in some sense it must be implicit in the work of Clozel \cite{Cl90} and Labesse \cite{Lab90}.  A proof is given here because no reference has come to light. \footnote{Clozel's related Lemma 2.12 in \cite{Cl90} concerns the special case where $\delta$ is $\sigma$-regular, but this is not enough for our purposes.  We also draw the reader's attention to the fact that Lemma \ref{non-elliptic_norms} is used implicitly in the descent step in the proof of \cite{Cl90}, Prop. 7.2.}

\begin{lemma} \label{non-elliptic_norms}
Suppose $\gamma \in G(F)$ is semi-simple and $M \subseteq G$ is an $F$-Levi subgroup with $G^\circ_\gamma \subseteq M$ (hence $G_\gamma^\circ = M^\circ_\gamma$ and $\gamma \in M(F)$).  Suppose $\delta \in G(E)$ has $\mathcal N\delta = \gamma$.  Then there exists an element $\delta_1 \in M(E)$ such that $\mathcal N_M(\delta_1) = \gamma$ and such that $\delta$ and $\delta_1$ are $\sigma$-conjugate under $G(E)$.

\end{lemma}

Here $\mathcal N_M$ denotes the norm function relative to the group $M$.

\begin{proof}
We will use the following notation: letting $\sigma$ denote an extension of $\sigma \in {\rm Gal}(E/F)$ to ${\rm Gal}(\overline{F}/F)$, we set $g^{\sigma} := \sigma(g)$ and $g^{-\sigma} := \sigma(g^{-1})$, for $g \in G(\overline{F})$.  

First we consider the case where $G_{\rm der} = G_{\rm sc}$.  In that case $G^\circ_\gamma = G_\gamma$ and stable $\sigma$-conjugacy (resp. stable conjugacy) is $\overline{F}$-$\sigma$-conjugacy (resp. $\overline{F}$-conjugacy), cf. \cite{Ko82}.  Choose any $h \in G(\overline{F})$ with $h (N\delta)h^{-1} = \gamma$.  Applying $\sigma$ to the equation $N\delta = h^{-1} \gamma h$ and using $(N\delta)^\sigma = \delta^{-1} (N\delta) \delta$, we see that $h \delta h^{-\sigma} \in G_\gamma \subseteq M$.

A similar calculation shows that $h h^{-\tau} \in G_\gamma \subseteq M$ for any $\tau \in {\rm Gal}(\overline{F}/E)$.  The map $\tau \mapsto h h^{-\tau}$ determines an element of ${\rm ker}(H^1(E,M) \rightarrow H^1(E,G))$, which is trivial by (\ref{H(M)->H(G)}) for $\mathfrak k = E$.  Writing $h h^{-\tau} = m^{-1}m^\tau$ for some $m \in M(\overline{F})$ and every $\tau \in {\rm Gal}(\overline{F}/E)$, we see that $mh \in G(E)$.  

Now $\delta$ is clearly $\sigma$-conjugate under $G(E)$ to $\delta_1 := (mh) \delta (mh)^{-\sigma} \in M(E)$.  Since $N\delta_1 = m\gamma m^{-1}$, we have $\mathcal N_M\delta_1 = \gamma$.

Let us now derive the general case of the lemma from the special case just considered.  Choose a $z$-extension $\alpha:G' \rightarrow G$ which is adapted to $E$ in the sense of \cite{Ko82}.  Let $Z = {\rm ker}(\alpha)$.  
Let $M' = \alpha^{-1}(M)$, an $F$-Levi subgroup of $G'$.  Choose elements $\gamma' \in G'(F)$ and  $\delta' \in G'(E)$ with $\alpha(\gamma') = \gamma$ and $\alpha(\delta') = \delta$.  We have $\gamma' \in M'(F)$, and $G'_{\gamma'} \subseteq M'$.   Further, $\mathcal N \delta = \gamma$ implies that $\mathcal N\delta' = \gamma' z$ for some $z \in Z(F)$; replacing $\gamma'$ with $\gamma'z$ we may assume $\mathcal N \delta' 
= \gamma'$.  
Since $G'_{\rm der}$ is simply-connected, we already know that $\delta'$ is $\sigma$-conjugate under $G'(E)$ to an element $\delta'_1 \in M'(E)$ for which $\mathcal N_{M'}(\delta'_1) = \gamma'$.  It follows that $\delta$ is 
$\sigma$-conjugate under $G(E)$ to  
$\delta_1 := \alpha(\delta'_1) \in M(E)$, and that $\mathcal N_M(\delta_1) = \gamma$ (use loc.~cit.~Lemma 5.6).
\end{proof}

\begin{Remark}
The lemma is stated under our standing hypothesis that $E/F$ is an unramified extension, with $\sigma$ denoting the Frobenius generator.  However, the same statement holds (with the same proof) when $E$ is any cyclic extension of our field $F$ and $\sigma$ is any generator of ${\rm Gal}(E/F)$. 
\end{Remark}

How will this result be used?  Fix a semi-simple element $\gamma \in G(F)$, and let $S$ denote the $F$-split component of the center of $G^\circ_\gamma$.  Let $M := {\rm Cent}_G(S)$, an $F$-Levi subgroup of $G$.  It is clear that $\gamma$ is an elliptic element in $M(F)$, and that $\gamma$ is elliptic in $G(F)$ if and only if $M =G$.

Now suppose that $\gamma$ is {\rm not} elliptic in $G$, so that $M \subsetneq G$.  The fundamental lemma is proved by induction on semi-simple rank (it being obvious for tori), and so we may assume it is already known for $M$.  Suppose we want to show $\phi$ and $b\phi$ are associated at $\gamma$.  If $\gamma$ is not a norm from $G(E)$, it is not a norm from $M(E)$, so that the vanishing of ${\rm O}^G_\gamma(b\phi)$ will follow by induction using the descent formula (\ref{eq:descent_over_F}) below.  Now suppose $\gamma$ is a norm from $G(E)$.  Then Lemma \ref{non-elliptic_norms} applies to show that $\gamma$ is a norm from $M(E)$.  The desired matching of stable (twisted) orbital integrals is then proved by descending to $M$, i.e. by using the equations (\ref{eq:descent_over_E}) and (\ref{eq:descent_over_F}) (and taking into account the crucial Lemma \ref{lemma_B} which allows us to compare these equations).

\subsection{Descent preliminaries} \label{descent_prelim_subsection}

We continue with the notation of the previous subsection.  Our standard Borel subgroup $B = TU$ determines a set of simple positive (relative) roots $\Delta_0$.  

As mentioned above, in descent we are given a proper $F$-Levi subgroup $M$ which is a Levi factor in an $F$-parabolic $P = MN$.  Since it is harmless to conjugate by elements of $G(F)$, we may assume $M$ and $P$ are {\em standard}, i.e. $M \supseteq T$ and $P \supseteq B$, i.e., $N \subseteq U$.

Let $U_M := U \cap M$ and $B_M = B \cap M$, so that $B_M$ is a Borel subgroup of $M$ with Levi decomposition $B_M = U_M T$.  The Levi $M$ corresponds to a subset $\Delta_M \subset \Delta_0$.  The relative Weyl group $_FW_M := N_M(T)(F)/T(F)$ is a Coxeter subgroup of $(_FW, \{ s_\alpha, \, \alpha \in \Delta_0 \})$ with generating set $\{ s_\alpha, \, \alpha \in \Delta_M \}$.  

For each element $w \in \, _FW$, we shall choose once and for all a representative in $K$, denoted by the same symbol.

Let $E/F$ be an unramified extension of degree $r$ contained in $L$, and fix a generator $\theta \in {\rm Gal}(E/F)$.

For $m \in M(F)$, we define following Harish-Chandra the functions $D_{G(F)/M(F)}$ and $\Delta_P$\footnote{We retain this standard notation for this function despite our very similar notation for sets of simple roots!} by
\begin{align*}
D_{G(F)/M(F)}(m) &= {\rm det}(1 - {\rm Ad}(m^{-1}); \, {\rm Lie}(G(F))/{\rm Lie}(M(F))) \\
\Delta_P(m) &= {\rm det}(1 - {\rm Ad}(m^{-1}); \, {\rm Lie}(N(F))).
\end{align*}

Let $\overline{N}$ denote the unipotent radical of the parabolic subgroup $\overline{P} \supset M$ which is opposite to $P$.  Using the decomposition ${\rm Lie}(G(F)) = {\rm Lie}(\overline{N}(F)) \oplus {\rm Lie}(M(F)) \oplus {\rm Lie}(N(F))$, one can prove the identity $|D_{G(F)/M(F)}|_F = |\Delta_P|_F^2 \delta_P$.

We have similar definitions for $D_{G(E)/M(E)}$, $\delta_{P(E)}$, etc. for $E$ replacing $F$.  For the time being, we will work over $E$, and use the symbols $G$, $P$, $J$, $W$, etc. in place of $G(E)$, $P(E)$, $J(E)$, $_EW$, etc.

The refined Iwasawa decomposition states that 
$$
G = \coprod_{w \in W^P} PwI,
$$
where $W^P$ denotes the set of minimal coset representatives for the elements of $W_M\backslash W$, with respect to the Bruhat order defined by $B$.  Since $I \subset J$, each $P,J$-double coset is a union of $P,I$-double cosets.  So each $P,J$-double coset is of the form $PwJ$ for an element $w \in W^P$ (which is not unique in general).  
Choose a set of such representatives, and denote it by $_EW(P,J)$.  Thus 
$$G = \coprod_{w \in \, _EW(P,J)} PwJ.$$  For technical purposes that will become clear later, we choose a lift for each $w$ in the hyperspecial compact subgroup $K(E) \supset I(E)$, and continue to denote that representative by the symbol $w$.

Fix once and for all Haar measures $dg$,$dj$ on $G,J$ respectively, such that ${\rm vol}_{dg}(J) = {\rm vol}_{dj}(J) = 1$.  On $P = MN$ we fix a (left) Haar measure $dp$ such that ${\rm vol}_{dp}(P \cap J) = 1$.  For any integrable smooth function $\phi$ on $G$ which is supported on $PwJ$, we have the following integration formula
\begin{equation} \label{eq:integration_on_PwJ}
\int_G \phi(g) \, dg = q^{-1}_{E,w} \int_{P} \int_J \phi(pwj) \, dj \, dp,
\end{equation}
where $q_{E,w} := {\rm vol}_{dp}(P \cap \, ^wJ)$, and $^wJ := wJw^{-1}$.  The two integrals are proportional because as distributions on $P \times J$, both are left $P$-invariant and right $J$-invariant.  
The proportionality can be computed by using a test function, for instance the characteristic function of $wJ$.

In the same manner, we have for any integrable smooth function $\phi$ on $G$, the formula
\begin{equation} \label{eq:integration_on_G}
\int_G \phi(g) \, dg = \sum_{w \in \, _EW(P,J)} q_{E,w}^{-1} \int_P \int_J \phi(pwj) \, dj \, dp.
\end{equation}

The following variant will be useful.  Let $dp_w$ (resp. $dm_w$, $dn_w$) denote the (left) Haar measure on $P$ (resp. $M$,$N$) such that ${\rm vol}_{dp_w}(P \cap \, ^wJ)$ (resp. ${\rm vol}_{dm_w}(M \cap \, ^wJ)$, ${\rm vol}_{dn_w}(N \cap \, ^wJ)$) has the value 1.  Then we have 
\begin{equation} \label{eq:normalized_integration_on_G}
\int_G \phi(g) \, dg = \sum_{w \in \, _EW(P,J)} \int_P \int_J \phi(pwj) \, dj \, dp_w.
\end{equation}

\subsection{Descent formulas} \label{descent_formulas_subsection}

Let us recall the set-up: $\gamma \in G(F)$ is a semi-simple element, $S$ denotes the $F$-split component of the center of $G_\gamma^\circ$, and $M = {\rm Cent}_G(S)$, an $F$-Levi subgroup.  We suppose $\gamma$ is not elliptic in $G(F)$, so that $M$ is a proper Levi subgroup of $G$.  Choose an $F$-parabolic $P = MN$ with $M$ as Levi factor.  As above, we may assume $P$ and $M$ are standard.

We have $\gamma \in M(F)$ and $G^\circ_\gamma = M^\circ_\gamma$.  We assume $\gamma = \mathcal N \delta$ for an element $\delta \in M(E)$ (see Lemma \ref{non-elliptic_norms}).

The twisted centralizer $G_{\delta \theta}$ of $\delta\theta$ is an inner form of $G_\gamma$ whose group of $F$-points is 
$$
G_{\delta\theta}(F) := \{ g \in G(E) ~ | ~ g^{-1}\delta\theta(g) = \delta \}.
$$
The inclusion $M^\circ_{\delta \theta} \subseteq G^\circ_{\delta \theta}$ is an equality for dimension reasons (they are inner forms of $M^\circ_\gamma = G^\circ_\gamma$). 

Since $G_{\delta\theta}$ and $G_\gamma$ are inner forms, 
we may choose compatible measures $dg_{\delta \theta}$ on $G_{\delta \theta}(F)$ and $dg_\gamma$ on $G_\gamma(F)$, and use the former to define 
\begin{equation} \label{TO_defn}
{\rm TO}_{\delta \theta}^G(\phi) := \int_{G^\circ_{\delta \theta}\backslash G(E)} \phi(g^{-1}\delta \theta g) d\bar{g}.
\end{equation}
Here, as in the sequel, we abbreviate by writing $G^\circ_{\delta \theta}$ in place of $G^\circ_{\delta \theta}(F)$.  Since $G^\circ_{\delta \theta} = M^\circ_{\delta \theta}$, we may choose the same measure $dg_{\delta \theta} = dm_{\delta \theta}$ in defining ${\rm TO}^M_{\delta \theta}(\psi)$ for an integrable smooth function $\psi$ on $M(E)$.

Using (\ref{eq:normalized_integration_on_G}) and the substitution $g = mnwj$ in each summand, we get for $\phi \in \mathcal H_J(G(E))$ the formula
\begin{align*}
{\rm TO}_{\delta \theta}(\phi) &= \int_{G^\circ_{\delta \theta}\backslash G} \phi(g^{-1}\delta \theta g) \, d\bar{g} \\
&= \sum_{w \in \, _EW(P,J)} \int_{M^\circ_{\delta \theta} \backslash M} \int_{N} \int_J \phi(j^{-1}w^{-1}n^{-1}m^{-1}\delta \theta m \theta n \theta(w) \theta j) \, dj \, dn_w \, d\bar{m}_w.
\end{align*}

Since $\phi$ is $J$-bi-invariant, we may suppress the integral over $J$.  For each $m$ we write $m_0 := m^{-1} 
\delta \theta m \in M(E)$.  Also, define the smooth function $^{w,\theta}\phi$ by the equality $^{w,\theta}\phi(g) = \phi(w^{-1}g\theta(w))$.  We easily see 
\begin{equation}\label{eq:m_0_before_Jacobian}
{\rm TO}_{\delta \theta}(\phi) = \sum_{w \in \, _EW(P,J)} \int_{M^\circ_{\delta \theta} \backslash M} \int_{N} \, ^{w,\theta}\phi(m_0(m_0^{-1}n^{-1}m_0 \theta n)) \, dn_w \, d\bar{m}_w.
\end{equation}

Consider the map of $F$-manifolds $N(E) \rightarrow N(E)$ given by $n \mapsto m_0^{-1}n^{-1}m_0 \theta n$.  The absolute value of the Jacobian of this transformation is a constant (independent of $m$, but depending on $\delta$)
\begin{align*}
|{\rm det}(\theta - {\rm Ad}(m_0^{-1}); \, {\rm Lie} \, N(E))|_F & = |{\rm det}(1 - {\rm Ad}(\mathcal Nm_0^{-1}); \, {\rm Lie} \,N(F))|_F \\
&= |D_{G(F)/M(F)}(\mathcal Nm_0)|^{1/2} \, \delta_{P(F)}(\mathcal Nm_0)^{-1/2} \\
&= |D_{G(F)/M(F)}(\mathcal N\delta)|^{1/2} \, \delta_{P(E)}(m_0)^{-1/2}.
\end{align*}
Here we regard $\theta - {\rm Ad}(m_0^{-1})$ as an $F$-linear endomorphism of ${\rm Lie} \,N(E)$.
These equalities follow from standard calculations, see e.g. \cite{Ko80}, $\S 8$. Thus by the change of variables formula for $F$-manifolds, we may write (\ref{eq:m_0_before_Jacobian}) as
\begin{equation} \label{eq:m_0_after_Jacobian}
{\rm TO}_{\delta \theta}(\phi) = |D_{G(F)/M(F)}(\mathcal N\delta)|^{-1/2}_F \, \sum_{w \in \, _EW(P,J)} \int_{M^\circ_{\delta \theta} \backslash M} \delta_{P(E)}^{1/2}(m_0)\int_N \, ^{w,\theta}\phi(m_0 n) \, dn_w \, d\bar{m}_w.
\end{equation}

Thus we have 

\begin{equation} \label{eq:descent_over_E}
{\rm TO}^{G(E)}_{\delta \theta}(\phi) = |D_{G(F)/M(F)}(\mathcal N\delta)|_F^{-1/2} \sum_{w \in \, _EW(P,J)} 
\, {\rm TO}^{M(E)}_{\delta \theta} ( (\, ^{w,\theta}\phi)^{(P(E))}),
\end{equation}
where it is understood that ${\rm TO}^{M(E)}$ is formed using $dm_w$ and $(\cdot)^{(P(E))}$ is formed using $dn_w$.

If we consider the special case $E = F$ ($\theta = 1$), and use the equality $\gamma = \mathcal N\delta = \delta$, we have a corresponding descent formula for functions in $\mathcal H_J(G(F))$.  We write it out in a special case: assume $\phi \in Z(\mathcal H_J(G(E)))$ and consider the descent formula for the function $b\phi \in Z(\mathcal H_J(G(F)))$:
\begin{equation} \label{eq:1st_descent_over_F}
{\rm O}^{G(F)}_\gamma (b\phi) = |D_{G(F)/M(F)}(\gamma)|_F^{-1/2} \sum_{w \in \, _FW(P,J)} \,
{\rm O}^{M(F)}_\gamma [(\,^wb\phi)^{(P(F))}],
\end{equation}
where for a function $\psi$ we define $^w\psi(x) := \psi(w^{-1}xw)$.  Using the compatibility of $b$ with conjugation by $w$ and constant term (Corollary \ref{b_constant_term_compatibility} and Lemma \ref{w_conjugation_compatibility}), this is 
\begin{equation} \label{eq:descent_over_F}
{\rm O}^{G(F)}_\gamma (b\phi) = |D_{G(F)/M(F)}(\gamma)|_F^{-1/2} \sum_{w \in \, _FW(P,J)} \,
{\rm O}^{M(F)}_\gamma b[(\,^w\phi)^{(P(E))}],
\end{equation}

\subsection{Lemmas needed to compare descent formulas}\label{lemmas_needed_to_compare_subsection}

Now, to compare (\ref{eq:descent_over_E}) with (\ref{eq:descent_over_F}), we need several lemmas.  Our first goal is to prove we may choose the sets of representatives $_FW(P,J)$ and $_EW(P,J)$ in such as way that $_FW(P,J) = \, _EW(P,J)^\theta$.

Recall that we are temporarily writing $W$ in place of the relative Weyl group $_EW$ over $E$.  We also sometimes denote the latter by $W(E)$, when we think of it as the ${\rm Aut}(L/E)$-invariants in the (absolute) Weyl group $_LW$.

\begin{defn} \label{W_M_W_min}
We define the subset $(W_M\backslash W)_{\rm min} \subset W$ to consist of the elements $w \in W$ which have minimal length in their cosets $W_Mw$.  Here, we use the Borel subgroup $B = B_0$ to define the Coxeter structure (and thus the notion of length) on $W$.
\end{defn}

\begin{lemma} \label{lemma_A}
\begin{itemize}
\item[(a)] If $w \in (W_M\backslash W)_{\rm min}$, then $^wI \cap M = I \cap M$.
\item [(b)] Let $p$ denote the projection of $\widetilde{W} = X_*(A) \rtimes W$ onto $W$.  Let $\overline{W}_J = p(\widetilde{W}_J)$.  
Then the canonical map $W \rightarrow P\backslash G/J$ induces a bijection
$$
W_M \backslash W/\overline{W}_J ~ \widetilde{\rightarrow} ~ P \backslash G/J.
$$
\item[(c)] $P(F)\backslash G(F)/J(F) = [P(E) \backslash G(E)/J(E)]^\theta$, and hence 
$$W_M(F)\backslash W(F)/\overline{W}_{J(F)} = [W_M(E) \backslash W(E) / \overline{W}_{J(E)}]^\theta.$$  
Thus, we may choose the sets $_FW(P,J)$ and $_EW(P,J)$ in such a way that 
$$_FW(P,J) = \, _EW(P,J)^\theta$$ 
and each set consists of elements $w$ which are minimal in their cosets $W_Mw$. 
\end{itemize}
\end{lemma}

Our second goal is to prove that the only summands in (\ref{eq:descent_over_E}) which are nonzero are those indexed by $w \in \, _FW(P,J)$. 

\begin{lemma} \label{lemma_B} The summands in (\ref{eq:descent_over_E}) indexed by $w \in \, _EW(P,J)$ with $\theta(w) \neq w$ are zero.
\end{lemma}

This lemma follows immediately using the twisted version of the Kazhdan density theorem (see \cite{KoRo}) and the following lemma.  To state this, recall that we say a (left) representation $\Pi$ of $G(E)$ is $\theta$-{\em stable} provided it extends to a (left) representation of the group $G^*(E) := G(E) \rtimes \langle \theta \rangle$.  Equivalently, there is an isomorphism of $G(E)$-modules $I_\theta : \Pi ~ \widetilde{\rightarrow}  ~ \Pi^\theta$, where $\Pi^{\theta}$ is the representation on the space of $\Pi$ where the $G(E)$-action is given by $\Pi^\theta(g) = 
\Pi(\theta(g))$.   If $\Pi$ is irreducible, the intertwiner $I_\theta$ is uniquely determined up to a non-zero scalar (Schur's lemma).  When in addition $[E:F] = r$, we normalize 
$I_\theta$ so that $I_\theta^r = {\rm id}$;  then $I_\theta$ is uniquely determined up to an $r$th root of unity.

\begin{lemma} \label{lemma_C} Let $\phi \in Z(\mathcal H_J(G(E)))$.  If $w \in (W_M(E)\backslash W(E))_{\rm min}$ and $\theta(w) \neq w$, then for every $\theta$-stable admissible representation $\sigma$ of $M(E)$, we have
$$\langle {\rm trace} ~ \sigma I_\theta, (\, ^{w,\theta}\phi)^{(P)} \rangle = 0.$$
\end{lemma}

The upshot of these lemmas is that we may effectively compare (\ref{eq:descent_over_E}) with (\ref{eq:descent_over_F}), and thereby by induction on semi-simple ranks, reduce the fundamental lemma to the case where $\gamma$ is an elliptic element.
\begin{Remark} \label{I_M_fixed_remark}
Suppose $\sigma$ is an admissible representation of $M$ such that $i^G_P(\sigma)$ has $I$-fixed vectors.  It follows using Lemma \ref{lemma_A}, part (a), that $\sigma$ has $I \cap M$-fixed vectors.  Indeed, suppose $0 \neq \Phi \in i^G_P(\sigma)^I$.  Then there exists $w \in (W_M\backslash W)_{\rm min}$ such that $\Phi(w) \neq 0$.  But $\Phi(w)$ is fixed by every operator $\sigma(i_M)$ for $i_M \in I \cap M$, since $\sigma(i_M)\Phi(w) = \Phi(w \cdot \, ^{w^{-1}}i_M)$ and $^{w^{-1}}i_M \in I$ by Lemma \ref{lemma_A}, part (a).
\end{Remark}

\begin{Remark} \label{w(J_cap_M)_vs_(J_cap_M)}
For our general parahoric $J$, and for $w$ a minimal element of its coset $W_Mw$, one might hope that the analogue of Lemma \ref{lemma_A} (a) holds true: $^wJ \cap M = J \cap M$.  However, this usually fails when $J \neq I$.  For instance, the inclusion $J \cap M \subset \, ^wJ$ fails for the group $G={\rm Sp}(4)$, where (numbering simple roots $\alpha_i$ and fundamental coweights $\omega^\vee_i$ as in \cite{Bou}), $M$ is the Levi corresponding to the simple root $\alpha_2$, $w = s_{\alpha_1}$, and $J$ is the parahoric subgroup fixing the vertex $\frac{1}{2}\omega^\vee_1$ in the base alcove ${\bf a}$).  In fact, $s_{\alpha_2} \in J \cap M$ but $s_{\alpha_2} \notin \, ^{s_{\alpha_1}}J$ since $s_{\alpha_2}$ does not fix the vertex $s_{\alpha_1}(\frac{1}{2}\omega^\vee_1)$.  
\end{Remark} 

\subsection{Proof of Lemma \ref{lemma_A}}

Part (a):  Since $I \cap M$ and $^wI \cap M$ are Iwahori subgroups of $M$ (Lemma \ref{J_cap_M_is_parahoric}), it is enough to prove $I \cap M \subset \, ^{w}I$.  Indeed, two Iwahori subgroups which both fix a given alcove must coincide (\cite{BT2}, 4.6.29).  For this proof, let ${\bf a}$ resp. ${\bf a}_M$ denote the alcove of the building for $G$ resp. $M$ corresponding to the Iwahori subgroup $I$ resp. $I \cap M$.  
Note that ${\bf a}_M$ is the union of ${\bf a}$ along with some other alcoves ${\bf a}'$.  Given $y \in I \cap M$, we want to show $y \in \, ^wI$, i.e., $y$ fixes $w{\bf a}$.  If $w{\bf a} \subset {\bf a}_M$, we are done.  If $w{\bf a} \nsubseteq {\bf a}_M$, there is a simple root $\alpha \in \Delta_M$ such that $w{\bf a}$ and ${\bf a}$ are separated by the hyperplane given by $\alpha  = 0$.  But then $s_\alpha w < w$ in the Bruhat order on $W$, a contradiction of $w \in (W_M\backslash W)_{\rm min}$.  

\smallskip

{\em Note: Let $B$ denote $B_0$ or $\overline{B}_0$.  Then the same proof shows $^wB \cap M = B \cap M$; this fact is used in subsection \ref{simplification_subsection}.}

\medskip

\noindent Part (b):  It is clear that the map is well-defined, and surjective, by the refined Iwasawa decomposition $G = \coprod_{w \in W^P}PwI$.  To show it is injective, we use our assumption that $P,M$ are standard, that is, $P \supseteq B$ and $M \supseteq A$.  Suppose $w_1, w_2 \in W$ represent two double cosets in $W_M\backslash W/\overline{W}_J$ such that $Pw_1 J = Pw_2 J$.  We may assume $w_1, w_2 \in W^P$.  Write $w_1j = pw_2$ with $j \in J$ and $p = nm = nu_M \varpi^\nu w_M i_M$ with $u_M \in U \cap M$, $\,\,w_M \in W_M$, $\,\, i_M \in I \cap M$, and $\nu \in X_*(A)$ (using the usual Iwasawa decomposition for $m \in M$).  By part (a), we have
$$
w_1j = nu_M \varpi^\nu w_M w_2 i
$$
for some $i \in I$.  

There is a sufficiently $B$-dominant cocharacter $\lambda \in X_*(A)$ such that $\varpi^\lambda (nu_M) \varpi^{-\lambda} \in I$.  Thus we have the identity
$$
\varpi^\lambda w_1 j = i' \varpi^{\lambda + \nu} w_M w_2 i,
$$ 
for some $i' \in I$.  It is easy to show that $\widetilde{W}_J \subset W_{\rm aff}$ (see the Claim in the proof of Lemma \ref{J_cap_M_is_parahoric}) and that $\widetilde{W}_J$ is generated by simple affine reflections which fix ${\mathbf a}_J$.  Now using $J = I\widetilde{W}_JI $ and the BN-pair relations, we obtain an element $w_J \in \widetilde{W}_J$ such that
$$
I \varpi^\lambda w_1 w_J I = I \varpi^{\lambda + \nu} w_M w_2 I,
$$
and hence by the uniqueness in the Bruhat-Tits decomposition $G = \coprod_{w \in \widetilde{W}} IwI$, we conclude that $\varpi^\lambda w_1 w_J = \varpi^{\lambda + \nu} w_M w_2$ and thus $w_1 p(w_J) = w_M w_2$.  

\medskip

\noindent Part (c): Let $X = P \backslash G$, an $F$-variety carrying the action of $J$ by right multiplications.  
We must show the canonical map $i: X(F) / J(F) \rightarrow [X(E)/J(E)]^\theta$ is bijective.  Extending $\theta$ to an automorphism $\theta \in {\rm Aut}(L/F)$ with fixed field $F$, it is enough to establish a canonical bijection
\begin{equation}
X(F)/J(F) ~ \widetilde{\rightarrow} ~ [X(L)/J(L)]^\theta
\end{equation}
(along with the analogue of this for $E$ replacing $F$).  The surjectivity follows from the fact that $H^1(\langle \theta \rangle, J(L)) = 1$, a consequence of Greenberg's theorem \cite{Gr} since $J(L)$ coincides with the $\mathcal O_L$-points of an $\mathcal O_F$-group scheme having connected geometric fibers.  To prove the injectivity, suppose $x_1,x_2 \in X(F)$ satisfy $x_1 j = x_2$ for $j \in J(L)$.  Thus $j\theta(j^{-1}) \in J_{x_1}(L)$, where $J_{x_1}$ denotes the stabilizer of $x_1$.  Representing $x_1$ by $w_1 \in \, _FW$, we see that $J_{x_1} = J \cap \, ^{w_1^{-1}}P = (J \cap \, ^{w_1^{-1}}M)(J \cap \, ^{w_1^{-1}}N)$.  Once again Greenberg's theorem implies that $H^1(\langle \theta \rangle, J_{x_1}(L)) = 1$ (see Lemma \ref{J_cap_M_is_parahoric}).  Hence $j \theta(j^{-1}) = j_1^{-1} \theta(j_1)$ for some $j_1 \in J_{x_1}(L)$.  But then $x_2 = x_1(j_1j)$ and so $x_1 = x_2$ modulo $J(F)$, proving the injectivity. 

This completes the proof of Lemma \ref{lemma_A}. \qed

\subsection{Proof of Lemma \ref{lemma_C}}

\subsubsection{Beginning of proof} \label{beginning_of_proof_lemma_C} 
 
Suppose $\sigma$ is a $\theta$-stable admissible representation of $M(E)$.  Write $I^M_\theta : \sigma ~ \widetilde{\rightarrow} ~ \sigma^\theta$ for the given intertwiner.  By \cite{Rog}, Lemma 2.1, it is enough to prove the lemma in the case where $\sigma$ is irreducible as an $M(E)$-representation, which we henceforth assume.  Without loss of generality, we may assume $\sigma$ has $I \cap M$-invariant vectors.  Then we may choose an unramified character $\xi$ of $T(E)$ such that $\sigma$ is a subrepresentation of $i^M_{B_M}(\xi)$.     

Given an element $\psi \in \mathcal H(M(E))$, we need to distinguish between its {\em left action} $L(\psi)$ and its {\em right action} $R(\psi)$ on the space $i^M_{B_M}(\xi)$ (and on $\sigma$).  
The right action $R(\psi)$ is given simply by $\Phi \mapsto \Phi * \psi$ (where $*$ is defined using a choice of Haar measure to be specified below).  The left action is defined by the operator identity $L(\psi) = R(\iota \psi)$, where $\iota$ is the anti-involution of the Hecke algebra defined by $\iota \psi(m) = \psi(m^{-1})$.  
We need to prove 
$${\rm trace}(L(\psi) \circ I^M_\theta; \, \sigma) = 0
$$ 
for all functions $\psi$ of the form $(\, ^{w,\theta}\phi)^{(P)}$, where $\phi$ is central and $w \neq \theta(w)$.  The notation on the left stands for ``the trace of the operator $L(\psi) \circ I^M_\theta$ on the space of $\sigma$'' (similar notation is used below in (\ref{right_action_trace_vanishes}) and (\ref{modified_right_action_trace_vanishes})).

We have the following relations
\begin{align} \label{relations}
\iota(\psi^{(P)}) &= (\iota \psi)^{(P)} \\
\iota(\, ^{w,\theta}\phi) &= \, ^{\theta(w), \theta^{-1}}\iota \phi, \notag
\end{align}
where $^{\theta(w), \theta^{-1}}\phi(g) := \phi(\theta(w)^{-1} \, g \, w)$.  
Since $\iota$ preserves centers of Hecke algebras, it will be enough to establish the corresponding identity with respect to the {\em right actions}, which is slightly more convenient.  That is, we need to show that
\begin{equation} \label{right_action_trace_vanishes}
{\rm trace}(R((\, ^{w,\theta}\phi)^{(P)}) \circ I^M_{\theta^{-1}}; \,\sigma)   = 0.
\end{equation}
By the identity 
\begin{equation} \label{F*deltaG}
F * (\delta H)(m) = \delta(m) \cdot (\delta^{-1}F * H)(m)
\end{equation}
for a character $\delta$ and functions $F,H$ on a group containing $m$, it is enough to prove that for all $\sigma$ we have
\begin{equation} \label{modified_right_action_trace_vanishes}
{\rm trace} (R(\delta_P^{1/2} (\, ^{w,\theta}\phi)^{(P)}) \circ I^M_{\theta^{-1}}; \, \sigma)  = 0.
\end{equation}
Note that we essentially replaced $\sigma$ with $\delta_P^{1/2} \sigma$ here.  This is legitimate since the latter is still $\theta$-stable, as $\delta_P \circ \theta = \delta_P$.

To prove (\ref{modified_right_action_trace_vanishes}), we need to relate the operator 
$R(\delta_P^{1/2}(\, ^{w,\theta}\phi)^{(P)}) \circ I^M_{\theta^{-1}}$ on $\sigma$ to a simpler operator $R(\theta(w) \, w^{-1}) \circ I^G_{\theta^{-1}}$ on the space $i^G_P(\delta_{P}^{-1/2}\sigma)$.  Let us define this operator.

We define an intertwiner
\begin{equation} \label{Ginterwiner_def}
I^G_{\theta} : i^G_P(\delta_P^{-1/2}\sigma) ~ \widetilde{\rightarrow} ~ i^G_P(\delta^{-1/2}_P \sigma)^\theta
\end{equation}
by setting $(I^G_\theta \widetilde{\Phi})(g) = I^M_{\theta}(\widetilde{\Phi}(\theta^{-1}g))$, for $\widetilde{\Phi} \in i^G_P(\delta_P^{-1/2}\sigma)$ and $g \in G(E)$.  We also define a vector-space isomorphism
$$
R(\theta(w) \, w^{-1}) : i^G_P(\delta_P^{-1/2} \sigma) ~ \widetilde{\rightarrow} ~ i^G_P(\delta_P^{-1/2}\sigma)
$$
by setting $(R(\theta(w) \, w^{-1})\widetilde{\Phi})(g) = \widetilde{\Phi}(g \theta(w) \, w^{-1})$, for 
$\widetilde{\Phi} \in i^G_P(\delta_P^{-1/2}\sigma)$ and $g \in G(E)$.

Next we define parahoric subgroups $J' := \, ^{\theta(w)}J$ in $G(E)$ and $J'_M := J' \cap M$ in $M(E)$, 
respectively.  Note that $I^G_{\theta^{-1}}$ takes $J'$-invariants into ${^wJ}$-invariants, and 
$R(\theta(w) \, w^{-1})$ sends the latter back into the former (similar statements hold for $I^M_{\theta^{-1}}$ resp. 
$R(\delta^{1/2}_P(\,^{w,\theta}\phi)^{(P)})$).  
Note also that $R(\delta_P^{1/2}(\, ^{w,\theta}\phi)^{(P)}) \circ I^M_{\theta^{-1}}$ takes $\sigma^{J'_M}$ into itself and vanishes identically on the natural complement of $\sigma^{J'_M}$ in $\sigma$ consisting of those vectors whose integral over $J'_M$ vanishes. 

Consider the natural surjective map
\begin{align} 
\mathcal P_{J'} ~: ~& i^G_P(\delta_P^{-1/2}\sigma)^{J'} \twoheadrightarrow \sigma^{J'_M} \label{first_projection} \\
& \hspace{.33 in} \widetilde{\Phi} \hspace{.2in} \mapsto \hspace{.2in} \widetilde{\Phi}(1). \notag
\end{align}
There is a canonical section of $\mathcal P_{J'}$
\begin{align*} \mathfrak s_{J'} ~ : ~& \sigma^{J'_M} \hookrightarrow i^G_P(\delta_P^{-1/2}\sigma)^{J'} \\
     & \hspace{.09in} v \hspace{.2in} \mapsto \hspace{.2in} \widetilde{\Phi}^{(v)}_1,
\end{align*}
where $\widetilde{\Phi}^{(v)}_1$ is the unique element supported on $PJ'$ such that 
$\widetilde{\Phi}^{(v)}_1(1) = v$. We will suppress the subscript $J'$ and write $\mathcal P$ resp. $\mathfrak s$ for these operators (even though $J'$ will vary in the sequel).  Note that 
\begin{align} \label{mathcalP_is_theta_equivariant}
\mathcal P \circ I^G_{\theta^{-1}} &= I^M_{\theta^{-1}} \circ \mathcal P \\
         \mathfrak s \circ I^M_{\theta^{-1}} &= I^G_{\theta^{-1}} \circ \mathfrak s. \notag
\end{align}

Let $\xi_1 := \delta_P^{-1/2}\xi$, an auxiliary character on $T(E)$ attached to $\xi$.  The key point in proving (\ref{modified_right_action_trace_vanishes}) turns out to be the following lemma.  

\begin{lemma} \label{key_calc}  Let the right action $R(\psi) = -*\psi$ on $i^M_{B_M}(\xi)$ (and $\sigma$) be defined using the Haar measure on $M(E)$ which gives $^wJ \cap M$ volume 1.  Then 
we have the following equality of linear functions $\sigma^{J'_M} \rightarrow \sigma^{J'_M}$
\begin{equation} \label{eq:key_calc}
ch_{\xi_1}(\, ^w\phi, \, ^wJ) \cdot \mathcal P \circ R(\theta(w) \, w^{-1}) \circ I^G_{\theta^{-1}} \circ \mathfrak s = 
R(\delta_P^{1/2}(\, ^{w,\theta}\phi)^{(P)}) \circ I^M_{\theta^{-1}}.
\end{equation}
\end{lemma}

Before proving this, let us use it to complete the proof of Lemma \ref{lemma_C}.  Suppose that (\ref{modified_right_action_trace_vanishes}) does not hold.  Fix a basis $v_1, \dots, v_n$ for $\sigma^{J'_M}$.  Then for some index $i$, the element 
$$
R(\delta_P^{1/2}(\,^{w,\theta}\phi)^{(P)}) \circ I^M_{\theta^{-1}}(v_i)
$$
has a non-zero $v_i$-component.  By Lemma \ref{key_calc}
\begin{equation} \label{R()Phi}
R(\theta(w) \, w^{-1}) \circ I^G_{\theta^{-1}}(\widetilde{\Phi}^{(v_i)}_1)
\end{equation}
has a non-zero $\widetilde{\Phi}^{(v_i)}_1$-component.  But this implies that the support of (\ref{R()Phi}) meets $PJ'$. (Use the fact that ${\rm ker}(\mathcal P)$ is spanned by the functions 
$\widetilde{\Phi}_\tau \in i^G_P(\delta_P^{-1/2}\sigma)^{J'}$ supported on sets $P\tau J'$, where $\tau \in \,_EW(P,J')$ and $\tau \neq 1$.)  This in turn means that the support of 
$I^G_{\theta^{-1}}\widetilde{\Phi}^{(v_i)}_1$ meets $P\theta(w) J w^{-1}$.  On the other hand, it is clear that
$$
{\rm supp}(I^G_{\theta^{-1}} \Phi^{(v_i)}_1) = P w J w^{-1}.
$$
Thus $P \theta(w) J = P w J$.  Since $w$ and $\theta(w)$ are each in $_EW(P,J)$, we must have $\theta(w) = w$, a contradiction.  This completes the proof of Lemma \ref{lemma_C}, modulo Lemma \ref{key_calc}.

\subsubsection{Proof of Lemma \ref{key_calc}}  In light of (\ref{mathcalP_is_theta_equivariant}), it is enough to prove the equality
\begin{equation} \label{modified_key_calc}
ch_{\xi_1}(\, ^w\phi, \, ^wJ) \cdot \mathcal P \circ R(\theta(w) \, w^{-1}) \circ \mathfrak s =  
R(\delta_P^{1/2}(\, ^{w,\theta}\phi)^{(P)}) 
\end{equation}
of linear functions $\sigma^{^wJ \cap M} \rightarrow \sigma^{J'_M}$.

Now the inclusion $\sigma \hookrightarrow i^M_{B_M}(\xi)$ induces an inclusion $i^G_P(\delta_P^{-1/2}\sigma) \hookrightarrow  i^G_P(\delta_P^{-1/2}i^M_{B_M}(\xi)) \cong i^G_B(\delta^{-1/2}_P \xi)$.  The following diagram commutes
\begin{equation} \label{P_s_diagram}
\xymatrix{
i^G_P(\delta_P^{-1/2}\sigma) \ar[d]_{\mathcal P} \ar@{^{(}->}[r] & i^G_B(\xi_1) \ar[d]_{\mathcal P'} \\
\sigma \ar@/_/[u]_{\mathfrak s} \ar@{^{(}->}[r] & i^M_{B_M}(\xi) \ar@/_/[u]_{\mathfrak s'},}
\end{equation}
where $\mathcal P'$ is the operator $\widetilde{\Phi} \mapsto \widetilde{\Phi}|_{M(E)}$.  Let us make explicit $\mathfrak s' = \mathfrak s'_{^wJ}$: given $\Phi \in i^M_{B_M}(\xi)^{^wJ \cap M}$, ${\mathfrak s}' \Phi =: \widetilde{\Phi}$ is the unique element in $i^G_B(\xi_1)^{^wJ}$ which is supported on $P\, ^wJ$ and which satisfies $\widetilde{\Phi}|_{M(E)} = \Phi$.

It is therefore enough to prove (\ref{modified_key_calc}) for $\mathcal P'$ resp. $\mathfrak s'$ in place of $\mathcal P$ resp. $\mathfrak s$.  Thus, Lemma \ref{key_calc} will follow from the next lemma.

\begin{lemma} \label{main_calc} Define the right action $*$ on $i^M_{B_M}(\xi)$ using the measure which gives $^wJ \cap M$ volume 1.  Let $\Phi \in i^M_{B_M}(\xi)^{^w J \cap M}$.  Let $\widetilde{\Phi} = \mathfrak s' \Phi \in i^G_B(\xi_1)^{^wJ}$, as defined above.  Then for any $y \in M(E)$, we have the identity
\begin{equation} \label{eq:main_calc}
ch_{\xi_1}(\, ^w\phi, \, ^wJ) \cdot \, \widetilde{\Phi} (y \theta(w) w^{-1}) = \Phi * [\delta_P^{1/2} (\, ^{w,\theta}\phi)^{(P)}](y).
\end{equation}
\end{lemma}

\begin{proof}
We have 
\begin{align*}
\Phi * [\delta^{1/2}_P (\, ^{w,\theta}\phi)^{(P)}] (y) &= \int_{M(E)} \Phi(m) 
\delta_P^{1/2}(m^{-1}y)(\, ^{w,\theta}\phi)^{(P)}(m^{-1}y) dm_w \\
&= \int_{M(E)} \int_{N(E)} \widetilde{\Phi}(m n) \, ^{w,\theta}\phi((mn)^{-1}y) \, dn_w \, dm_w \\
&= \int_{M(E)} \int_{N(E)} \int_{J} \widetilde{\Phi}(mn wjw^{-1}) \, ^{w,\theta}\phi((mn wjw^{-1})^{-1}y) \, dj \, dn_w \, dm_w,
\end{align*}
 where ${\rm vol}_{dj}(J) = 1$.  
On the other hand, since ${\rm supp}(\widetilde{\Phi}) \subseteq MN \, ^{w}J$, we can write this via the substitution $x = mnwjw^{-1}$ as
\begin{equation*}
\int_{G(E)} \widetilde{\Phi}(x) \, ^{w,\theta}\phi(x^{-1}y) \, dx = \widetilde{\Phi} * \, (^w\phi) (y \theta(w) w^{-1}).
\end{equation*}
 The result now follows from the fact that $^w\phi \in Z(\mathcal H_{\, ^wJ}(G(E)))$.  
We have proved Lemma \ref{main_calc}, and thus also Lemma \ref{key_calc}.  Thus, we have proved Lemma \ref{lemma_C}.
\end{proof}

\subsection{Proof of Proposition \ref{B_and_constant_term_compatibility}} \label{proof_of_B_constant_term_compatibility}
Consider (\ref{eq:main_calc}) in the special case $w = 1$ and $E = F$.  Assume $\phi \in Z(\mathcal H_J(G(F)))$.  In light of (\ref{F*deltaG}), equation (\ref{eq:main_calc}) asserts that the function $\phi^{(P)} := \phi^{(P(F))}$ acts 
on the right on $i^M_{B_M}(\xi_1)^{J_M}$ by the scalar $ch_{\xi_1}(\phi,J) := \xi_1^{-1}(B^{-1}\phi)$.  Since this holds for every unramified character $\xi_1$, the function $\phi^{(P)}$ necessarily belongs to the center $Z(\mathcal H_{J \cap M}(M(F)))$.  Furthermore, we have
$$
B^{-1}(\phi^{(P)}) = B^{-1}(\phi)
$$
as elements of $R^{W_M(F)}$.  In other words, the diagram in Proposition \ref{B_and_constant_term_compatibility} commutes.  \qed

\section{Various reductions} \label{reductions_section}

In this section we reduce the fundamental lemma to the special case where $G = G_{\rm ad}$ and $\gamma$ is a strongly regular elliptic element, which is a norm.  (Recall that we call a regular semi-simple element $\gamma \in G$ {\em strongly regular} if $G_\gamma$ is a torus.)   Further, we show that it is enough to prove the required matching for a dense subset of such elements $\gamma$.  We proceed using a series of steps, most of which are very similar to the corresponding reductions in the case of spherical Hecke algebras.  We will summarize these steps, explaining in some detail those which are substantially different from the spherical case.  No claim of originality is made in this section.

\subsection{Definition of stable twisted orbital integral (Review)}  As usual $G$ denotes an unramified connected reductive $F$-group.  Let $\delta \in G(E)$ be such that $\mathcal N \delta \in G(F)$ is semi-simple.  Recall that $G^\circ_{\delta \theta}$ denotes the identity component of the $F$-group $G_{\delta \theta}$ (the latter is denoted by $I_{s\delta}$ in \cite{Ko82}).  Let $e(\delta) := e(G^\circ_{\delta \theta})$ denote the sign attached by Kottwitz \cite{Ko83} to the connected reductive $F$-group $G^\circ_{\delta \theta}$.  Further, define $a(\delta)$ to be the cardinality of the set
\begin{equation} \label{a(delta)}
{\rm ker}[H^1(F, G^\circ_{\delta \theta}) \rightarrow H^1(F,G_{\delta \theta})].
\end{equation}
For any function $\phi \in C^\infty_c(G(E))$, we define its stable twisted orbital integral by
\begin{equation} \label{STO_defn}
{\rm SO}_{\delta \theta}(\phi) = \sum_{\delta'} e(\delta') \, a(\delta') \, {\rm TO}_{\delta' \theta}(\phi),
\end{equation}
where ${\rm TO}_{\delta'\theta}(\phi)$ is defined as in (\ref{TO_defn}).  
Here $\delta'$ ranges over $\theta$-conjugacy classes in $G(E)$ which are stably $\theta$-conjugate to $\delta$ (cf. \cite{Ko82}).  The definition of 
$
{\rm SO}_\gamma (f)
$
for a semi-simple element $\gamma \in G(F)$ and $f \in C^\infty_c(G(F))$ can be recovered from the special case $E=F$ (and so $\theta = {\rm id}$).

\subsection{The case where $\gamma$ is not a norm}

As is now well-known, the vanishing statement when $\gamma$ is not a norm was justified incorrectly in \cite{Cl90}, using a global argument.  Labesse later gave a correct, purely local, argument in \cite{Lab99}, Prop. 3.7.2.  Here, we simply adapt Labesse's reasoning to our setting of parahoric Hecke algebras.  

Fix $\phi \in Z(\mathcal H_J(G(E)))$.  If $\gamma$ is not a norm, then we must show that ${\rm SO}_\gamma(b\phi) = 0$.  First, assume $\gamma$ is elliptic.  Then Labesse's arguments from loc.~cit. Prop. 2.5.3, Lemme 3.7.1, and Prop. 3.7.2 apply to our situation (replacing $\mathcal H_K(G(E))$ with $Z(\mathcal H_J(G(E)))$ throughout) to show the following stronger vanishing statement: if $\gamma' \in G(F)$ is stably conjugate to $\gamma$, then $b\phi(g^{-1}\gamma'g) = 0$ for every $g \in G(F)$; thus obviously ${\rm O}_{\gamma'}(b\phi) = 0$.  

Now if $\gamma$ is non-elliptic, we can use our descent formula (\ref{eq:descent_over_F}) to deduce that ${\rm O}_{\gamma'}(b\phi)$ still vanishes for any $\gamma' \in G(F)$ which is stably conjugate to $\gamma$. 

Thus, in the reductions to follow, we may assume whenever convenient that $\gamma$ is a norm.

\subsection{Lemmas needed to pass between $G$ and certain $z$-extensions of $G$}

Here we follow closely the ideas in \cite{Cl90}, $\S 6.1$ and \cite{Ko86b}, $\S 4$, but we arrange things so that it is clear everything works in our setting.

Choose a finite unramified extension $F' \supset F$ which contains $E$ and splits $G$.  Consider a $z$-extension of $F$-groups
$$
\xymatrix{
1 \ar[r] & Z \ar[r] & H \ar[r]^p & G \ar[r] & 1,}
$$
where $Z$ is a finite product of copies of $R_{F'/F}{\mathbb G}_m$.   As usual, we are assuming $H_{\rm der} = H_{\rm sc}$.  Recall that $p$ is surjective on $E$- and $F$-points.  Choose an extension of $\theta$ to an element, still denoted $\theta$, in ${\rm Gal}(F'/F)$.

Let $Z(E)_1 := Z(E) \cap Z(L)_1$, the maximal compact subgroup of $Z(E)$.   Endow $Z(E)$ (resp. $Z(F)$) with the Haar measure giving $Z(E)_1$ (resp. $Z(F)_1$) volume 1.  The norm homomorphism $N: Z(E) \rightarrow Z(F)$ is surjective and determines a measure-preserving isomorphism
$$
N: \overline{Z(E)} \rightarrow Z(F)
$$
where $\overline{Z(E)} := Z(E)/(1-\theta)(Z(E))$.  The surjectivity of $N: Z(E) \rightarrow Z(F)$ follows from the proof of \cite{Ko82}, Lemma 5.6 (our $z$-extension is {\em adapted to $E$} in the terminology of loc.~cit.).  Also, here we give the compact subgroup $(1-\theta)(Z(E)) = (1-\theta)(Z(E)_1)$ the measure with total volume 1, and in the sequel we use this and our chosen measure on $Z(E)$ to determine the quotient measure on $\overline{Z(E)}$.   Of course $N$ sends the maximal compact subgroup $Z(E)_1/(1-\theta)(Z(E))$ of $\overline{Z(E)}$ isomorphically onto the maximal compact subgroup $Z(F)_1$ of $Z(F)$, and hence $N : Z(E)_1 \twoheadrightarrow Z(F)_1$ is surjective.   

Let $\chi: Z(F) \rightarrow \mathbb C^\times$ denote an arbitrary smooth character.  For $f \in C^\infty_c(H(F))$ define the locally constant function $f_\chi$ on $h \in H(F)$ with compact support modulo $Z(F)$ by
$$
f_\chi(h) = \int_{Z(F)} f(hz) \, \chi^{-1}(z) \, dz.
$$
For $\phi \in C^\infty_c(H(E))$ define $\phi_{\chi N}$ analogously (here $\chi N = \chi \circ N$ is a smooth character on $Z(E)$).  The (twisted) orbital integrals of $f_\chi$ (resp. $\phi_{\chi N}$) exist at ($\theta$)-semi-simple elements.

In the following lemma, write $\chi = 1$ for the trivial character , and let $\overline{\phi}$ resp. $\overline{f}$ denote the function $\phi_1$ resp. $f_1$ when it is viewed as an element in $C^\infty_c(G(E))$ resp. $C^\infty_c(G(F))$. 

\begin{lemma} \label{phi_chi} Let $\phi \in C^\infty_c(H(E))$ and $f \in C^\infty_c(H(F))$.  Then:
\begin{enumerate}
\item[(i)] $(\phi,f)$ are associated if and only if $(\phi_{\chi N},f_\chi)$ are associated for every $\chi$;
\item[(ii)] suppose $\phi$ resp. $f$ are bi-invariant under compact open subgroups $\widetilde{K}$ resp. $K$ which satisfy $N(\widetilde{K} \cap Z(E)) = K \cap Z(F)$ and $\widetilde{K} \supset (1-\theta)(Z(E))$; then in (i) we need consider only characters $\chi$ which are trivial on $K \cap Z(F)$; 
\item[(iii)] let $J$ denote a parahoric subgroup of $H(F)$; if $\phi \in \mathcal H_J(H(E))$ resp. $f \in \mathcal H_J(H)$, then in (i) we need consider only unramified characters $\chi$ on $Z(F)$; 
\item[(iv)] $(\phi_1,f_1)$ are associated if and only if $(\overline{\phi}, \overline{f})$ are 
associated.

\end{enumerate}
\end{lemma}

\begin{proof}

 We change notation and suppose $\delta \in H(E)$ (resp. $\gamma \in H(F)$) and write $\overline{\delta} := p(\delta)$ (resp. $\overline{\gamma} := p(\gamma)$).   We have the formulas holding for all $\chi$:
\begin{align} 
\int_{\overline{Z(E)}} \chi^{-1}(Nv) ~ {\rm SO}^H_{v\delta \theta}(\phi) \, dv &= {\rm SO}^H_{\delta\theta}(\phi_{\chi N})  \label{integrate_TO} \\
\int_{Z(F)} \chi^{-1}(z) ~ {\rm SO}^H_{z \gamma}(f) \, dz &= {\rm SO}^H_{\gamma}(f_{\chi}) \label{integrate_O}.
\end{align}
These imply (i) and (ii).  For (ii), the point of the hypotheses on $\widetilde{K}$ and $K$ is to ensure that $N$ induces an isomorphism from $\overline{Z(E)}$ modulo $\widetilde{K} \cap Z(E)/ (1-\theta)(Z(E))$ onto $Z(F)$ modulo $K \cap Z(F)$.

Part (iii) follows from  (ii), by taking $\widetilde{K} = J(E)$ and $K = J$ and noting that $J(E) \cap Z(E)$ resp. $J \cap Z(F)$ is the maximal compact subgroup of $Z(E)$ resp. $Z(F)$.

For $\chi = 1$ we have for $\gamma = \mathcal N \delta$
\begin{align} 
{\rm SO}^H_{\delta \theta}(\phi_1) & = \, {\rm SO}^G_{\overline{\delta}\theta}(\overline{\phi}) \label{integrate_SO_chi=1} \\
{\rm SO}^H_{\gamma}(f_1) &= \, {\rm SO}^G_{\overline{\gamma}}(\overline{f}) \label{TO_chi=1}. 
\end{align}
Here we have used $H_{\gamma}(F) = p^{-1}(G^\circ_{\overline{\gamma}}(F))$ and $Z(E)H_{\delta\theta}(F) = p^{-1}(G^\circ_{\overline{\delta}\theta}(F))$.  These follow from the surjectivity of $H_\gamma(F) \rightarrow G^\circ_{\overline{\gamma}}(F)$ and $H_{\delta \theta}(F) \rightarrow G^\circ_{\overline{\delta} \theta}(F)$.  Part (iv) follows.
\end{proof}


\begin{remark}
Recall that the factors $a(\delta')$ resp. $a(\gamma')$ in (\ref{STO_defn}) are needed precisely to make 
the relation (\ref{integrate_SO_chi=1}) resp. (\ref{TO_chi=1}) hold. Indeed, $p$ induces a surjective map from the set
$$
\{ \mbox{$\theta$-conjugacy classes $\delta' \in H(E)$ stably $\theta$-conjugate to $\delta$} \}
$$
onto the set
$$
\{ \mbox{$\theta$-conjugacy classes $\overline{\delta}' \in G(E)$ stably $\theta$-conjugate to $\overline{\delta}$} \},
$$ 
and the fiber over the class of $\overline{\delta}'$ can be identified with the set
$$
{\rm ker}[H^1(F,G^\circ_{\overline{\delta}'\theta}) \rightarrow H^1(F,G_{\overline{\delta}'\theta})].$$  
Denoting $\widetilde{G} := {\rm Res}_{E/F}(G_E)$, the key point here is that $p$ induces a bijection
$${\rm ker}[H^1(F,H_{\delta\theta}) \rightarrow H^1(F,\widetilde{H})] ~\widetilde{\longrightarrow} ~ {\rm ker}[H^1(F,G^\circ_{\overline{\delta}\theta}) \rightarrow H^1(F,\widetilde{G})],$$ 
and moreover the set of $\theta$-conjugacy classes which are stably $\theta$-conjugate to $\overline{\delta}$ corresponds to the image of 
$${\rm ker}[H^1(F,G^\circ_{\overline{\delta}\theta}) \rightarrow H^1(F,\widetilde{G})]$$
in 
$$
{\rm ker}[H^1(F,G_{\overline{\delta}\theta}) \rightarrow H^1(F,\widetilde{G})].
$$
   See \cite{Ko82}, especially p.~804.
\end{remark}

\begin{lemma} \label{P_Z} Let $J$ denote the parahoric subgroups in $H$ and $G$ corresponding to the facet ${\bf a}_J$. 
The map $\phi \mapsto \overline{\phi}$ determines a surjective homomorphism $Z(\mathcal H_J(H(E))) \rightarrow Z(\mathcal H_J(G(E)))$.  This is compatible with the base change homomorphisms, in the sense that
\begin{equation} \label{b_vs_P_Z}
b(\overline{\phi}) = \overline{b\phi}.
\end{equation}
\end{lemma}

\begin{proof}
For $g \in G(E)$ we write $\overline{\phi}(g) = \mathcal P_{Z}\phi(g): = \int_{Z(E)} \phi(hz) \, dz$, where $h \in H(E)$ is any element with $p(h) = g$.  By considering the actions on unramified principal series, one checks that if $\phi \in Z(\mathcal H_J(H(E)))$, then $\mathcal P_{Z}\phi \in Z( \mathcal H_J(G(E)))$ and the following formula holds (relative to either $E$ or $F$):
$$
B^{-1}(\mathcal P_{Z}\phi) = p(B^{-1}\phi).
$$ 
On the right hand side, $p$ denotes the homomorphism 
$\mathbb C[X_*(A^{E}_H)] \rightarrow \mathbb C[X_*(A^E)]$ induced by $p:H \rightarrow G$, where $A^{E}_H$ denotes a maximal $E$-split torus in $H$ whose image under $p$ is $A^E$. (There is an obvious variant of all this for $E = F$.)  This easily implies both the surjectivity of $\phi \mapsto \mathcal P_Z \phi$ and the formula (\ref{b_vs_P_Z}).
\end{proof}

\begin{lemma}\label{adjoint_reduction}  Assume $Z = Z(H)$.  Suppose that $(\overline{\phi}, b\overline{\phi})$ are associated for every $\phi \in Z(\mathcal H_J(H(E)))$.  Then $(\phi, b\phi)$ are associated for 
every $\phi \in Z(\mathcal H_J(H(E)))$.
\end{lemma}

\begin{proof}
We follow the argument of \cite{Cl90}, bottom of p.~284.  Let $\pi: H \rightarrow H/H_{\rm der} =: D$ denote the projection of $H$ onto its cocenter.  The map $\pi: Z(F) \rightarrow D(F)$ is an isogeny, and thus induces a surjection $\eta \mapsto \eta\pi$ from unramified characters of $D(F)$ to those of $Z(F)$.  Note that this surjection is not necessarily a bijection. 

Now suppose $\chi$ is an unramified character on $Z(F)$.  Let $\mathcal H_{J,\chi N}$ denote the algebra of $J(E)$-bi-invariant functions on $H(E)$ which are compactly supported modulo $Z(E)$ and which transform by $\chi N$ under $Z(E)$.  The product of $\phi_1, \phi_2 \in \mathcal H_{J,\chi N}$ is a function whose value at $h \in H(E)$ is given by
$$
\phi_1 \ast \phi_2 (h) = \int_{H(E)/Z(E)} \phi_1(y) \, \phi_2(y^{-1} h) ~ d\overline{y}
$$
where $d\overline{y}$ is the Haar measure on $G(E) = H(E)/Z(E)$ compatible with the usual one on $H(E)$ determined by $J(E)$ and that on $Z(E)$ fixed above.  

Write $\chi = \eta \pi$ for some unramified character $\eta$ on $D(F)$.  
For $\phi \in C^\infty(H(E))$, define a function $\phi \otimes \eta N$ by 
$$
(\phi \otimes \eta N) (h) = \phi(h) \, \eta(N\pi(h)), \,\, h \in H(E).
$$
Then the map $\phi \mapsto \phi \otimes \eta N $ determines an algebra isomorphism $\mathcal H_{J,1N} ~ \widetilde{\rightarrow} ~ \mathcal H_{J,\chi N}$.  Of course over $F$ we have similar definitions and statements, but we can drop the $N$ from the notation.

We have the following three equalities:
\begin{align}
{\rm SO}^H_{\delta \theta}(\psi \otimes \eta N) &= \eta\pi(N\delta)~ {\rm SO}^H_{\delta \theta}(\psi), \,~ ~  \,\,\,\,\, \psi \in  C^\infty(H(E)) \label{SO_phi_eta} \\
(\phi \otimes \eta N)_1 & = \, \phi_{\chi^{-1}N} \otimes \eta N, \,\, ~ \,\,\,\,\,\, \,\,\,\,\, \,\, \phi \in \mathcal H_J(H(E)) \label{phi_eta_1} \\
b(\phi \otimes \eta N) & = b\phi \otimes \eta, \,\, ~ \,\,\,\,\,\,\,\,\,\,\,\, \,\,\,\,\,\,\,\,\, \,\,\,\,\,\,\,\, \phi \in Z(\mathcal H_J(H(E))). \label{b(phi_eta)}
\end{align}

The first two equalities are immediate (for the first, we are assuming $\psi$ is such that the twisted orbital integrals of $\psi$ exist!).  For the third, assume $\phi \in Z(\mathcal H_J(H(E)))$.  
Using (\ref{F*deltaG}) we can compute the action of $\phi \otimes \eta N$ on unramified principal series $i^{H(E)}_{B(E)}(\xi)^{J(E)}$.  We find it acts by the scalar by which $\phi$ acts on $i^{H(E)}_{B(E)}(\xi \cdot \eta^{-1}N\pi)^{J(E)}$.  Thus, $\phi \otimes \eta N$ belongs to $Z(\mathcal H_J(H(E)))$.  Now suppose $\xi$ is an unramified character of $T(F) \subset H(F)$.  Then $b(\phi \otimes \eta N)$ acts on $i^{H(F)}_{B(F)}(\xi)^{J}$ by the scalar by which $\phi \otimes \eta N$ acts on $i^{H(E)}_{B(E)}(\xi N)^{J(E)}$, in other words the scalar by which $\phi$ acts on $i^{H(E)}_{B(E)}(\xi N \cdot \eta^{-1}N \pi)^{J(E)}$.  On the other hand, that is the scalar by which $b\phi \otimes \eta$ acts on $i^{H(F)}_{B(F)}(\xi)^{J}$.  This proves the third equality.

\smallskip

Now we can prove the lemma.  By Lemmas \ref{P_Z} and \ref{phi_chi} (iv), our hypothesis is that $(\phi_1, (b\phi)_1)$ are associated for all $\phi \in Z(\mathcal H_J(H(E)))$.  Now fix such a $\phi$.  Then for all $\eta$, $(\phi \otimes \eta N)_1$ and $b(\phi \otimes \eta N)_1 =  (b\phi \otimes \eta)_1$ are associated (cf. (\ref{b(phi_eta)})).  Then using (\ref{phi_eta_1}) and (\ref{SO_phi_eta}) we deduce that $(\phi_{\eta^{-1}\pi N}, (b\phi)_{\eta^{-1} \pi})$ are associated.  Since this holds for all $\eta$ and any unramified character of $Z(F)$ is of the form $\eta \pi$, the lemma follows from Lemma \ref{phi_chi} (iii).
\end{proof}

\subsection{Summary of reduction steps}  As noted above we may assume $\gamma$ is a norm; write $\gamma = \mathcal N \delta$.  

\begin{enumerate}

\item[(1)] {\em We may assume $G_{\rm der} = G_{\rm sc}$.}  Indeed, for given $G$ take $H$ to be a $z$-extension as above.  Then using Lemmas \ref{phi_chi} and \ref{P_Z} we see that the fundamental lemma for $H$ implies the fundamental lemma for $G$.

\smallskip

\item[(2)] {\em We may assume $\gamma$ is elliptic.}  We work under assumption (1).  If $\gamma$ is non-elliptic, we use Lemma \ref{non-elliptic_norms} and the discussion after it to associate to $\gamma$ a proper $F$-Levi subgroup $M$ containing $G_\gamma$.  By induction on semi-simple ranks, the fundamental lemma is known for $M$.   We then use our descent formulas (\ref{eq:descent_over_E}) and (\ref{eq:descent_over_F}), which can be compared (also in their stable incarnations) because of the lemmas of subsection \ref{lemmas_needed_to_compare_subsection}.

\smallskip

\item[(3)]  {\em We may assume $\gamma$ is regular.}  This is Clozel's Proposition 7.2 in \cite{Cl90}.  Note that it is explained there under assumptions (1) and (2).  

\smallskip

\item[(4)] {\em We may assume $G$ is such that $G_{\rm der} = G_{\rm sc}$ and $Z(G)$ is an induced torus (of the form considered above).}   We work under assumptions (1) and (3).  Clozel proves this reduction in \cite{Cl90}, $\S 6.1 (b)$, in the case of spherical Hecke algebras.  Exactly the same reasoning applies here, with the following minor modifications: where Clozel uses the Satake isomorphism $S$, we use the inverse $B^{-1}$ of the Bernstein isomorphism.  Where Clozel uses the characteristic functions of the sets $G(\mathcal O_F) \varpi^\lambda G(\mathcal O_F)$, we use the Bernstein functions $z_\lambda$.  Further, where Clozel invokes the fundamental lemma for the unit elements in spherical Hecke algebras, we invoke the analogous result for the unit elements in parahoric Hecke algebras (also proved in \cite{Ko86b}).  Finally, where Clozel uses a descent argument, we need to use our formula (\ref{eq:1st_descent_over_F}) with $b\phi$ replaced by appropriate unit elements in parahoric Hecke algebras for $G(F)$.  

\smallskip

\item[(5)] {\em We may assume $G$ is adjoint.}  We consider again our $z$-extension $p:H \rightarrow G$, but now we assume $Z = Z(H)$, so that $G = H_{\rm ad}$.  By (4), it is now enough to verify the fundamental lemma for groups having the properties enjoyed by $H$.  But Lemma \ref{adjoint_reduction} shows that the fundamental lemma for $G$ implies the fundamental lemma for $H$.  Hence it is enough to verify the fundamental lemma for adjoint groups.

\smallskip

\item[(6)] {\em We may assume $\gamma$ is strongly regular elliptic.}  Indeed, as explained in \cite{Cl90}, p. 292, it is now enough (by a continuity argument) to verify the matching of (twisted) orbital integrals for any dense set of $\theta$-regular $\theta$-elliptic elements $\delta$.  Note: when approximating a regular element in an adjoint group by strongly regular elements, the presence of the factors $a(\delta')$ in (\ref{STO_defn}) complicates the reduction, since those factors are trivial for strongly ($\theta$-)regular elements.  To get around this we may carry out that approximation in an appropriate $z$-extension and then use formulas  (\ref{integrate_SO_chi=1}) and (\ref{TO_chi=1}).

\end{enumerate}

\medskip

\noindent {\bf Conclusion:}  We may assume $G = G_{\rm ad}$, and $\gamma$ is a strongly regular elliptic element, which is a norm.  Moreover, it is enough to prove the required matching for any dense subset of such elements $\gamma$.

\section{A compact trace identity and its inversion} \label{compact_trace_identity_section}

A very useful theory of compact traces was worked out by Clozel in \cite{Cl89}.  In loc.~cit.~Prop.~1 Clozel gives an expression for the trace of a function $f \in C^\infty_c(G(F))$ in terms of the compact traces on Jacquet modules of constant terms associated to $\overline{f}$, the conjugation-average of $f$ over a good maximal compact open subgroup of $G(F)$.  He also establishes an inversion of this formula which expresses compact traces in terms of ordinary traces on Jacquet modules.

In our setting, we need a version of this which applies to central elements in parahoric Hecke algebras $\mathcal H_J(G)$.  Since the operation $f \mapsto \overline{f}$ does not preserve $J$-bi-invariance, it is necessary to establish a version of Clozel's formula (more importantly its inversion) which {\em does not involve} the average $\overline{f}$.  Our version of Clozel's formula (resp. its inversion) is given in Prop. \ref{compact_trace_identity_lemma} (resp.~Prop.~\ref{compact_trace_inversion}), and it holds for all smooth compactly-supported functions on $G(E)$.  For functions in $\mathcal H_J(G(E))$, this inversion formula can be simplified (Cor.~ \ref{compact_trace_inversion_cor}); for functions in the center $Z(\mathcal H_J(G(E)))$, further simplification is possible (subsection \ref{simplification_subsection}) and this yields the statement, namely Prop. \ref{simplification_prop}, which is actually used in the final step of the proof in section \ref{putting_it_all_together_section}.

The approach to compact traces which seemed easiest to adapt to our setting is contained not in Clozel \cite{Cl89}, but rather in Kottwitz' (unpublished) notes for his Chicago seminar talks on Clozel's work.  I am grateful to Kottwitz for providing me with this very useful reference.

\subsection{Notation} \label{compact_trace_notation_subsection}

We denote $\Gamma := {\rm Gal}(\overline{F}/F)$ throughout this subsection.  For the moment $G$ denotes an arbitrary connected reductive $F$-group, but later when the norm map $\mathcal N$ appears, we will simplify things by assuming $G$ is quasi-split over $F$.

For $M$ an $F$-Levi subgroup of $G$, let $A_M$ denote the split component of the center of $M$.  Let $\mathfrak a_M = X_*(A_M)_\mathbb R$.  If $P = MN$ is an $F$-parabolic subgroup of $G$ with Levi factor $M$, let $\Delta_P$ denote the set of simple roots for the action of $A_M$ on ${\rm Lie}(N)$.   

Suppose we have fixed a minimal $F$-parabolic $P_0 = M_0N_0$ \footnote{In case $G$ is unramified, assume $P_0 = B$, the Borel subgroup specified in section \ref{preliminaries_section}.}, and consider only standard $F$-parabolics $P = MN$, so that $M \supseteq M_0$ and $N \subseteq N_0$.  Set $\Delta_0 := \Delta_{P_0}$ and $\mathfrak a_0 = \mathfrak a_{M_0}$.  We have 
\begin{eqnarray*}
\mathfrak a_M \subseteq \mathfrak a_0 \hspace{.5in} \mbox{and} \hspace{.5in} \mathfrak a_0 \rightarrow \mathfrak a_M,
\end{eqnarray*}
where the map $\mathfrak a_0 \rightarrow \mathfrak a_M$ is given by averaging over the Weyl group $W_M$ of 
$M$.   By duality the resulting map $\mathfrak a_0/\mathfrak a_G \rightarrow \mathfrak a_M/\mathfrak a_G$ determines an inclusion of bases $\Delta_P \subset \Delta_0$.

We can then define $\tau^G_P$ to be the characteristic function of the ``acute cone'' in $\mathfrak a_0$
$$
\{ H \in \mathfrak a_0 ~ | ~ \alpha(H) > 0, \,\,\,\, \mbox{for all $\alpha \in \Delta_P$} \}.
$$
Similarly, we define the characteristic function $\widehat{\tau}^G_P$ for the ``obtuse cone''
$$
\{ H \in \mathfrak a_0 ~ | ~ c_\alpha(H) > 0, \,\,\,\, \mbox{for $\alpha \in \Delta_P$} \},
$$
where $H = \sum_{\alpha \in \Delta_0} c_\alpha(H) \alpha^\vee$ modulo $\mathfrak a_G$.

\medskip

We also have the Harish-Chandra functions $H_M: M(F) \rightarrow \mathfrak a_M$.  Since $\mathfrak a_M$ is the $\mathbb R$-vector space dual of $X^*(M)^\Gamma \otimes \mathbb R$, to define $H_M$ it is enough to specify the pairing $$
M(F) \times X^*(M)^\Gamma \rightarrow \mathbb R
$$
which is done by mapping $(m,\lambda)$ to ${\rm log}_q|\lambda(m)|_F$, where $q$ denotes the cardinality of the residue field of $F$.  

Now (for simplicity) assume $G$ (hence also $M$) is quasi-split over $F$, so that the norm map $\mathcal N$ on $M(E)$ takes values in $M(F)$.  Then we also define functions on $M(E)$ by
\begin{align*}
\chi_N(m) &:= \tau^G_P \circ H_M (\mathcal N m) \\
\widehat{\chi}_N(m) &:= \widehat{\tau}^G_P \circ H_M (\mathcal N m).
\end{align*}
Note that $H_M(\mathcal Nm)$ is a well-defined and continuous function of $m \in M(E)$. To see this we pass to a $z$-extension $M'$ of $M$, and use the fact that $H_M$ extends to a homomorphism $M(\overline{F}) \rightarrow \mathfrak a_M$. 

\subsection{The compact trace identity}

In this subsection we assume that $G$ is an unramified $F$-group.  Recall that $g \in G(F)$ belongs to the subset $G(F)_c$ of {\em compact elements} if and only if the eigenvalues $\chi$ of 
${\rm Ad}(g)$ 
acting on ${\rm Lie}(G)$ satisfy $|\chi|_F \leq 1$.   If $M$ is an $F$-Levi factor of an $F$-parabolic subgroup $P = MN$ of $G$, then we define 
\begin{equation*}
M(F)_{c,+} = \{ m \in M(F)_c ~ | ~ \mbox{the eigenvalues ${\chi}$ for ${\rm Ad}(m)$ acting on ${\rm Lie}(N)$ satisfy $|\chi|_F < 1$} \}.
\end{equation*}
Recall also that $G(E)_{\theta-c}$ consists of the $\theta$-semi-simple elements $g \in G(E)$ such that $\mathcal N g \in G(F)_c$ (in case $G_{\rm der} = G_{\rm sc}$, this is equivalent to requiring that the concrete norm $Ng$ belongs to $G(E)_c$).  Moreover, we define
\begin{equation}
M_{\theta-c,+} = M(E)_{\theta-c,+} := \{ m \in M(E) ~ | ~ \mathcal N m \in M(F)_{c,+} \}.
\end{equation}

Fix an $F$-Levi $M$ in $G$.  Consider the open subset of $G(F)$
$$G(F)_M := \{ g^{-1}mg ~ | ~ g \in G(F), \,\, m \in M(F)_{c,+} \}.
$$
The following basic identity holds for $f \in L^1(G(F), dg)$:
\begin{equation}
\int_{G(F)_M} f(g) \, dg = \int_{M(F)\backslash G(F)} \int_{M_{c,+}} \delta_{P(F)}^{-1}(m) \, f(g^{-1}mg) \, dm \, \frac{dg}{dm}.
\end{equation}
The same argument which leads to the above formula also proves the following twisted version.  Consider the open subset of $G(E)$
$$
G(E)(M_{\theta-c,+}) = \{ g^{-1}m\theta g ~ | ~ g \in G(E), \,\, m \in M(E)_{\theta-c,+} \}.
$$
Then for any $f \in L^1(G(E),dg)$, we have
\begin{equation} \label{one_Levi_contribution}
\int_{G(E)(M_{\theta-c,+})} f(g) \, dg = \int_{M(E)\backslash G(E)} \int_{M_{\theta-c,+}} \delta_{P(F)}^{-1}(\mathcal Nm) \, f(g^{-1}m\theta g) \, dm \, \frac{dg}{dm}.
\end{equation}

Now let $P_0 = M_0N_0 = TU$ be our fixed minimal $F$-parabolic, and consider only standard $F$-parabolics $P = MN$.  Summing (\ref{one_Levi_contribution}) over all standard $F$-parabolics yields
\begin{lemma} \label{proto_compact_trace_identity} For all $f \in L^1(G(E),dg)$, we have
\begin{equation} \label{eq:proto_compact_trace_identity}
\int_{G(E)} f(g) \, dg = \sum_{P = MN} \int_{M(E)\backslash G(E)} \int_{M_{\theta-c,+}} \delta_{P(E)}^{-1}(m) \, 
f(g^{-1}m\theta g) \, dm \, \frac{dg}{dm}.
\end{equation}
\end{lemma}

On the other hand, we have the integration formula for $\psi \in L^1(M(E)\backslash G(E), \frac{dg}{dm})$
\begin{equation} \label{G/M}
\int_{M(E)\backslash G(E)} \psi(g) \, \frac{dg}{dm} = \sum_{w \in \, _EW(P,J)} q_{E,w}^{-1} \int_N \int_J \psi(nwj) \, dj\, dn,
\end{equation}
a consequence of (\ref{eq:integration_on_G}); in particular we are using the measures $dj$ and $dn$ from (\ref{eq:integration_on_G}).   Now we define for any $w \in \, _EW(P,J)$, a function $f_{P,J,w}$ on $M(E)$ by the formula
\begin{equation} \label{f_PJw}
f_{P,J,w}(m) := \int_N \int_J f(j^{-1}w^{-1}m n \,\theta(w) \, \theta j) \, dj \, dn_w,
\end{equation}
where $dn_w$ is the measure on $N(E)$ such that ${\rm vol}_{dn_w}(N \cap \, ^wJ) = 1$.  
Consider the automorphism of $N(E)$ given by $n \mapsto n^{-1}m \theta(n) \,m^{-1}$ for $m \in M_{\theta-c,+}$.  The absolute value of the Jacobian of this map is identically 1.  Substituting (\ref{G/M}) and (\ref{f_PJw}) into 
(\ref{eq:proto_compact_trace_identity}), and using this change of variable, we derive the following equation
\begin{equation} \label{second_proto_compact_trace_identity}
\int_{G(E)} f(g) \, dg = \sum_{P=MN} \sum_{w \in \, _EW(P,J)} \int_{M_{\theta-c,+}} f_{P,J,w}(m) dm_w,
\end{equation}
where $dm_w$ is the measure on $M(E)$ such that ${\rm vol}_{dm_w}(M \cap \, ^wJ) = 1$.  
Consider two functions $f_1,f_2$ on $G(E)$ whose product is integrable, and set
\begin{align*}
\langle f_1,f_2 \rangle = \int_{G(E)} f_1(g)f_2(g) \, dg \\
\langle f_1,f_2 \rangle_c = \int_{G(E)_{\theta-c}} f_1(g)f_2(g) \, dg.
\end{align*}
Similar notation will apply to functions on $M(E)$.  Following Harish-Chandra, let $\Theta_{\Pi \theta}$ denote the locally integrable function on $G(E)$ representing the functional $\phi \mapsto \langle {\rm trace} \, \Pi I_\theta, \phi \rangle$, for $\phi \in C_c^\infty(G(E))$.  This means that
\begin{equation} \label{twisted_char_defn}
\langle {\rm trace} \, \Pi I_\theta, \phi \rangle = \langle \Theta_{\Pi \theta}, \phi \rangle,
\end{equation}
where the intertwiner $I_\theta: \Pi  ~ \widetilde{\rightarrow} ~ \Pi^\theta$ is understood to be fixed (and the representing function $\Theta_{\Pi \theta}$ obviously depends on the choice of $I_\theta$).

 We may apply (\ref{second_proto_compact_trace_identity}) to $f(g) := \Theta_{\Pi \theta}(g) \phi(g)$, and in the above notation we get the following compact trace identity.

\begin{prop} \label{compact_trace_identity_lemma} For every $\phi \in C^\infty_c(G(E))$, we have
\begin{equation} \label{compact_trace_identity} 
\langle \Theta_{\Pi \theta}, \phi  \rangle = \sum_{P = MN} \sum_{w \in \, _EW(P,J)} \langle \Theta_{\Pi_N\theta} \,\, , \,\, \phi_{P,J,w} \cdot \chi_N \rangle_c,
\end{equation}
where in the inner summand the Haar measure on $M(E)$ used to form $\langle \Theta_{\Pi_N\theta}\,, \, \phi_{P,J,w} \cdot \chi_N \rangle_c$ is $dm_w$, i.e. the one giving $^wJ \cap M$ volume 1. 
\end{prop}

\begin{proof}
The identity
\begin{equation*}
\int_{G(E)} \Theta_{\Pi\theta}(g) \phi(g) \, dg = \sum_{P = MN} \sum_{w \in \, _EW(P,J)} \int_{M_{\theta-c,+}} \Theta_{\Pi_N\theta}(m) \, \phi_{P,J,w}(m) \, dm_w 
\end{equation*}
follows easily using (\ref{eq:proto_compact_trace_identity}) and (\ref{second_proto_compact_trace_identity}) together with the identities for $g \in G(E)$ and $m \in M_{\theta-c,+}$
\begin{equation}
\Theta_{\Pi \theta}(g^{-1}m\theta g) = \Theta_{\Pi \theta}(m) = \Theta_{\Pi_N\theta}(m).
\end{equation}
The second equality is the twisted version due to Rogawski \cite{Rog} of a theorem of Casselman \cite{Cas77}.

Finally, we use the well-known (and easy) fact that $m \in M(E)_{\theta-c}$ belongs to $M_{\theta-c,+}$ if and only if $\chi_N(m) = 1$.
\end{proof}

\subsection{Inversion of the compact trace identity}

Let $a_P := {\rm dim}(\mathfrak a_M)$.  

\begin{prop}\label{compact_trace_inversion}  For any $\phi \in C^\infty_c(G(E))$, we have the formula
\begin{equation} \label{eq:compact_trace_inversion}
\langle \Theta_{\Pi \theta} \, , \, \phi \rangle_c = \sum_{P=MN} (-1)^{a_P-a_G} 
\sum_{w \in \, _EW(P,J)} \langle \Theta_{\Pi_N\theta} \, , \, \phi_{P,J,w} \cdot \widehat{\chi}_N \rangle,
\end{equation}
where we use $dm_w$ to form $\langle \cdot \, , \, \cdot \rangle$ in the inner summand on the right hand side.
\end{prop}

\begin{proof}
We will apply Proposition \ref{compact_trace_identity_lemma} to each summand in the right hand side of (\ref{eq:compact_trace_inversion}).  Let $Q = LR$ denote a standard $F$-parabolic contained in $P$, with Levi factor $L$ and unipotent radical $R$.  We have $L \subseteq M$, $R \supseteq N$, and $Q \cap M = L(R \cap M)$ is a Levi decomposition of the $F$-parabolic $Q \cap M \subseteq M$.  

For each $w \in \, _EW(P,J)$, consider the subset $_EW(Q \cap M, \,^wJ \cap M)$ of $W_M(E)$, and let $\tau$ denote an element of that subset.  It is clear that $\tau w \in \, (W_L\backslash W)_{\rm min}$.  Furthermore, we can define the subset $_EW(Q,J) \subset W(E)$ in such a way that every element of $_EW(Q,J)$ can be written in the form $\tau w$ for unique elements $w \in \, _EW(P,J)$ and $\tau \in \, _EW(Q \cap M, \, ^wJ \cap M)$.  
Indeed, in light of Lemma \ref{lemma_A} it suffices to observe that for the natural projection
$$
Q \backslash G/J \rightarrow P \backslash G/J,
$$
the fiber over $w \in P \backslash G/J$ can be identified with $Q \cap M \backslash M/ \,^wJ \cap M$.  (Use Lemma \ref{J_cap_M_is_parahoric}, part (b).)

\begin{lemma} \label{transitivity_of_sub_P_J_w}
With the notation above, for $\phi \in C^\infty_c(G(E))$, we have
\begin{equation}
(\phi_{P,J,w})_{Q \cap M, \, ^wJ \cap M, \tau} = \phi_{Q,J,\tau w}.
\end{equation}
\end{lemma}
 
\begin{proof}
The left hand side is the function on $L(E)$ which assigns to $l \in L(E)$ the value
\begin{equation}
\int_N \int_{R \cap M} \int_J \int_{\, ^wJ \cap M} \phi(j^{-1}w^{-1}j_0^{-1}\tau^{-1} \, lr \, \theta\tau \, \theta j_0 \, n \, \theta(w) \, \theta j) \, dj_0 \, dj \, dr_{w,\tau} \, dn_w,
\end{equation}
where $j_0 \in \, ^wJ \cap M$, $j \in J$, $r \in R \cap M$, and $n \in N$.  Further, $dr_{w,\tau}$ is the measure on $R \cap M$ giving volume 1 to $\, ^\tau(\, ^wJ \cap M) \cap R \cap M = \, ^{\tau w}J \cap R \cap M$.

The integrand can be rewritten as
$$
\phi(j^{-1} \, (\, ^{w^{-1}}j_0) \, (\tau w)^{-1} \, l r (\, ^{\theta(\tau j_0)} n) \, \theta(\tau w) \, \theta(\, ^{w^{-1}}j_0) \theta j).
$$
Since $^{w^{-1}}j_0 \in J$, we may suppress the integral over $^wJ \cap M$.  Also, $n \mapsto \, ^{\theta(\tau j_0)}n$ is a measure-preserving transformation of $N(E)$, so we may suppress the superscript $\theta(\tau j_0)$ above $n$.

Letting $dr_{\tau w}$ denote the measure on $R$ giving $^{\tau w}J \cap R$ volume 1, we have by Fubini's theorem the equality of measures on $R = (R \cap M) \cdot N$ 
\begin{equation*}
dr_{\tau w} = dr_{w,\tau} \cdot dn_w.
\end{equation*}
Indeed, 
\begin{align*}
^{\tau w}J \cap R &= (\, ^{\tau w}J \cap R \cap M) \cdot (\, ^{\tau w}J \cap N) \\
&= (\, ^{\tau w}J \cap R \cap M) \cdot \, ^\tau(\, ^wJ \cap N),
\end{align*}
and conjugation by $\tau$ leaves $dn_w$ invariant.  These remarks are enough to prove the lemma.
\end{proof}
 
Now we continue with the proof of the proposition.  In the expressions below, $w$ resp. $\tau$ will be understood to range over $_EW(P,J)$ resp. $_EW(Q \cap M, \, ^wJ \cap M)$.  Further, we will abbreviate $\phi_{P,J,w}$ by $\phi_w$ and $\psi_{Q \cap M, \, ^wJ \cap M, \tau}$ by $\psi_\tau$.  Finally, we simply write $\tau^G_P$ resp. $\widehat{\tau}^G_P$ in place of $\chi_N = \tau^G_P \circ H_M \circ \mathcal N$ resp. $\widehat{\chi}_N = \widehat{\tau}^G_P \circ H_M \circ \mathcal N$.  Then applying Proposition \ref{compact_trace_identity_lemma} to each summand in (\ref{eq:compact_trace_inversion}), the right hand side becomes
\begin{eqnarray*}
& & \sum_{P = MN} (-1)^{a_P - a_G} \sum_{w} \sum_{\underset{P \supseteq Q}{Q= LR}} \sum_{\tau} \langle \Theta_{\Pi_R\theta} \, , \, (\phi_w \widehat{\tau}^G_P)_\tau \cdot \tau^P_Q \rangle_c  \\
& = & \sum_{Q = LR} \langle \Theta_{\Pi_R\theta} \, , \, \sum_{P \supseteq Q}(-1)^{a_P-a_G}\sum_{w,\tau} \phi_{\tau w} \, \widehat{\tau}^G_P \tau^P_Q  \rangle_c \\
& = & \sum_{Q = LR} \langle \Theta_{\Pi_R \theta} \, , \, [\sum_{P \supseteq Q} (-1)^{a_P-a_G} ~ \widehat{\tau}^G_P\tau^P_Q] \sum_{y \in \, W(Q,J)}\phi_{y} \rangle_c \\
&= & \langle \Theta_{\Pi \theta} \, , \, \phi \rangle_c. 
\end{eqnarray*}
The equality $(\phi_w)_\tau = \phi_{\tau w}$ we used is Lemma \ref{transitivity_of_sub_P_J_w}.  The final equality results from the well-known result of Arthur (\cite{Ar78}, Lemma 6.1) that for fixed standard parabolic $Q = LR$,
$$
\sum_{P \supseteq Q} (-1)^{a_P-a_G} ~ \widehat{\tau}^G_P \tau^P_Q = \begin{cases} 1, \,\,\,\, \mbox{if $Q = G$} 
\\ 0, \,\,\,\, \mbox{if $Q \neq G$}. \end{cases}
$$
We also used implicitly the equality $(\widehat{\tau}^G_P \circ H_M \circ \mathcal N_M)_{\tau} = \widehat{\tau}^G_P \circ H_L \circ \mathcal N_L$ (which was again denoted simply by $\widehat{\tau}^G_P$ above).  This is not hard, the main ingredients being the fact that the following diagram commutes
$$
\xymatrix{
L(E) \ar[r] \ar[d]_{H_L\mathcal N_L} & M(E) \ar[d]^{H_M\mathcal N_M} \\
\mathfrak a_L/\mathfrak a_G \ar[r] & \mathfrak a_M/\mathfrak a_G,}
$$
where the lower horizontal map is the projection given by averaging over the Weyl group $W_M$, and the fact that the dual of that projection induces the inclusion of bases $\Delta_P \hookrightarrow \Delta_Q$.  

This completes the proof of Proposition \ref{compact_trace_inversion}.
\end{proof}

\begin{cor} \label{compact_trace_inversion_cor}
Assume $\phi \in \mathcal H_J(G(E))$.  Then 
\begin{equation} \label{eq:compact_trace_inversion_cor}
\langle \Theta_{\Pi \theta} \, , \, \phi \rangle_c = \sum_{P=MN} (-1)^{a_P-a_G} 
\sum_{w \in \, _EW(P,J)} \langle \delta_{P(E)}^{-1/2} \Theta_{\Pi_N\theta} \, , \, (\, ^{w,\theta}\phi)^{(P)} \cdot \widehat{\chi}_N \rangle,
\end{equation}
where $(\cdot)^{P(E)}$ is formed using the measure $dn_w$ and $\langle \cdot, \cdot \rangle$ is formed using the measure $dm_w$.
\end{cor}

\subsection{Simplification of the inverted compact trace identity} \label{simplification_subsection}

In this section, we assume $w \in \, _EW(P,J)$ and $\phi \in Z(\mathcal H_J(G(E)))$, and we give a formula for the quantity 
\begin{equation} \label{term_to_be_simplified}
\langle \delta_{P(E)}^{-1/2} \Theta_{\Pi_N\theta}, (\, ^{w,\theta}\phi)^{(P)} \widehat{\chi}_N \rangle,
\end{equation}
especially in the case where $\phi = z_\mu$.   The useful consequence will be the simple description of the compact trace of a central element, given in Proposition \ref{simplification_prop} below.

\subsubsection{A variant of Lemma \ref{key_calc}}
It will be enough for us to analyze (\ref{term_to_be_simplified}) with $\Pi_N$ replaced by a $\theta$-stable admissible representation $\sigma$ of $M(E)$ (which is a subquotient of $\Pi_N$).  Furthermore, by \cite{Rog}, Lemma 2.1, it is enough to consider the case where $\sigma$ is irreducible as an $M(E)$-representation, which we henceforth assume.  

Now given a representation $\Pi$ with Iwahori-fixed vectors, we may choose an unramified character $\xi$ of $T(E)$ such that $\Pi$ is an irreducible subquotient of $i^G_B(\xi)$.  By Casselman's theorem (\cite{Cas}, 6.3.4), the irreducible subquotients of the Jacquet module $\Pi_N$ are subquotients of
$$
\bigoplus_{\eta \in (W_M(E)\backslash W(E))_{\rm min}} i^M_{B_M}( \, ^\eta\xi \cdot \delta_P^{1/2}).
$$
(Here we used implicitly the identity $^\eta B \cap M = B \cap M$; see Lemma \ref{lemma_A} (a).)   Thus, there exists an element $\eta \in W(E)$ such that our given $\sigma$ is a subrepresentation of $i^M_{B_M}(\, ^\eta \xi\delta_P^{1/2})$.  We fix $\eta$ and denote $^\eta \xi \delta_P^{1/2}$ by $\xi'$.   We set $\xi'_1 = \, ^\eta \xi$. 

\medskip

By definition (\ref{term_to_be_simplified}) is related to the {\em left} action of $M(E) \rtimes \langle \theta \rangle$ on $\sigma$.  In order to make explicit computations, we need to pass to {\em right} actions, as we did in the proof of Lemma \ref{lemma_C}.  Recall the anti-involution $\iota$ of the Hecke algebra, defined by $\iota f(x) = f(x^{-1})$.  


We have the following identity for $f \in C^\infty_c(M(E))$ and a $\theta$-stable $\sigma$:
\begin{equation} \label{basic_left_right_action_identity}
{\rm trace}(L(f) \circ I_\theta; \,\sigma) = {\rm trace}( R(\iota f) \circ I_\theta;\, \sigma).
\end{equation}
We apply this to $f = \delta_P^{-1/2}(\, ^{w,\theta}\phi)^{(P)}\widehat{\chi}_N$ and use (\ref{relations}) to get
\begin{equation} \label{in_terms_of_R}
\langle \delta_{P(E)}^{-1/2} \Theta_{\sigma\theta}, (\, ^{w,\theta}\phi)^{(P)} \widehat{\chi}_N \rangle = 
{\rm trace} \,(R(\delta_{P}^{1/2}(\, ^{\theta(w),\theta^{-1}}\iota \phi)^{(P)} \iota\widehat{\chi}_N) \circ I_{\theta}; \, \sigma).
\end{equation}
We will find an explicit formula for the right hand side of (\ref{in_terms_of_R}).


Now for any parahoric $J$ and $\phi \in Z(\mathcal H_J(G(E)))$, any unramified character $\xi$ of $T(E)$, and any standard $F$-parabolic $P = MN$, we define a scalar quantity 
\begin{equation} \label{ch(...,N)}
ch_{\, \xi}(\phi,J,N) :=
\sum_{\nu \in X_*(A^E)} \widehat{\chi}_N(\varpi^\nu) \, 
\xi(\varpi^\nu) \, B^{-1}(\phi)(\varpi^\nu).
\end{equation}

Define the morphisms $I^G_\theta$, $R(w \, \theta(w)^{-1})$, $\mathcal P$, and $\mathfrak s$ as in the proof of Lemma \ref{key_calc}.  The following variant of that lemma is the key ingredient in our simplification.  We postpone its proof to the next subsection.

\begin{lemma} \label{key_calc_variant}   Suppose $\sigma$ and $\xi'$ resp. $\xi'_1$ are as above, so that $\sigma 
\subseteq i^M_{B_M}(\xi')$.  Let $w \in \, _EW(P,J)$.  Define the right action $R(\psi) = -*\psi$ on $i^M_{B_M}(\xi')$ (and $\sigma$) using the Haar measure on $M(E)$ which gives $^wJ_M := \, ^wJ \cap M$ (or equivalently, $J'_M$) volume 1.  Let $\phi \in Z(\mathcal H_J(G(E)))$.  Then 
we have the following equality of linear functions $\sigma^{^wJ_M} \rightarrow \sigma^{^wJ_M}$
\begin{equation} \label{eq:key_calc_variant}
ch_{\xi'_1}(\, ^{\theta(w)}\phi, \, ^{\theta(w)}J, N) \cdot \mathcal P \circ R(w \, \theta(w)^{-1}) \circ I^G_{\theta} \circ \mathfrak s = 
R(\delta_P^{1/2}(\, ^{\theta(w),\theta^{-1}}\iota \phi)^{(P)} \, \iota \widehat{\chi}_N) \circ 
I^M_{\theta}.
\end{equation}
\end{lemma}

This easily yields via (\ref{in_terms_of_R}) the desired simplification of (\ref{term_to_be_simplified}): it is a linear combination of terms of the form $ch_{^\eta \xi}(\, \phi, \, J, N)$.  Indeed, take the trace of both 
sides of (\ref{eq:key_calc_variant}) and note that the left hand side is a scalar times a matrix which is independent of $\phi$.  An argument like the one in subsection \ref{beginning_of_proof_lemma_C}  (see also Lemma \ref{v_lemma} below) shows that the trace of (\ref{eq:key_calc_variant}) vanishes unless $w = \theta(w)$.  In that case $w \in W(F)$, and the equality $B^{-1}(\, ^w\phi) = B^{-1}(\, \phi)$ from subsection \ref{w_conjugation_subsection} implies 
$ch_{\xi_1'}(\, ^w\phi, \, ^wJ, N) = ch_{\xi'_1}(\phi, J, N)$, as desired.

Together with Corollary \ref{compact_trace_inversion_cor} this yields the following important proposition.  To state it we first introduce some new notation by writing
$$\langle \Theta_{\Pi \theta}, \phi \rangle_c =: \langle {\rm trace} \, \Pi I_\theta, \phi \rangle_c =: \langle {\rm trace}\, \Pi I_\theta , \phi\rangle_{\theta-c}$$ 
in anticipation of the argument in section 9 where various kinds of traces (usual, compact, or $\theta$-compact) need to be distinguished.

\begin{prop} \label{simplification_prop}  Let $\Pi$ be an irreducible subquotient of $i^G_B(\xi)$, where $\xi$ is an unramified character of $T(E)$.  Then the functional on $\phi \in  Z(\mathcal H_J(G(E)))$ given by
$$
\phi \mapsto \langle {\rm trace}\, \Pi I_\theta , \phi \rangle_c
$$
is a finite linear combination of functionals of the form
$$
\phi \mapsto \sum_{\nu \in X_*(A^E)} \widehat{\chi}_N(\varpi^\nu) \, ^\eta\xi(\varpi^\nu) \, B^{-1}(\phi)(\varpi^\nu),$$
where $P =MN$ ranges over standard $F$-parabolic subgroups, and $\eta$ ranges over certain elements of the Weyl group $W(E)$.  Restricting our attention to the Bernstein functions $\phi = z_\mu$, we see that the functional on $\mu \in X_*(A^E))$ given by
$$
 \mu \mapsto \langle {\rm trace}\, \Pi I_\theta , z_\mu \rangle_c$$
 is a finite linear combination of functionals which send $\mu$ to
\begin{equation} \label{compact_trace_sum}
\sum_{\lambda \in W(E) \mu} \widehat{\chi}_N(\varpi^\lambda) \, \, ^\eta\xi(\varpi^\lambda).  
\end{equation}
\end{prop}

\subsubsection{Proof of Lemma \ref{key_calc_variant}}

We proceed as in the proof of Lemma \ref{key_calc}.  Since the analogue of (\ref{mathcalP_is_theta_equivariant}) holds, it is enough to prove the equality
\begin{equation} \label{modified_key_calc_variant}
ch_{\xi'_1}(\, ^{\theta(w)}\phi, \, ^{\theta(w)}J, N) \cdot \mathcal P \circ R(w \, \theta(w)^{-1}) \circ \mathfrak s =  
R(\delta_P^{1/2}(\, ^{\theta(w),\theta^{-1}}\iota \phi)^{(P)} \iota\widehat{\chi}_N) 
\end{equation}
of linear functions $\sigma^{J'_M} \rightarrow \sigma^{^wJ_M}$.

Also, we have a commutative diagram as in (\ref{P_s_diagram}), except that $\xi$ resp. $\xi_1$ is replaced with $\xi'$ resp. $\xi'_1$.  As before, it is therefore enough to prove (\ref{modified_key_calc_variant}) for $\mathcal P'$ resp. $\mathfrak s'$ replacing $\mathcal P$ resp. $\mathfrak s$.  Explicitly, for $\Phi \in i^M_{B_M}(\xi')^{J'_M}$, the element $\mathfrak s'\Phi =: \widetilde{\Phi} \in i^G_B(\xi'_1)^{J'}$ is the unique one supported on $PJ'$ and satisfying $\widetilde{\Phi}|_{M(E)} = \Phi$.  

Since elements of $i^M_{B_M}(\xi')^{^wJ_M}$ are determined by their values on $_EW(B_M,\, ^wJ_M)\, ^wJ_M$, it is enough to prove the following analogue of Lemma \ref{main_calc}.

\begin{lemma} \label{main_calc_variant} Define the right action $*$ on $i^M_{B_M}(\xi')$ using the measure which gives $^wJ_M$ (equivalently, $J'_M$) volume 1.  Let $\Phi \in i^M_{B_M}(\xi')^{J'_M}$.  Let $\widetilde{\Phi} = \mathfrak s' \Phi \in i^G_B(\xi'_1)^{J'}$, as defined above.  Then for any $y \in \tau_0 \, ^wJ_M$ ($\tau_0 \in \, _EW(B_M,\, ^wJ_M)$), we have the identity
\begin{equation} \label{eq:main_calc_variant}
ch_{\xi'_1}(\, ^{\theta(w)}\phi, \, ^{\theta(w)}J, N) \cdot \, \widetilde{\Phi} (y w \, \theta(w)^{-1}) = \Phi * [\delta_P^{1/2} (\, ^{\theta(w),\theta^{-1}}\iota \phi)^{(P)} \iota \widehat{\chi}_N](y).
\end{equation}
\end{lemma}

\begin{proof}
First, we extend $\widehat{\chi}_N$ to a function on $M(E)N(E) \, ^{\theta(w)}J \subset G(E)$, denoted by $\widetilde{\chi}_N$, by defining
$$
\widetilde{\chi}_N (m\, n \, \theta(w)j \theta(w)^{-1}) = \widehat{\chi}_N(m),
$$
for $j \in J(E)$, $n \in N(E)$, and $m \in M(E)$.  Note that $\widetilde{\chi}_N$ is well-defined, since $H_M$ is trivial on compact subgroups of $M(E)$.

 Now as in the proof of Lemma \ref{main_calc}, we have 
\begin{eqnarray*}
& &\Phi * [\delta_P^{1/2}(\, ^{\theta(w),\theta^{-1}}\iota\phi)^{(P)} \iota \widehat{\chi}_N](y) \\ 
&=& \int_{M(E)} \int_{N(E)} \int_{J'} \widetilde{\Phi}(mn j') \, (\, ^{\theta(w),\theta^{-1}}\iota \phi)((mn j')^{-1}y) \, \iota\widetilde{\chi}_N((mn j')^{-1}y) \, dn_{\theta(w)} \, dm_{\theta(w)} \, dj' \\
&=& \int_{G(E)} (\widetilde{\chi}_N \, \widetilde{\Phi})(x) (\, ^{\theta(w)}\iota \phi)(x^{-1}y\, w \, \theta(w)^{-1}) \, dx \\
&=& (\widetilde{\chi}_N \, \widetilde{\Phi}) * (\, ^{\theta(w)}\iota \phi) (y \, w \, \theta(w)^{-1}).
\end{eqnarray*}
Here we have used $y \in \tau_0 \, ^wJ_M$ (thus $H_M(y) = 0$) to justify $\iota \widetilde{\chi}_N(x^{-1}y) = \iota \widetilde{\chi}_N(x^{-1}) = 
\widetilde{\chi}_N(x)$.  We have made use of the integration formula (\ref{eq:normalized_integration_on_G}) for a function supported on $P(E)\, J'$.

Recall that $\widetilde{\Phi} \in i^G_B(\xi'_1)^{J'}$ is supported on 
$$
PJ' = \bigcup_{\tau \in \, _EW(B_M,J'_M)} B \tau  J'.
$$
We may write
\begin{equation}
\widetilde{\Phi} = \sum_{i \in \mathcal I} c_i \widetilde{\Phi}_{\tau_i},
\end{equation}
for a finite set of non-zero scalars $c_i$ ($i \in \mathcal I$) and basis elements $\widetilde{\Phi}_{\tau_i} \in i^G_B(\xi'_1)^{J'}$ supported on $B\tau_i J'$, where $\tau_i \in \, _EW(B_M, \, J'_M) \subset W_M(E)$.  We normalize so that $\widetilde{\Phi}_{\tau_i}(\tau_i) = 1$, for $i \in \mathcal I$.

We need to prove, finally, that
\begin{equation} \label{final_step}
ch_{\xi'_1}(\, ^{\theta(w)}\phi, \, ^{\theta(w)}J, N) \cdot \, \widetilde{\Phi}_{\tau_i} (y w \, \theta(w)^{-1}) = (\widetilde{\chi}_N \, \widetilde{\Phi}_{\tau_i}) * (\, ^{\theta(w)}\iota \phi) (y \, w \, \theta(w)^{-1}).
\end{equation}

From now on fix $i \in \mathcal I$ and denote $\tau_i$ simply by $\tau$.  Since ${\rm supp}(\widetilde{\Phi}_{\tau}) \subset B \tau \, ^{\theta(w)}J$, we may consider the left Haar measure $db = dtdu$ on $B = TU$ and its variant $db_{\tau ,\theta(w)}$ which is normalized such that the analogue of (\ref{eq:normalized_integration_on_G}) holds for functions supported on $B \tau \, ^{\theta(w)}J$.  Using the substitution $x = tu\tau j'$ for $x \in B\tau \, ^{\theta(w)}J$, we have 
\begin{eqnarray} \label{last_equation_above}
&& \\
& & (\widetilde{\chi}_N \widetilde{\Phi}_{\tau}) * (\, ^{\theta(w)}\iota \phi)(y\, w \, \theta(w)^{-1})  \notag \\
&=& \int_{T} \int_{U} \int_{J'} \widetilde{\chi}_N(tu\tau  j') \, \widetilde{\Phi}_{\tau}(tu\tau j') \, 
\,^{\theta(w)}\iota \phi((tu\tau j')^{-1}y \, w \, \theta(w)^{-1}) \, dj' \, du_{\tau ,\theta(w)} \, dt  \notag \\
& = &  \sum_{\nu \in X_*(A^E)} \widehat{\chi}_N(\varpi^{\nu}) \, (\delta^{1/2}_B\xi'_1)(\varpi^\nu) \, \int_U \int_{J} \widetilde{\Phi}_{\tau}(\tau) \, ^{\theta(w)}\iota \phi((\varpi^\nu u \tau j')^{-1} y \, w \, \theta(w)^{-1}) \, dj' \, 
du_{\tau ,\theta(w)}. \notag 
\end{eqnarray}
In writing $\widetilde{\chi}_N(tu\tau  j') = \widehat{\chi}_N(\varpi^{\nu})$ for $t \in \varpi^\nu (T(E) \cap T(L)_1)$ here we have used the definition of $\widetilde{\chi}_N$ along with the fact that $\tau \in K_M(E)$ (so that $H_M(\tau) = 0$). 

Setting $v_{\varpi^\nu \tau} := 1_{U\varpi^\nu\tau \, ^{\theta(w)}J}$, 
the double integral in the last line of (\ref{last_equation_above}) can be written as
\begin{eqnarray*}
&& \int_U \int_{J} 
v_{\varpi^\nu\tau}(\varpi^\nu u \tau j') \, ^{\theta(w)}\iota \phi((\varpi^\nu u \tau j')^{-1} y \, w \, \theta(w)^{-1}) \, dj' \, du_{\tau ,\theta(w)} \\
&=& (v_{\varpi^\nu \tau} * \, ^{\theta(w)}\iota \phi)(y \, w \, \theta(w)^{-1}).
\end{eqnarray*}
On the other hand, recalling (\ref{norm_left_convolution}-\ref{R_action_char_of_B}), by the very definition of the Bernstein isomorphism 
$$B : \mathbb C[X_*(A^E)]^{W(E)} 
~~ \widetilde{\rightarrow} ~~ Z(\mathcal H_{\,^{\theta(w)}J}(G(E))),$$
we have
$$
v_{\varpi^\nu \tau} * \, ^{\theta(w)}\iota \phi = B^{-1}(\,^{\theta(w)}\iota \phi) \cdot v_{\varpi^\nu \tau},
$$
where the right hand side is defined by a normalized integration over $A^E(E)$, so that
\begin{eqnarray*}
&& (v_{\varpi^\nu_F \tau} * \, ^{\theta(w)}\iota \phi)(y\, w \, \theta(w)^{-1}) \\
&=& \sum_{\lambda \in X_*(A^E)} \delta_B^{1/2}(\varpi^\lambda) \, B^{-1}(\, ^{\theta(w)}\iota \phi)(\varpi^\lambda) \, 
v_{\varpi^\nu \tau}(\varpi^{-\lambda} y \, w \, \theta(w)^{-1}) \\
&=& \delta_B^{1/2}(\varpi^{-\nu}) \, B^{-1}(\, ^{\theta(w)}\iota \phi)(\varpi^{-\nu}) \, v_{\varpi^\nu \tau}(\varpi^\nu y \, w\, \theta(w)^{-1}) \\
&=& \delta_B^{1/2}(\varpi^{-\nu}) \, B^{-1}(\, ^{\theta(w)}\phi)(\varpi^{\nu}) \, \widetilde{\Phi}_\tau (y \, w\, \theta(w)^{-1}).
\end{eqnarray*}
(See the lemma below.) Substituting this into the last line of (\ref{last_equation_above}) above yields the desired result.  Modulo the following lemma, we have verified (\ref{final_step}).  
This completes the proof of Lemma \ref{main_calc_variant} and thus of Lemma \ref{key_calc_variant} as well.  \end{proof}

\begin{lemma} \label{v_lemma}  Write $y = \tau_0 \, ^wj$, for $j \in J$ such that $^wj\in \, ^wJ \cap M$.  Then 
\begin{equation} \label{v_thing}
v_{\varpi^\nu \tau}(\varpi^{-\lambda} y w \, \theta(w)^{-1}) = v_{\varpi^\nu \tau}(\varpi^{-\lambda} \tau_0 w j \, \theta(w)^{-1})
\end{equation}
vanishes unless $\theta(w) = w$, $\tau = \tau_0$, and $\lambda = -\nu$, in which case (\ref{v_thing}) is 1.  Similarly, $\widetilde{\Phi}_\tau(y w \, \theta(w)^{-1})$ vanishes unless $\theta(w) = w$ and $\tau = \tau_0$, in which case we also have $\widetilde{\Phi}_\tau(y w \, \theta(w)^{-1}) = 1$.
\end{lemma}

\begin{proof}
The quantity (\ref{v_thing}) does not vanish if and only if
$$
\varpi^{-\lambda} \tau_0 w j \theta(w)^{-1} \in U \varpi^{\nu} \tau \, ^{\theta(w)}J,
$$
in other words, if and only if
$$
\varpi^{-\lambda} \tau_0 w \in U \varpi^{\nu} \tau \theta(w) J.
$$
In that event, we deduce first that $\theta(w) = w$, and then that $\tau_0 = \tau$ (see the beginning of the proof of Proposition \ref{compact_trace_inversion}).  Finally, these together imply that 
$\varpi^{-\lambda - \nu}$ belongs to $^{\tau_0 w}J$, and hence $\lambda + \nu = 0$.  The rest is easy.
\end{proof}

\subsection{Deformation of the parameter} \label{deformation_subsection}

In this subsection, we assume $G = G_{\rm ad}$.  We assume $\xi \in \widehat{A^E}$ is a parameter such that the full unramified principal series representation $\Pi := \Pi_\xi = i^{G(E)}_{B(E)}(\xi)$ is $\theta$-stable.  We claim that the compact trace $\langle {\rm trace} \, \Pi I_\theta, \phi \rangle_{\theta-c}$ depends only on the unitary part $\xi_u$ of the parameter $\xi$ (and thus this compact trace remains unchanged if we alter $\xi$ by an $\mathbb R_{>0}$-valued character, so that we may assume, as we shall later on, that the parameter $\xi$ is {\em non-unitary}).   

It is enough to give a suitable formula for $\Theta_{\Pi \theta}(g)$, where $g$ belongs to the open dense subset of $G(E)$ consisting of $\theta$-regular $\theta$-semi-simple elements of $G(E)$ (equivalently, $\mathcal Ng$ is regular semi-simple; see \cite{Cl90}, $\S2.2$).  We will then examine this formula in the special case where $g$ is also $\theta$-compact (i.e. $\mathcal N g$ is compact).   By assumption the $W(E)$-conjugacy class of $\xi$ is $\theta$-stable.  By replacing $\xi$ with a suitable $W(E)$-conjugate, we may assume $\theta(\xi) = \xi$ (see \cite{Cl90}, Lemma 4.7; this is where we need our assumption that $G = G_{\rm ad}$).  

Now since $\theta(\xi) = \xi$, we can describe $\Theta_{\Pi\theta}(g)$ using Clozel's twisted Atiyah-Bott formula (\cite{Cl84}, Prop.6).  Following Clozel, we use the canonical intertwining operator 
$I_\theta : \Pi ~ \widetilde{\rightarrow} ~ \Pi^\theta$, which takes $\Phi \in i^G_B(\xi)$ to the function $I_\theta\Phi$ defined by $I_\theta\Phi(g) = \Phi(\theta^{-1}(g))$.  Then Clozel's formula says that for $g \in G(E)$ any $\theta$-regular element we have
\begin{equation} \label{twisted_Atiyah-Bott}
\Theta_{\Pi \theta}(g) = \sum_{x \in X^{g\theta}} \dfrac{(\delta_{B(E)}^{1/2}\xi)(x^{-1}\, g\, \theta(x))}{|{\rm det}(1- {\rm d}(g\theta)_x)|_F}.
\end{equation}
The notation is as in loc.~cit., but let us remark here that $X : = G(E)/B(E)$, and $x$ ranges over the (necessarily finite) set of $g\theta$-fixed points $X^{g\theta}$ in $X$.  Further, $\xi$ is viewed as an unramified character on $T(E)$ which has been inflated to $B(E)$.  Finally, the numerator is a slight abuse of notation whose precise meaning is the following.  If $g\theta(x) = x$ and $y \in G(E)$ represents $x$, then we have $y^{-1}g\theta(y) \in B(E)$.  We write this as $y^{-1}g\theta(y) = tu$, for $t \in T(E)$ and $u \in U(E)$.  Note that $t$ is well-defined up to $\theta$-conjugacy in $T(E)$.  By definition, we have
$$
(\delta_{B(E)}^{1/2}\xi)(x^{-1}g\theta(x)) := (\delta_{B(E)}^{1/2}\xi)(t).
$$

It is clear that if $g$ is $\theta$-compact, then so is $t$: since $\mathcal Ng$ is semi-simple, it is stably-conjugate to $\mathcal Nt$.  But now setting $\xi' = \delta_{B(E)}^{1/2}\xi$, we see that that since $\xi' \circ \theta = \xi'$ and $\mathcal Nt$ is a compact element (contained, without loss of generality, in $T(F)$ \footnote{To see this, it helps to use the concrete norm $N$ in place of $\mathcal N$.  This is legitimate, after passing to a $z$-extension of $G$.}), we have
$$
(\xi'(t))^r = \xi'(\mathcal Nt) = \xi'_u(\mathcal Nt) = (\xi'_u(t))^r.
$$
It follows that $\xi'(t) = \xi'_u(t)$, and our claim is proved.

\section{Clozel's temperedness argument} \label{temperedness_section}

We need the following variant of \cite{Cl90}, Lemma 5.5.  Recall that the inclusion of $Z(\mathcal H_J(G(E)))$ into 
$\mathcal C_J(G(E))$ induces an injective homomorphism $Z(\mathcal H_J(G(E))) \hookrightarrow \mathcal C_J(G(E))/{\rm ker}$.

\begin{lemma} \label{temperedness_lemma}
Let $t_i \,\, (i = 1, \dots, N)$ be distinct elements of $\widehat{A^E}/W(E)$. Consider the linear form
\begin{equation}
\phi \mapsto \sum_i c_i \, \widehat{\phi}(t_i) \hspace{.5in} (c_i \neq 0)
\end{equation}
on $Z(\mathcal H_J(G(E)))$.  Suppose that this linear form is continuous with respect to the topology on $Z(\mathcal H_J(G(E)))$ inherited from the Schwartz topology on $\mathcal C_J(G(E))/{\rm ker}$ (or equivalently, from that on $\mathcal C_J(G(E))$).  Then, for all $i$, we have $t_i \in \widehat{A^E}_u/W(E)$.
\end{lemma}

\begin{proof}
Since $\mathcal C_J(G(E))/{\rm ker}$ is topologically isomorphic to $C^\infty[\widehat{A^E}_u]^{W(E)}$ (Proposition \ref{arthur_prop}), Clozel's proof goes over almost word-for-word to the present situation.  Indeed, we consider the following diagram in place of the one in loc.cit.:
$$
\xymatrix{
Z(\mathcal H_J(G(E))) \ar[r]_{\sim}^{\phi \mapsto \widehat{\phi}} \ar[d] & \mathbb C[\widehat{A^E}]^{W(E)} \ar[d] \\
\mathcal C_J(G(E))/{\rm ker} \ar[r]_{\sim}^{\phi \mapsto \widehat{\phi}} & C^\infty[\widehat{A^E}_u]^{W(E)}.
}
$$ 

We make only the following additional remark.  At the end of Clozel's proof, he makes use of the family  of spherical functions (indexed by dominant coweights $\lambda_0$) whose Satake transforms are the functions $\sum_{\lambda \in W(E)\lambda_0} t_\lambda$.  In our context, we may use instead the Bernstein functions $z^J_{\lambda_0}$, which have the analogous property with respect to the Bernstein isomorphism.
\end{proof}  

\begin{Remark}
In Lemma 5.5 of \cite{Cl90}, Clozel assumes that the linear form in question extends to the Schwartz space.  We do not make the analogous assumption that our linear form extends from $Z(\mathcal H_J(G(E)))$ to $\mathcal C_J(G(E))/{\rm ker}$.  In fact the assumption is unnecessary: only the continuity of the form on its  original domain is used in Clozel's proof.  This is fortunate, since in our application of this lemma (see section \ref{putting_it_all_together_section}) it is quite unclear whether one could extend the linear form we are working with to $\mathcal C_J(G(E))/{\rm ker}$, or even to $\mathcal C_J(G(E))$!  
\end{Remark}

\section{The global trace formula and existence of local data} \label{global_trace_formula_section}

\subsection{The global set-up} \label{global_set-up_subsection}

In this subsection we will summarize a global argument of Clozel
relying on the (twisted) Deligne-Kazhdan trace formula and Kottwitz' 
stabilization of its geometric side, which plays a crucial
role in the proof of Theorem \ref{fl} for strongly regular elliptic elements.

We assume $G$ is adjoint (this is necessary, mainly because it is essential for subsection 
\ref{deformation_subsection}).  Clozel's argument is stated for certain 
groups $G$ with $G_{\rm der} = G_{\rm sc}$.  
Our goal here is to show how it goes over almost unchanged for adjoint groups, as long as the elements in question are always {\em strongly} ($\theta$-) regular, as we can and shall assume.

The first step is to embed the local situation into an appropriate
global one. We may assume $G$ is split over an unramified
extension $K/F$ such that $E \subset K$. We choose a degree $[K:F]$ 
cyclic extension of
(totally complex) number fields ${\un K}/{\un F}$ and a finite
place $v_0$ of ${\un F}$ over $p$ such that ${\un K}_{v_0}$ is a
field and ${\un K}_{v_0}/{\un F}_{v_0} \cong K/F$.  There is then
a degree $r = [E:F]$ cyclic extension ${\un E}/{\un F}$ with ${\un E} \subset {\un
K}$ and ${\un E}_{v_0}/{\un F}_{v_0} \cong E/F$.

Using the Tchebotarev density theorem (e.g. \cite{Ser}, I-2.2), 
we find a place $v_1$ of ${\un F}$, $v_1 \neq v_0$, such that the corresponding prime ideal remains prime in $K$, and  
${\rm Gal}({\un K}_{v_1}/{\un F}_{v_1}) = {\rm Gal}({\un K}/{\un F})$.\footnote{This contrasts with Clozel's assumption that $v_1$ splits completely in ${\un K}$.  Our argument in subsection \ref{kappa_vanishing_subsection} requires that $v_1$ remain inert (we need the equality ${\rm Gal}({\un K}/{\un F}) = {\rm Gal}({\un K}_{v_1}/{\un F}_{v_1})$ there), and this explains why we use different ``stabilizing'' functions than Clozel -- see item (b) in the next subsection.  This all seems necessary because we need to work with adjoint groups rather than groups with simply connected derived group.}     
In addition we fix two more auxiliary finite places $v_2$ and $v_3$ 
of ${\un F}$ where ${\un E}/{\un F}$ splits completely.  We can assume $v_0 \notin \{v_1,v_2,v_3 \}$.

There is a quasi-split group ${\un G}$ over ${\un F}$ with the
property that ${\un G} \times_{{\un F}} {\un F}_{v_0} \cong G$.  We set $\widetilde{{\un G}} = {\rm Res}_{{\un E}/{\un F}} {\un G}_{\un E}$.  Let $\theta$ denote the ${\un F}$-linear automorphism of $\widetilde{\un G}$ coming from $\theta \in {\rm Gal}(E/F) \cong {\rm Gal}({\un E}/{\un F})$.  

The groups ${\un G}$ and $\widetilde{\un G}$ are adjoint, and we may assume they are 
split over ${\un K}$.  
Further, they satisfy the Hasse principle for
$H^1$, e.g. ${\rm ker}^1({\un F}, {\un G}) = 1$ (see \cite{Cl90}, p. 293; 
the Hasse principle is used in the stabilization of the
geometric side of the (twisted) trace formula; see $\S 6$ of
\cite{Cl90}, and subsection \ref{kappa_vanishing_subsection} below).  We have for $i=2,3$ an identification
$$
{\un G}({\un E}_{v_i}) = {\un G}({\un F}_{v_i}) \times \cdots
\times {\un G} ({\un F}_{v_i})
$$
($r$ factors), the Galois group acting by cyclic permutations.  By construction ${\un E}/{\un F}$ also 
splits at every archimedean
place $v$ of ${\un F}$, and so ${\un G}({\un E}_v)$ decomposes similarly.

As in \cite{Cl90}, p. 295, we may assume we have (simultaneous) weak approximation for $\widetilde{\un G}$ at the places $v_0, v_1, v_2, v_3$.  This is used in the proof of Proposition \ref{Clozel}, wherein it is necessary to  choose a strongly $\theta$-regular $\delta \in \widetilde{\un G}({\un F})$ which is $\theta$-elliptic at $v_1,v_2,v_3$, close to a given (strongly $\theta$-regular) element $\delta_0 \in G(E)$ at $v_0$, and close to 1 at $v_2$.

We will write ${\un \phi} = \phi^{v_0} \otimes \phi_{v_0} =
\otimes_v {\phi}_v$ for a pure tensor element of $C^\infty_c({\un
G}(\A_{{\un E}}))$. Similarly we will write ${\un f} = f^{v_0}
\otimes  f_{v_0} = \otimes_v f_v$ for a pure tensor element of
$C^\infty_c({\un G}(\A_{{\un F}}))$. We will write $\phi$ for
$\phi_{v_0}$ and $f$ for $f_{v_0}$.

We will always use $S_3$ to denote a finite set of finite places
of ${\un F}$ such that $v_3 \in S_3$ and $v_1, v_2 \notin S_3$. We
will consider {\em triples} $(\phi^{v_0}, f^{v_0}, S_3)$
satisfying the following conditions.

\medskip
\noindent (a) At any finite place $v \notin S_3 \cup \lbrace
v_1,v_2 \rbrace$, the group ${\un G} \times_{{\un F}} {\un F}_v$
and the extension ${\un E}_v / {\un F}_v$ are unramified.

\smallskip
\noindent (b)  $f_{v_1}$ resp. $\phi_{v_1}$ is (up to a scalar) a {\em (generalized) Kottwitz function} on ${\un G}({\un F}_{v_1})$ resp. $\widetilde{\un G}({\un F}_{v_1})$, in the sense of Labesse \cite{Lab99}, $\S 3.9$.  We shall discuss their properties below.

\smallskip
\noindent (c)  $f'_{v_2}$ is a coefficient of a supercuspidal
representation,  $\phi_{v_2} = f'_{v_2} \otimes \cdots \otimes
f'_{v_2}$, and $f_{v_2} = f'_{v_2} * \cdots * f'_{v_2}$.  
Thus $(\phi_{v_2},f_{v_2})$ are associated (\cite{Lan}, $\S8$), and have 
non-vanishing (twisted) orbital integrals at ($\theta$-)elliptic strongly ($\theta$-)regular elements which are close to the identity.

\smallskip
\noindent (d)  For any $v \in S_3$, $f_v$ (resp. $\phi_v$) is
supported on the set of strongly regular elements (resp. elements
with strongly regular norms), and ($\phi_v$, $f_v$) are associated.  
Moreover, at $v_3$, the function
$f_{v_3}$ (resp. $\phi_{v_3}$) is supported on the set of elliptic
elements (resp. elements with elliptic norms).

\smallskip
\noindent (e)  For any finite place $v \notin S_3 \cup \lbrace
v_0, v_1, v_2 \rbrace$, the function $f_v$ (resp. $\phi_v$) is the
unit element of the spherical Hecke algebra of 
${\un G}({\un F}_v)$ (resp. ${\un G}({\un E}_v)$).  Then $(\phi_v,f_v)$ are associated by \cite{Ko86b}.

\smallskip
\noindent (f)  At every place $v$ of ${\un F}$ where ${\un E}/{\un
F}$ splits, we have $\phi_v = f'_v \otimes \cdots \otimes f'_v$ and $f_v = f'_v * \cdots * f'_v$, for an appropriate function $f'_v$ (so that as in (c), $(\phi_v,f_v)$ are associated).

\smallskip

{\em Note: at all places $v \neq v_0$, the functions $f_v$ and
$\phi_v$ are assumed to be associated.}

\subsection{(Generalized) Kottwitz functions} \label{Kottwitz_fcn_subsection}

Suppose again $G$ is adjoint (more generally we require $G$ to have $F$-anisotropic center).  In \cite{Ko88}, Kottwitz introduced the Euler Poincar\'e functions $f_{EP}$ on $G(F)$, compactly-supported functions whose orbital integrals are non-zero only on elliptic semi-simple elements and are constant on stable conjugacy classes in $G(F)$.  One can introduce $\theta$-twisted analogues $\phi_{EP}$ of these on $G(E)$, and Kottwitz' proof goes over without difficulty (cf. \cite{Lab99}, Prop. 3.9.1) to show that their twisted orbital integrals are non-zero only on $\theta$-elliptic $\theta$-semi-simple elements and are constant on stable $\theta$-conjugacy classes.

Labesse \cite{Lab99} calls both $f_{EP}$ and the twisted variants $\phi_{EP}$ {\em Kottwitz functions} (the $\phi_{EP}$ are called Lefschetz functions in \cite{Cl93}).

One can use Kottwitz functions to construct associated ``stabilizing'' functions at the (inert) place $v_1$.  The following theorem of Labesse (\cite{Lab99}, Prop. 3.9.2) explains how to define the functions $\phi_{v_1}$ and $f_{v_1}$ in \ref{global_set-up_subsection}, item (b).

From now on, write $E/F$ in place of ${\un E}_{v_1}/{\un F}_{v_1}$.  We adapt slightly the notation of loc.~cit., letting $\xi = \sum \chi$ denote the sum of characters of ${\bf H}^0_{ab}(F,G)$ which are trivial on the norms from $E$, and putting $f_\xi(x) = f(x)\xi(x)$.  If $f$ is a Kottwitz function, Labesse calls $f_\xi$ a {\em generalized Kottwitz function}.

\begin{prop} [Labesse]   Let $\phi \in C^\infty_c(G(E))$ and $f \in C^\infty_c(G(F))$ be Kottwitz functions.  Let $h = h^1/h^0$, where $h^0$ denotes the number of characters $\chi$ appearing in $\xi$, and $h^1 := |{\rm ker}[H^1(F,G) \rightarrow H^1(F,\widetilde{G})]|$.  Then $\phi$ and $h f_\xi$ are associated functions.
\end{prop}

\subsection{Continuing the global argument}

The {\em very simple trace formula} of Deligne-Kazhdan (see \cite{DKV} $\S
A.1$) asserts that if ${\un f}$ satisfies conditions (c) and (d)
above, then convolution by ${\un f}$ has image in the cuspidal
spectrum.  A similar assertion holds for ${\un \phi}$.  Moreover,
the geometric side of this very simple (twisted) trace formula involves
only ($\theta$-)elliptic strongly ($\theta$-)regular terms.  Since we have also
imposed condition (b), the geometric side is automatically
``stable'': there are no (twisted) $\kappa$-orbital integrals for
$\kappa \neq 1$.   Indeed, it is not hard to adapt the proof of Lemma 6.5 of \cite{Cl90} to adjoint groups.  A few small differences do appear, which we explain in subsection \ref{kappa_vanishing_subsection}.

Let $r({\un f})$ denote the action of ${\un f}$ on the cuspidal
spectrum $L^2_{cusp}({\un G}({\un F})\backslash
{\un G}(\A_{{\un F}}))$. Similarly, let $R({\un \phi})$ denote the
action of ${\un \phi}$ on $L^2_{cusp}({\un G}({\un E}) 
\backslash {\un G}(\A_{{\un E}}))$. Let $I_\theta$ denote the
intertwining operator on the same space given by the action of $\theta$, that is, $I_\theta(\psi)(x) = \psi(\theta^{-1}(x))$ for $\psi \in L^2_{cusp}({\un G}({\un E}) 
\backslash {\un G}(\A_{{\un E}}))$ and $x \in {\un G}(\A_{{\un E}})$.

The following result of Clozel (\cite{Cl90}, $\S 6.3$) gives a
crucial relation between a family of global character identities
and the property that $f_{v_0}$ and $\phi_{v_0}$ are associated at 
strongly regular norms.  We say that $f_{v_0}$ and $\phi_{v_0}$
are associated at strongly regular (elliptic) norms if the condition in
Definition \ref{associated_defn} holds for every strongly regular (elliptic) semi-simple
norm $\gamma = \mathcal N \delta$ in $G(F)$.

\begin{prop}[Clozel] \label{Clozel}
There is a constant $c \neq 0$ (depending only on $({\un G}, {\un
E}/{\un F})$) with the following property: the functions $f_{v_0}$
and $\phi_{v_0}$ are associated at strongly regular norms if and only if
we have the equality of traces
\begin{equation} \label{1st_spectral_identity}
{\rm trace}\, (R({\phi}^{v_0} \otimes \phi_{v_0}) I_\theta) =
c \,\,{\rm trace} \, r({f}^{v_0} \otimes f_{v_0}),
\end{equation}
for every triple $({\phi}^{v_0}, {f}^{v_0}, S_3)$ satisfying
conditions {\rm (a)-(f)} above.
\end{prop}

\begin{Remark}
For $G = G_{\rm ad}$, the constant $c$ is given by 
$$
c = \dfrac{|Z(\widehat{{\un G}})^\Gamma|}{|{\rm Im} \, Z(\widehat{\widetilde{\un G}})^\Gamma|},
$$
in the notation of subsection \ref{kappa_vanishing_subsection}.  The constant $c = c_1 c(D)$ appearing in \cite{Cl90} has 
$$
c(D) = \dfrac{|\pi_0(\widehat{D}^\Gamma)|}{|{\rm Im} \, \pi_0(\widehat{\widetilde{D}}^\Gamma)|}, 
$$
where $D$ resp. $\widetilde{D}$ denotes the cocenter of ${\un G}$ resp. $\widetilde{\un G}$.
\end{Remark}

The details appear on p.~295 of \cite{Cl90}.  The idea for the ``if'' statement is as follows.  If $\gamma_{0} = \mathcal N \delta_{0}$ is given at the place $v_0$, we approximate $\delta_{0}$ by $\delta_{v_0}$ for an appropriate global element $\delta \in \widetilde{\un G}({\un F})$; we require $\delta_{v_i}$ to be $\theta$-elliptic and strongly $\theta$-regular at $v_i = v_1,v_2,v_3$ and close to 1 at $v_2$.  Thus $\delta$ is $\theta$-elliptic and strongly $\theta$-regular.  Then we may choose the set $S_3$ and the associated functions $f_{S_3}$ and $\phi_{S_3}$ such that the geometric sides (of the very simple trace formulas) corresponding to (\ref{1st_spectral_identity}) -- which thanks to \ref{kappa_vanishing_subsection} below take the stabilized form of equations (6.20) resp. (6.22) of \cite{Cl90} -- involve only the term indexed by $\gamma := \mathcal N \delta$.  The sum of adelic (twisted) orbital integrals remaining (on either side) can be written as a single product over all places of local stable (twisted) orbital integrals, and at all places except possibly $v_0$, these are non-zero (we may arrange) and matching.  The character identities (\ref{1st_spectral_identity}) thus force the matching at $v_0$, namely ${\rm SO}_{\delta_{v_0}\theta}(\phi_{v_0}) = {\rm SO}_{\gamma_{v_0}}(f_{v_0})$.   A continuity argument then forces the desired identity ${\rm SO}_{\delta_{0}\theta}(\phi_{v_0}) = {\rm SO}_{\gamma_{0}}(f_{v_0})$.

\subsection{Vanishing of certain $\kappa$-orbital integrals ($G$ adjoint)} \label{kappa_vanishing_subsection}

As stated above, we need to justify how our choice of functions in (b) at the place $v_1$ ensures that there are no (twisted) $\kappa$-orbital integrals for $\kappa \neq 1$.  We need to alter the argument of  Clozel \cite{Cl90}, Lemma 6.5, so that it works for adjoint groups.  However, we temporarily assume $G$ is only connected and reductive.

 Let us switch (mostly) to Clozel's notation, writing $K$ resp. $E$ resp. $F$ for our ${\un K}$ resp. ${\un E}$ resp. ${\un F}$, and $F_{v_1}$ for ${\un F}_{v_1}$, etc.  Set $\Gamma = {\rm Gal}(\overline{F}/F)$ and $\Gamma_1 = {\rm Gal}(\overline{F}_{v_1}/F_{v_1})$.  Also, we write $G$ resp. $\widetilde{G}$ in place of ${\un G}$ resp. $\widetilde{\un G}$. 

Now assume $T$ is an $F$-subtorus of $G$ of the form $T = G_\gamma$ for a strongly regular element $\gamma \in G(F)$.  We have an inclusion
\begin{equation}
T \hookrightarrow G \hookrightarrow \widetilde{G},
\end{equation}
where the second inclusion is the diagonal $F$-embedding which identifies $G$ with the group $\widetilde{G}^\theta$ of elements in $\widetilde{G}$ invariant under the $F$-automorphism $\theta$ (cyclic permutation of coordinates).  Dually, this yields $\Gamma$-equivariant maps 
\begin{equation} \label{1st_dual}
Z(\widehat{\widetilde{G}}) \twoheadrightarrow Z(\widehat{G}) \hookrightarrow \widehat{T}.
\end{equation}
We can identify $Z(\widehat{G})$ canonically with $Z(\widehat{\widetilde{G}})_\theta$, and then with $Z(\widehat{\widetilde{G}})^\theta$ using the isomorphism
$$
N: Z(\widehat{\widetilde{G}})_\theta ~ \widetilde{\rightarrow} ~ Z(\widehat{\widetilde{G}})^\theta 
$$
given by the norm relative to $\theta(t_1, \dots, t_r) = (t_2, \dots, t_r, t_1)$.  When this is done, 
the first map in (\ref{1st_dual}) is identified with the norm map from $Z(\widehat{\widetilde{G}})$ to itself.

Now it makes sense to define the finite abelian groups
\begin{align} \label{A(T/F)}
A(T/F) &= \pi_0(\widehat{T}^\Gamma)/{\rm Im} \, \pi_0(Z(\widehat{\widetilde{G}})^\Gamma) \\
A(T/F_{v_1}) &= \pi_0(\widehat{T}^{\Gamma_1})/{\rm Im} \, \pi_0(Z(\widehat{\widetilde{G}})^{\Gamma_1}). \notag
\end{align}
(Compare with \cite{Cl90}, Definition 6.3.)  

\smallskip

Suppose that $\delta = (\delta_v)_v \in \widetilde{G}(\A_{F})$ is such that $\mathcal N\delta_v = \gamma$ for all places $v$ of $F$.  Using the Hasse principle for $\widetilde{G}$ as in \cite{Cl90}, $\S 6.2$, we associate to $\delta$ an element
\begin{equation} \label{obs}
{\rm obs}(\delta) \in H^1(\A_F/F, T)
\end{equation}
which vanishes if and only if $\delta$ is $\theta$-conjugate in $\widetilde{G}(\A_F)$ to an element of $\widetilde{G}(F)$.  

We claim that ${\rm obs}(\delta)$ has trivial image under the composition (see \cite{Ko86a})
$$ 
H^1(\A_F/F,T) ~ \widetilde{\rightarrow} ~ \pi_0(\widehat{T}^\Gamma)^D  \rightarrow 
\pi_0(Z(\widehat{\widetilde{G}})^\Gamma)^D
$$ 
(notation as in loc.~cit.).  To check this we may pass to a $z$-extension of $G$, adapted to $E/F$ in the sense of \cite{Ko82}, and reduce to the case $G_{\rm der} = G_{\rm sc}$.  Let $D = G/G_{\rm der}$ and set $\widetilde{D} = {\rm Res}_{E/F}D_E$ resp. $\widetilde{T} = {\rm Res}_{E/F}T_E$.  The following diagram commutes
$$
\xymatrix{
H^1(\A_F/F, T) \ar[r] \ar[d] & \pi_0(\widehat{T}^\Gamma)^D \ar[d] \ar[dr] &  \\
H^1(\A_F/F, \widetilde{D}) \ar[r] & \pi_0(\widehat{\widetilde{D}}^\Gamma)^D & 
\pi_0(Z(\widehat{\widetilde{G}})^\Gamma)^D \ar[l]_{\sim}.}
$$
By construction ${\rm obs}(\delta)$ is the image in $H^1(\A_F/F,T)$ of a $T(\overline{\A}_F) \widetilde{T}(\overline{F})$-valued 1-cocycle which becomes a 1-coboundary in $\widetilde{G}(\overline{\A}_F)$.  Therefore ${\rm obs}(\delta)$ has trivial image in $H^1(\A_F/F, \widetilde{D})$, and then the diagram yields the claim.

Therefore for $\kappa \in A(T/F)$ we may define the pairing $\langle {\rm obs}(\delta), \kappa \rangle \in \mathbb C$.  It is clear that 
\begin{equation*}
{\rm obs}(\delta) = 1 \,\,\,\,\,\,\, \mbox{if and only if} \,\,\,\,\,\,\ 
\langle {\rm obs}(\delta), \kappa \rangle = 1 \,\, \mbox{for all $\kappa \in A(T/F)$}.
\end{equation*}

This means that $A(T/F)$ is the group coming into the stabilization of the twisted trace formula, and elements $\kappa \in A(T/F)$ are used to define the twisted $\kappa$-orbital integrals.  Consider the sum
\begin{equation} \label{kappa_orb_int}
\sum_{\delta} \langle {\rm obs}(\delta), \kappa \rangle 
{\rm TO}_{\delta\theta}({\un \phi}),
\end{equation}
where $\delta$ ranges over elements in $G(\A_E)$ such that $\mathcal N \delta = \gamma$, taken up to $\theta$-conjugacy under $G(\A_E)$.  (See \cite{Cl90}, equation (6.15).)

At this point, we once again assume $G$ is an adjoint group.  Finally we have reached the goal of this subsection, which is to check that (\ref{kappa_orb_int}) vanishes for $\kappa \neq 1$, whenever $\phi_{v_1}$ is a Kottwitz function as in \ref{Kottwitz_fcn_subsection}.  (The analogue for $f_{v_1}$ is similar, but easier.)

Since Kottwitz functions have non-zero twisted orbital integrals only on $\theta$-elliptic elements, and are constant over stable $\theta$-conjugacy classes, the proof works exactly as in \cite{Cl90}, Lemma 6.5, modulo the following ingredient.

\begin{lemma} \label{A_injectivity} Suppose $T$ is a maximal $F$-torus in $G$ which is elliptic at $v_1$, the place fixed in subsection \ref{global_set-up_subsection}.  Then the canonical map
\begin{equation} \label{A_map}
A(T/F) \rightarrow A(T/F_{v_1})
\end{equation}
is injective.
\end{lemma}

\begin{proof}
  Since $T$ is anisotropic over $F_{v_1}$ (resp. $\widetilde{G}$ is semi-simple) the groups $\widehat{T}^\Gamma$ and $\widehat{T}^{\Gamma_1}$ (resp. $Z(\widehat{\widetilde{G}})^\Gamma$ and $Z(\widehat{\widetilde{G}})^{\Gamma_1}$) are finite.  Write $\overline{\Gamma}$ for the group ${\rm Gal}(K/F) = {\rm Gal}(K_{v_1}/F_{v_1})$.  Since $\widetilde{G}$ splits over $K$, we have
$$
Z(\widehat{\widetilde{G}})^{\Gamma} = Z(\widehat{\widetilde{G}})^{\overline{\Gamma}} = 
Z(\widehat{\widetilde{G}})^{\Gamma_1}.
$$
Then (\ref{A_map}) becomes the canonical map
$$
\widehat{T}^{\Gamma}/N(Z(\widehat{\widetilde{G}})^{\overline{\Gamma}}) \rightarrow \widehat{T}^{\Gamma_1}/N(Z(\widehat{\widetilde{G}})^{\overline{\Gamma}}), 
$$
where $N$ is the norm homomorphism (the first arrow in (\ref{1st_dual})).  This is clearly injective.
\end{proof}

We remark that the proof does not require that $G$ be adjoint, but only semi-simple.  The idea of using an injectivity statement like the one in this lemma goes back to Kottwitz \cite{Ko88}.

\subsection{Existence of local data}

Again we assume $G = G_{\rm ad}$ we let $F$ denote our $p$-adic field.  Let $\gamma$ denote a strongly regular element of $G(F)$, and $\phi$ an element of $Z(\mathcal H_J(G(E)))$.   For an element $\delta \in G(E)$, define
$$
\Delta(\gamma,\delta) = \begin{cases} 1, \,\,\,\ \mbox{if $\mathcal N\delta = \gamma$} \\
                                      0, \,\,\,\, \mbox{if $\mathcal N\delta \neq \gamma$}. \end{cases}
$$
Also, define
$$
\Lambda(\gamma,\phi) := \sum_\delta \Delta(\gamma,\delta) {\rm SO}_{\delta \theta}(\phi) - {\rm SO}_\gamma(b\phi),
$$
where the sum ranges over stable $\theta$-conjugacy classes of elements $\delta \in G(E)$.  The fundamental lemma is then  equivalent to
$$
\Lambda(\gamma,\phi) = 0
$$
holding for all strongly regular $\gamma$ (see section \ref{reductions_section}).
 
Following Hales \cite{Ha95}, we define {\em local data} for $G(E),G(F)$ to consist of the data (a),(b), and (c), subject to conditions (1) and (2) below.  Let ${\rm Irr}_I(G)$ (resp. ${\rm Irr}^\theta_I(G(E))$) denote the set of irreducible (resp. irreducible $\theta$-stable) admissible representations of $G(F)$ (resp. $G(E)$) having non-zero Iwahori-fixed vectors.

(a) An indexing set $\mathcal I$ (possibly infinite);

(b) A collection of complex numbers $a_i(\pi)$ for $i \in \mathcal I$ and $\pi \in {\rm Irr}_I(G)$;

(c) A collection of complex numbers $b_i(\Pi)$ for $i \in \mathcal I$ and $\Pi \in {\rm Irr}^\theta_I(G(E))$.

\smallskip

(1) For $i$ fixed, the constants $a_i(\pi)$ and $b_i(\Pi)$ are zero for all but finitely many $\pi$ and $\Pi$.

(2) For every function $\phi \in Z(\mathcal H_J(G(E)))$, the following are equivalent:

\hspace{.5in}  (A)  For all $i$, we have $\sum_{\pi} a_i(\pi) \langle {\rm trace} \, \pi, b\phi  \rangle + \sum_{\Pi} b_i(\Pi) \langle {\rm trace} \, \Pi I_\theta , \phi \rangle = 0 $;

\hspace{.5in}  (B)  For all strongly regular $\gamma$, we have $\Lambda(\gamma, \phi) = 0$.

\bigskip

The existence of local data for $G(E),G(F)$ follows from Proposition \ref{Clozel}.  Indeed, varying the functions $\phi_v, f_v$ at the archimedean places shows that the identity (\ref{1st_spectral_identity}) is equivalent to a family of identities, indexed by the archimedean components $\pi_\infty$ of the representations of ${\un G}(\A_{\un F})$.  Fixing the functions at the places other than $v_0$, for each $\pi_\infty$ we get an identity of the form
$$
\sum_{\pi_{v_0}} a_{\pi_\infty}(\pi_{v_0}) \langle {\rm trace} \, \pi_{v_0}, f_{v_0}  \rangle + \sum_{\Pi_{v_0}}  b_{\pi_\infty}(\Pi_{v_0}) \langle {\rm trace} \, \Pi_{v_0} I_\theta , \phi_{v_0} \rangle = 0.
$$
Assuming $f_{v_0}$ and $\phi_{v_0}$ range only over functions bi-invariant under a fixed Iwahori subgroup, these sums involve only a fixed finite number of terms (the archimedean components and the level at all finite places being fixed -- use Harish-Chandra's finiteness theorem for cusp forms, \cite{BJ}).  In other words, the finiteness conditions in the definition of local data are satisfied.  Finally, Proposition \ref{Clozel} says that the functions $\phi_{v_0} = \phi$ and $f_{v_0} = b\phi$ are associated at all strongly regular norms if and only if these identities hold.  Since we may as well assume in (B) that $\gamma$ is a norm (section \ref{reductions_section}), we have verified the equivalence of (A) and (B) in condition (2) of the definition of local data.

\section{Proof of the theorem in the strongly regular elliptic case} \label{putting_it_all_together_section}

In this section we assume $G = G_{\rm ad}$.  Let $\gamma$ denote a strongly regular element of $G(F)$.  
We use the symbol $\phi$ to denote an element in  $Z(\mathcal H_J(G(E)))$.

Recall that the fundamental lemma for $(\gamma,\phi)$ is equivalent to 
\begin{equation} \label{Lambda_vanishes}
\Lambda(\gamma,\phi) = 0.
\end{equation}
By the existence of local data, it is enough to show that for every $i$, the sum of traces in (A) above vanishes.  Fixing $i$ and dropping it from our notation, we need to show 
\begin{equation} \label{*}
\sum_{\pi} a(\pi) \langle {\rm trace} \, \pi, b\phi \rangle + \sum_{\Pi} b(\Pi) \langle {\rm trace} \, \Pi I_\theta , \phi \rangle = 0,
\end{equation}
when this is viewed as a distribution on $\phi \in Z(\mathcal H_J(G(E)))$.  
Roughly, our plan is to prove the vanishing of the left hand side of (\ref{*}) by writing it in two different ways, which can represent the same quantity only if the left hand side of (\ref{*}) is identically zero as a distribution on $Z(\mathcal H_J(G(E)))$.  Our argument is parallel to that in \cite{Cl90}, end of $\S$6.  At certain key points, we use minor variations on the arguments of Hales in \cite{Ha95}.

\smallskip

Let us first make a preliminary remark on (\ref{*}).  By descent and our induction hypothesis that the fundamental lemma holds for groups with smaller semi-simple rank, we know that (\ref{*}) holds provided that $\Lambda(\gamma,\phi) = 0$ for all regular elliptic elements $\gamma$ (see sections \ref{descent}, \ref{reductions_section}).  In particular, if we had
\begin{enumerate}
\item[(a)]  ${\rm O}_\gamma(b\phi) = 0, \,\,\,\,\, \mbox{for all regular elliptic elements $\gamma$}$, and 
\item[(b)]  ${\rm TO}_{\delta \theta} (\phi)  = 0, \,\,\, \mbox{for all $\theta$-regular $\theta$-elliptic elements $\delta$}$,
\end{enumerate}
then (\ref{*}) would hold.  (Note that if $\gamma$ is (regular) elliptic, then of course any $\delta$ appearing in 
$\Lambda(\gamma,\phi)$ is ($\theta$-regular) $\theta$-elliptic.)  Now by Howe's conjecture and its twisted version (see \cite{Cl85}, Prop.~1, and \cite{Cl90}, Thm.~2.8), there exist finite sequences $\gamma_1, \cdots, \gamma _r$ and $\delta_1,\cdots, \delta_s$ of regular elliptic resp. $\theta$-regular $\theta$-elliptic elements of $G(F)$ resp. $G(E)$ such that (a) and (b) above are equivalent to
\begin{enumerate}
\item[(a')]  ${\rm O}_{\gamma_i}(b\phi) = 0, \,\,\,\,\, \mbox{for all $i = 1, \cdots, r$}$, and 
\item[(b')]  ${\rm TO}_{\delta_j \theta} (\phi)  = 0, \,\,\, \mbox{for all $j = 1, \cdots, s$}$.
\end{enumerate}
Thus, by linear algebra, the distribution in (\ref{*}) can be written as 
\begin{equation} \label{*_as_O+TO}
\sum_{i} a_i \, {\rm O}_{\gamma_i}(b\phi) + \sum_j b_j \, {\rm TO}_{\delta_j\theta} (\phi),
\end{equation}
for certain scalars $a_i, b_j \in \mathbb C$.
\smallskip

Next note that since $\phi$ resp. $b\phi$ acts by a scalar on the $J$-fixed vectors in $\Pi$ resp. $\pi$, we can express the traces in terms of the Fourier transform $\widehat{\phi}$.  In particular, we can write the distribution in (\ref{*}) as a sum
\begin{equation} \label{*_as_FT}
\sum_i c_i ~ \widehat{\phi}(t_i),
\end{equation}
for scalars $c_i \in \mathbb C$ and pairwise distinct parameters $t_i \in \widehat{A^E}/W(E)$, where $i = 1, \dots, N$.

\medskip

The first way to express the left hand side of (\ref{*}) is a consequence of Clozel's temperedness argument (section \ref{temperedness_section}).   Let us consider the distribution $\phi \mapsto \sum_i a_i \, {\rm O}_{\gamma_i}(b\phi)$ on $\phi \in Z(\mathcal H_J(G(E)))$, where $Z(\mathcal H_J(G(E)))$ is given the Schwartz topology it inherits from $\mathcal C_J(G(E))$ (cf. Cor. \ref{b_is_Schwartz_continuous}).  This distribution is Schwartz-continuous, since $b$ is (Cor. \ref{b_is_Schwartz_continuous}) and since orbital integrals at regular semi-simple elements are tempered distributions.   Furthermore, twisted orbital integrals at $\theta$-regular elements are also tempered distributions (\cite{Cl90}, Prop. 5.2).  Thus the distribution 
$$
\phi \mapsto \sum_j b_j \, {\rm TO}_{\delta_j\theta}(\phi)
$$
on $\mathcal H_J(G(E))$ is also Schwartz-continuous.   Thus we conclude that (\ref{*_as_O+TO}) hence (\ref{*_as_FT}) is a Schwartz-continuous linear form on $Z(\mathcal H_J(G(E)))$.  Applying Clozel's temperedness argument (Lemma \ref{temperedness_lemma}) to (\ref{*_as_FT}) yields:
$$
\mbox{in (\ref{*_as_FT}), each $t_i \in \widehat{A^E}_u/W(E)$.}
$$
Therefore the left hand side of (\ref{*}) for $\phi = z_\mu$ is a sum of the form
\begin{equation} \label{*_as_FT_mu}
\sum_{i=1}^N \sum_{\lambda \in W(E)\mu} c_i \, \lambda(t_i),
 \end{equation}
where the $t_i$ are pairwise distinct elements in $\widehat{A^E}_u/W(E)$.
(The scalars $c_i$ here differ from those in (\ref{*_as_FT}) by an overall non-zero scalar.)

\medskip

The second way of expressing the left hand side of (\ref{*}) comes from the inverted compact trace identity of section \ref{compact_trace_identity_section}.  First we claim that the distribution (\ref{*_as_O+TO}) can be expressed as a finite linear combination of {\em compact traces} of $\phi$ and $b\phi$ on certain tempered representations.  This assertion is a consequence of Howe's conjecture, the Kazhdan density theorem, their twisted analogues, and the fact that the ($\theta$-)regular ($\theta$-)elliptic elements $\gamma_i$ (resp. $\delta_j$) are ($\theta$-)compact (see the proof of \cite{Ha93}, Cor.1.3).  Thus there exist parameters $s_i \in \widehat{A}$ and $S_j \in \widehat{A^E}$ along with scalars $a'_i,b'_j \in \mathbb C$ ($1 \leq i \leq m$, $1 \leq j \leq M$) such that (\ref{*}) can be written as
\begin{equation} \label{first_*_from_compact_trace}
\sum_i a'_i \, \langle {\rm trace} \, \pi_{i}, b\phi \rangle_c + \sum_j b'_j \, \langle {\rm trace} \, \Pi_{j}I_\theta, \phi \rangle_{\theta-c},
\end{equation}
where $\pi_i$ resp. $\Pi_j$ is an irreducible constituent of $\pi_{s_i}$ resp. $\Pi_{S_j}$.  

In case $s_i$ resp. $S_j$ is {\em unitary}, then the result of Keys \cite{Keys} implies that $\pi_i$ resp. $\Pi_j$ is a full unramified principal series, i.e., that $\pi_i = \pi_{s_i}$ resp. $\Pi_j = \Pi_{S_j}$.  But then by subsection \ref{deformation_subsection}, its compact trace is the same as that for a principal series induced from a {\em non-unitary} parameter.  Therefore we may assume in the expression above that all the parameters $s_i$ and $S_j$ are {\em non-unitary}.

\smallskip

We now write (\ref{first_*_from_compact_trace}) more explicitly, using Proposition \ref{simplification_prop} (along with its non-twisted antecedent).  We can write (\ref{first_*_from_compact_trace}) for $\phi = z_\mu$ in the form
\begin{equation} \label{second_*_from_compact_trace_mu}
\sum_P \sum_{j_P} \sum_{\lambda \in W(E)\mu}  c'_{j_P} \, \widehat{\chi}_N(\varpi^\lambda) \, \lambda(t'_{j_P}),
\end{equation}
where $P = MN$ ranges over standard $F$-parabolics, and $j_P$ ranges over a finite index set depending on $P$.  Also, the parameters $t'_{j_P} \in \widehat{A^E}/W(E)$  are all {\em non-unitary}. 

Now following \cite{Ha95}, p. 986, there exists a finite collection of hyperplanes $\{H_i\}_{i=1}^l$ through the origin in $X_*(A^E)_{\mathbb R}$ such that $\widehat{\chi}_N(\varpi^{w\lambda}) = \widehat{\chi}_N(\varpi^{w\lambda'})$, for all $N$ and all $w \in W(E)$, as long as $\lambda$ and $\lambda'$ belong to the same component of $X_*(A^E)_{\mathbb R} \backslash (H_1 \cup \cdots \cup H_l)$.  (As in loc.cit., we take the collection consisting of all root hyperplanes and all $W(E)$-conjugates of the walls of all obtuse Weyl chambers ${\rm supp}(\widehat{\tau}^G_P)$.)   Fix one component $C$, on which for each $P$ the function $\lambda \mapsto \widehat{\chi}_N(\varpi^\lambda)$ takes the value 1.  
There is a subset $W'(P) \subset W(E)$, depending on $C$ and $P$, with the property that for $\mu$ any element of $C \cap X_*(A^E)$, we have 
$$
\widehat{\chi}_N(\varpi^{w\mu}) = \begin{cases} 1, \,\,\,\, \mbox{if $w \in W'(P)$} \\
                                               0, \,\,\,\, \mbox{if $w \notin W'(P)$}. \end{cases}
$$
Using this we can rewrite (\ref{second_*_from_compact_trace_mu}) for $\mu \in C$, and we can compare the result with (\ref{*_as_FT_mu}).  We get an equation of the form
\begin{equation} \label{comparison}
\sum_{i=1}^N\sum_{w \in W(E)} c_i \, ^w\mu(t_i) = \sum_{P} \sum_{j_P} \sum_{w' \in W'(P)} c'_{j_P} \, ^{w'}\mu(t'_{j_P}).
\end{equation}
The parameters $t_i$ are unitary, and the parameters $t'_{j_P}$ are non-unitary.  
By an independence of characters argument each side must be identically zero as a function of $\mu \in C \cap X_*(A^E)$.  Since $X_*(A^E)$ is generated as an abelian group by $C \cap X_*(A^E)$, it follows that the left hand side of (\ref{comparison}) vanishes for all $\mu \in X_*(A^E)$.  Consequently, the distribution (\ref{*_as_FT}) hence the left hand side of (\ref{*}) vanishes identically on $Z(\mathcal H_J(G(E)))$.  This completes the proof of Theorem \ref{fl}. \qed

\small
\bigskip
\obeylines
\noindent
University of Maryland
Department of Mathematics
College Park, MD 20742-4015 U.S.A.
email: tjh@math.umd.edu


\end{document}